\documentclass[11.5pt,twoside, a4paper, english, reqno]{amsart} 

\usepackage{amsaddr}
\usepackage{a4wide}
\usepackage{amscd}
\usepackage{amssymb}
\usepackage{amsthm}
\usepackage{graphicx}
\usepackage{amsmath,calrsfs}
\usepackage{latexsym}
\usepackage{enumitem,dsfont}

\usepackage{upref}

\usepackage[usenames,dvipsnames]{xcolor}
\usepackage[
colorlinks=true,
linkcolor=blue,
citecolor=blue,
urlcolor=blue
]{hyperref}

\usepackage{color,graphics}
\usepackage{pdfsync}
\usepackage{ulem,cancel,pgf}

\usepackage{frcursive}

\theoremstyle{plain}
\newtheorem{theorem}{Theorem}[section]
\newtheorem{lemma}[theorem]{Lemma}
\newtheorem{proposition}[theorem]{Proposition}

\theoremstyle{definition}
\newtheorem{definition}{Definition}[section]

\newtheorem{remark}{Remark}[section]

\newtheorem{question}{Question}[section]

\newtheorem*{maintheorem*}{Main Theorem}
\newtheorem*{maincorollary*}{Main Corollary}

\DeclareFontFamily{U}{BOONDOX-calo}{\skewchar\font=50 }
\DeclareFontShape{U}{BOONDOX-calo}{m}{n}{<-> s*[1.05] BOONDOX-r-calo}{}
\DeclareFontShape{U}{BOONDOX-calo}{b}{n}{<-> s*[1.05] BOONDOX-b-calo}{}
\DeclareMathAlphabet{\mathcalb}{U}{BOONDOX-calo}{m}{n}
\SetMathAlphabet{\mathcalb}{bold}{U}{BOONDOX-calo}{b}{n}
\DeclareMathAlphabet{\mathbcalb}{U}{BOONDOX-calo}{b}{n}

\numberwithin{equation}{section}
\allowdisplaybreaks

\title[A new class of nonlocal operators: the fractional logarithmic $p$-Laplacian and its properties]{A new class of nonlocal operators: the fractional logarithmic $p$-Laplacian and its properties}

\date{\today}

\author[Anouar Bahrouni]{Anouar Bahrouni}
\address[Anouar Bahrouni]{Mathematics Department, Faculty of Sciences, University of Monastir, 5019 Monastir, Tunisia.}
\email[Anouar Bahrouni]{anouar.bahrouni@fsm.rnu.tn}

\author[Abdelhamid Gouasmia]{Abdelhamid Gouasmia}
\address[Abdelhamid Gouasmia]{
	Department of Mathematics, Faculty of Sciences And Technology, \\
	Mohamed Cherif Messaadia University, \\
	P.O. Box 1553, Souk Ahras 41000, Algeria. \\
	Laboratoire d'equations aux d\'{e}riv\'{e}es partielles non lin\'{e}aires et histoire des math\'{e}matiques, \\
	Ecole Normale Sup\'{e}rieure, 
	B.P. 92, Vieux Kouba, 16050 Algiers, Algeria.}
\email[Abdelhamid Gouasmia]{gouasmia.abdelhamid@gmail.com}

\author[Hichem Hajaiej]{Hichem Hajaiej}
\address[Hichem Hajaiej]{Department of Mathematics, California State University, Los Angeles, CA 90032, USA.}
\email[Hichem Hajaiej]{hichem.hajaiej@gmail.com}

\author[Anass Ouannasser]{Anass Ouannasser}
\address[Anass Ouannasser]{Faculty of Sciences, Mohammed V University, 4 Avenue Ibn Batouta, P.O. Box 1014, 10000 Rabat, Morocco.}
\email[Anass Ouannasser]{anass.ouannasser@um5r.ac.ma}

\usepackage{comment}
\begin{document}
	
\begin{abstract}
     In this paper, we introduce and investigate the fractional logarithmic $p$-Laplacian $(-\Delta)_{p}^{s+\log}$, defined as the first-order derivative with respect to the parameter $t$ of the fractional $p$-Laplacian $(-\Delta)_{p}^{t}$ evaluated at $t=s$. We establish that this operator admits the following integral representation
\[
\begin{aligned} (-\Delta)_{p}^{s+\log} u(x) &= B(N,s,p)(-\Delta)_{p}^{s}u(x)\\ &\quad -p\,  C(N,s,p)\mathrm{P.V.}\int_{\mathbb{R}^{N}}\frac{|u(x)-u(y)|^{p-2}(u(x)-u(y))\ln |x-y|}{|x-y|^{N+sp}}dy,
\end{aligned}
\]
where $C(N,s,p)$ denotes the standard normalization constant associated with the fractional $p$-Laplacian, and $B(N,s,p)=\frac{d}{ds}\left(\ln C(N,s,p)\right)$. As a consequence of this representation, it follows that the operator is nonlocal and of logarithmic type, and may be viewed as a nonlinear analogue of the fractional logarithmic Laplace operator recently introduced by Chen et al. \cite{Chen-Chen-Hauer}. We further develop the associated functional framework in both $\mathbb{R}^{N}$ and bounded Lipschitz domains by introducing the natural energy spaces adapted to problems driven by $(-\Delta)_{p}^{s+\log}$. Within this framework, fundamental functional inequalities are established, in particular Pohozaev-type identities and D\'{\i}az-Saa inequalities, which are of independent interest and applicable to a broader class of problems. Moreover, we derive results concerning density, continuity, and compact embedding properties. We emphasize that the compactness of the embedding is proved at the critical exponent $p^{*}_{s}=\frac{Np}{N-sp}$, which distinguishes the present setting from the classical Sobolev and fractional Sobolev frameworks. Finally, as an application, we investigate the associated Dirichlet eigenvalue problem and derive existence, uniqueness, and boundedness results for the corresponding solutions.
\end{abstract}

\maketitle
\\[2mm]

\newpage
\setcounter{tocdepth}{1}
\tableofcontents

	\section{Introduction}
\label{sec 1}
The fractional $p$-Laplacian $(-\Delta)^s_p$ arises naturally as the first variation of the fractional Dirichlet energy, also referred to as the Gagliardo seminorm, defined for $N \geq 1$, $0 < s < 1$, and $1 < p < \infty$ by
\begin{equation*}
\Phi_{s,p}(u) = \iint_{\mathbb{R}^{2N}} \frac{|u(x) - u(y)|^p}{|x - y|^{N+sp}} dx dy.
\end{equation*}
More precisely, for suitably regular functions $u$, the fractional $p$-Laplace operator is defined by
\begin{equation}
\label{equ0}
\begin{gathered}
\begin{aligned}
(-\Delta)^s_p u(x) &:= C(N, s, p) \mathrm{P.V.} \int_{\mathbb{R}^N} \frac{\left| u(x) - u(y)\right|  ^{p-2}(u(x) - u(y))}{|x - y|^{N + sp}} dy,\\
& = C(N, s, p) \lim_{\epsilon \to 0} \int_{\{y \in \mathbb{R}^N : |y - x| \geq \epsilon\}} \frac{\left| u(x) - u(y)\right|  ^{p-2}(u(x) - u(y))}{|x - y|^{N + sp}} dy,
\end{aligned}
\end{gathered}
\end{equation}
where $C(N, s, p)$ denotes a normalization constant. It is noteworthy that this operator, along with nonlocal integro-differential operators in general, appears in a wide range of applications, including optimization, finance, phase transitions, thin soft films, and image processing. Furthermore, the fractional $p$-Laplacian constitutes a fundamental model for certain jump L\'evy processes in probability theory, for porous media in physics, and for various other phenomena. For a comprehensive account of the principal properties of such operators, we refer the reader to Bertoin \cite{Bertoin}, Caffarelli and V\'azquez \cite{Caffarelli-Vazquez}, De Pablo et al. \cite{DePablo-Quiros-Rodriguez-Vazquez}, Del Teso et al. \cite{DelTeso-GomezCastro-Vazquez}, Di Nezza et al. \cite{DiNezza-Palatucci-Valdinoci}, Foghem \cite{Foghem}, Majda and Tabak \cite{Majda-Tabak}, V\'azquez \cite{Vazquez} and the references therein. Recently, the fractional $p$-Laplacian has attracted significant attention, primarily due to the intricate interplay between its non-linear and nonlocal features. Considerable attention has been devoted to the study of existence, uniqueness, regularity, and qualitative properties of solutions for a broad class of problems, including elliptic and parabolic boundary-value problems, as well as eigenvalue problems. This research encompasses several fundamental inequalities, such as those of Picone, Hardy, Sobolev, and Poincar\'e, which play a central role in the analysis of these problems. For comprehensive surveys and further references beyond those cited herein, we refer the reader to Bisci et al. \cite{Bisci-Radulescu-Servadei}, Brasco and Franzina \cite{Brasco-Franzina}, Del Pezzo and Quaas \cite{DelPezzo-Quaas}, Frank et al. \cite{Frank-Konig-Tang}, Iannizzotto et al. \cite{Iannizzotto-Liu-Perera-Squassina}, without giving an exhaustive list.  Concerning the normalization constant $C(N, s, p)$, in the particular case $p = 2$, it is given by
\begin{equation}\label{equ1}
C(N, s) = C(N, s, 2) := \frac{s 2^{2s} \Gamma \left(\frac{N + 2s}{2}\right)}{\pi^{\frac{N}{2}} \Gamma(1 - s)},
\end{equation}
where $\Gamma$ denotes the Gamma function. This specific choice ensures that the operator $(-\Delta)^s_2 = (-\Delta)^s$ has the Fourier symbol $|\xi|^{2s}$ for $ \xi \in \mathbb{R}^{N}$ (see Di Nezza et al. \cite[Section 3]{DiNezza-Palatucci-Valdinoci}). For $p \neq 2$, the normalization constant $C(N,s,p)$ is often chosen arbitrarily in much of the literature and is frequently set equal to $2$ for the sake of simplifying computations. In the present work, however, it is essential to employ explicit expressions. Accordingly, we define the normalization constant in such a way that the following limits are satisfied
\begin{equation}\label{equ2}
\lim_{s \to 0^+} (-\Delta)^s_p u(x) = |u(x)|^{p-2} u(x), 
\quad 
\lim_{s \to 1^-} (-\Delta)^s_p u(x) = -\Delta_p u(x), 
\quad x \in \mathbb{R}^N.
\end{equation}
Hence, we define the normalization constant $C(N,s,p)$ as follows (see Dyda et al. \cite[Remark 2.5]{Dyda-Jarohs-Sk})
\begin{equation}
\label{equ4}
C(N,s,p) =
\begin{cases}
\dfrac{s p 2^{2(s-1)} \Gamma \left(\frac{N+sp}{2}\right)}{\pi^{\frac{N-1}{2}}  \Gamma \left(1 - s\right) \Gamma \left(\frac{p+1}{2}\right)}, & \text{if } s > \dfrac{1}{2}, \\[10pt]
\dfrac{s p  2^{2s-1} \Gamma \left(\frac{N+sp}{2}\right)}{\pi^{\frac{N}{2}} \Gamma \left(1 - s\right)}, & \text{if } s \leq \dfrac{1}{2}.
\end{cases}
\end{equation}
This choice of $C(N,s,p)$ preserves several essential properties, including the validity of the limits in \eqref{equ2}, and is consistent with the classical case $p=2$, as in \eqref{equ1}. For further details and additional properties, we refer the reader to Foghem \cite[Section 9.4]{Foghem} (see also Del Teso et al. \cite{DelTeso-GomezCastro-Vazquez}).

On the other hand, for $p = 2$, based on the first limit in \eqref{equ2}, Chen and Weth \cite{Chen-Weth} observed that the function $z \mapsto |z|^{-N}$ is expected to appear in a suitably renormalized limit of the integral kernels $z \mapsto |z|^{-N - sp}$ appearing in \eqref{equ0}, for $s$ sufficiently close to $0$. This observation led the authors to introduce a novel operator, referred to as the logarithmic Laplace operator $L_{\Delta}$, which constitutes the first-order correction in the asymptotic expansion as $s \to 0^{+}$.  More precisely, for every $u \in C^{2}_{c}(\mathbb{R}^{N})$ and $x \in \mathbb{R}^{N}$, one has
\begin{equation}
\label{equ3}
(-\Delta)^{s} u(x) = u(x) + s L_{\Delta} u(x) + O(s) \quad \text{as } s \to 0^{+} \text{ in } L^{q}(\mathbb{R}^{N}), \quad 1 < q \leq \infty.
\end{equation}
Furthermore, it was established in the same paper \cite[Theorem 1.1]{Chen-Weth} that the logarithmic Laplace  operator $L_{\Delta}$ admits the following equivalent representation
\begin{equation*}
\begin{aligned}
L_{\Delta} u(x) &= \left.\frac{d}{ds}(-\Delta)^{s} u(x)\right|_{s=0} \\
&= C(N) \text{P.V.} \int_{\mathcal{B}_{1}(x)} \frac{u(x) - u(y)}{|x-y|^{N}} dy - C(N) \int_{\mathbb{R}^{N} \setminus \mathcal{B}_{1}(x)} \frac{u(y)}{|x-y|^{N}} dy + \rho(N) u(x),
\end{aligned}
\end{equation*}
where $\mathcal{B}_{1}(x) \subset \mathbb{R}^{N}$ denotes the open Euclidean ball of radius $1$ centered at $x \in \mathbb{R}^{N}$, and
\begin{equation*}
C(N) := \pi^{-N/2} \Gamma \left(N/2\right), \qquad \rho(N) := 2 \ln 2 + \psi \left(N/2\right) - \gamma,
\end{equation*}
where $\gamma := -\Gamma'(1)$ denotes the Euler-Mascheroni constant, and $\psi := \Gamma'/\Gamma$ is the digamma function. In fact, in the same work, the authors developed a comprehensive variational framework for Dirichlet problems posed on bounded domains, thereby providing a rigorous functional setting for the logarithmic Laplace operator $L_{\Delta}$. Within this framework, they investigated the associated eigenvalue problem and derived several fundamental qualitative properties. In particular, they established a Faber-Krahn type inequality, proved appropriate weak and strong maximum principles, and obtained continuity results for weak solutions to the corresponding Poisson problems up to the boundary. For the nonlinear counterpart, Dyda et al. \cite{Dyda-Jarohs-Sk} extend the operator $L_{\Delta}$ to the logarithmic $p$-Laplacian $L_{\Delta_{p}}$, defined by
\begin{equation}
\label{equ16}
\begin{aligned}
&L_{\Delta_{p}} u(x) = \left.\frac{d}{ds}(-\Delta)_{p}^{s} u(x)\right|_{s=0}  = C(N, p) \text{P.V.}\int_{\mathcal{B}_{1}(x)} \frac{|u(x) - u(y)|^{p-2} (u(x) - u(y))}{|x-y|^{N}} dy \\
&+ C(N, p)  \int_{\mathbb{R}^{N}\setminus \mathcal{B}_{1}(x)} \frac{|u(x) - u(y)|^{p-2} (u(x) - u(y)) - |u(x)|^{p-2} u(x)}{|x-y|^{N}} dy   + \rho(N, p) \left| u(x)\right| ^{p-2} u(x),
\end{aligned}
\end{equation}
where
\begin{equation*}
C(N, p) := \frac{p \Gamma \left(N/2\right)}{2 \pi^{N/2}}, \qquad \rho(N, p) := 2 \ln 2 + \frac{p}{2} \psi \left(N/2\right) - \gamma.
\end{equation*}
In analogy with the case $p=2$, the authors developed a rigorous variational framework to investigate Dirichlet problems involving $L_{\Delta_{p}}$ on bounded domains. Within this framework, they elucidated the relationship between the first Dirichlet eigenvalue and corresponding eigenfunction of the fractional $p$-Laplacian and those of the logarithmic $p$-Laplacian. Furthermore, they extended the Faber-Krahn inequality, established comprehensive maximum principles, and derived a boundary Hardy-type inequality in the associated energy space.

The publication of these two pioneering works has sparked significant interest in boundary value problems involving the logarithmic Laplacian, although this area remains at an early stage of development. This growing interest is driven not only by the wide range of applications (see, for instance, Pellacci and Verzini \cite{Pellacci-Verzini}, Antil and Bartels \cite{Antil-Bartels}, as well as Feulefack and Jarohs \cite{Feulefack-Jarohs}, where fractional models with very small order $s$ in the operator \eqref{equ0} arise as an optimal choice in several contexts) but also by significant advances in the analysis of nonlocal phenomena in partial differential equations.  For representative contributions, we refer to Arora et al. \cite{Arora-Giacomoni-Vaishnavi, Arora-Mukherjee} and Chen et al. \cite{Chen-Hauer-Weth} in the case $p = 2$, and to Arora et al. \cite{Arora-Giacomoni-Hajaiej-Vaishnavi, Arora-Hajaiej-Perera} for the more general setting $p > 1$. In this direction, and motivated by the preceding discussion, a natural question arises:
\begin{question}\label{question}
	For $1 < p < \infty$, is it possible to compute the first-order derivative
	\begin{equation}\label{equ33}
	\left. \frac{d}{d t}(-\Delta)_{p}^{t} u(x) \right|_{t=s},
	\end{equation}
	for an arbitrary order $0 < s < 1$ and for a suitable function $u$?
\end{question}
In this context, Chen et al. \cite{Chen-Chen-Hauer} recently established the existence of the first-order derivative and introduced the logarithmic Laplacian, denoted by $(-\Delta)^{s+\log}$ in the linear case (i.e., when $p=2$). This operator arises as the first-order correction term in the expansion \eqref{equ3} for any $0 < s < 1$ and for suitable functions $u$. Precisely, for $u \in C^{2}_{c}(\mathbb{R}^{N})$, the logarithmic Laplacian $(-\Delta)^{s+\log}$ is defined by
\begin{equation*}
\begin{aligned}
(-\Delta)^{s+\log} u(x) &= \left. \frac{d}{d t}(-\Delta)^{t} u(x) \right|_{t=s} \\
&=  B(N,s) (-\Delta)^{s} u(x) +C(N,s) \mathrm{P.V.} \int_{\mathbb{R}^{N}} \frac{u(x) - u(y)}{|x-y|^{N+2s}} \big(-2 \ln|x-y|\big) dy,
\end{aligned}
\end{equation*}
where $C(N,s)$ is given by \eqref{equ1}, and
\begin{equation*}
B(N,s) :=    \frac{d}{dt} \big( \ln C(N,t) \big)\Big|_{t=s}  = 2 \ln 2 + \frac{1}{s} + \psi(1-s) + \psi\left(\frac{N+2s}{2}\right).
\end{equation*}
The authors present several equivalent formulations of the operator $(-\Delta)^{s+\log}$, including its characterization as a Fourier multiplier, its definition via spectral calculus, and a description based on an extension procedure. Furthermore, a comprehensive variational framework is developed to address problems both in $\mathbb{R}^{N}$ and on bounded smooth domains. This framework introduces the natural energy spaces associated with the operator and establishes the corresponding embedding results. In particular, a compact embedding at the critical exponent $ 2^{*}_{s} = \frac{2N}{N - 2s} $ is obtained, a phenomenon that contrasts with the classical Sobolev and fractional Sobolev settings. Within this framework, existence and boundedness results for Poisson-type problems are rigorously derived, along with several related results. The Dirichlet eigenvalue problem is also investigated, and qualitative spectral properties are established.  Motivated by the aforementioned studies, this article aims to extend the operator introduced above to the case $p>1$, specifically to the fractional logarithmic $p$-Laplacian, denoted by $(-\Delta)_{p}^{s+\log}$, and to establish the associated functional framework.
	
	A principal difficulty in the present analysis arises from the differentiation of the nonlocal operator under the integral sign. Introducing the notation $g_p(t) := |t|^{p-2}t$, a formal differentiation of the fractional kernel leads to the following pointwise identity:
	\begin{equation*}
		\partial_t \left( \frac{g_p(u(x)-u(y))}{|x-y|^{N+tp}} \right)\Bigg|_{t=s}
		= \frac{g_p(u(x)-u(y))}{|x-y|^{N+sp}} \left( -p \ln|x-y|\right).
	\end{equation*}
	This logarithmic factor significantly enhances the singularity in a neighborhood of the origin. In the singular range $1 < p < 2$, the function $g_p$ is only $(p-1)$-H\"older continuous, that is, $ |g_p(a)-g_p(b)| \leq C |a-b|^{p-1}. $ As a consequence, the resulting singular behavior is asymptotically of the form $ |x-y|^{(p-1)-N-sp} |\ln|x-y||, $ which necessitates stronger regularity assumptions, such as $u \in C_c^{1,\alpha}(\mathbb{R}^N)$ with $\alpha > \frac{1-p(1-s)}{p-1}$, in order to ensure local integrability for $ s \in \left(0, 2(p-1)/p\right) $. Moreover, the logarithmic kernel changes sign at $|x-y| = 1$. The associated energy functional admits the following decomposition into positive and negative contributions
	\begin{equation*}
		\mathcal{J}_{s+\log, p}(u,u)
		= \mathcal{J}_+(u,u) - \mathcal{J}_-(u,u) + \mathcal{J}_{s}(u,u).
	\end{equation*}
	The negative part, up to equivalence, is characterized by
	\begin{equation*}
		\mathcal{J}_-(u,u)
		\sim \iint_{|x-y|>1}
		\frac{|u(x)-u(y)|^p}{|x-y|^{N+sp}}  \ln|x-y|  dx  dy
		> 0.
	\end{equation*}
	Analogously, the positive part is characterized by
	\begin{equation*}
		\mathcal{J}_+(u,u)
		\sim \iint_{|x-y|<1}
		\frac{|u(x)-u(y)|^p}{|x-y|^{N+sp}}  (-\ln|x-y|)  dx  dy
		> 0.
	\end{equation*}
	Furthermore, the standard fractional term is given by
	\begin{equation*}
		\mathcal{J}_{s}(u,u)
		\sim \iint_{\mathbb{R}^{2N}}
		\frac{|u(x)-u(y)|^{p}}{|x-y|^{N+sp}}  dx  dy.
	\end{equation*}
This shows that the functional $\mathcal{J}_{s+\log, p}(u,u)$ is not globally nonnegative on $\mathbb{R}^{2N}$, which in turn invalidates the usual coercivity arguments and makes the application of the direct methods in the calculus of variations more delicate. To overcome these difficulties, we develop a robust functional analytic framework based entirely on real-variable methods. In the case $p=2$, Chen et al.~\cite[Theorem 1.1]{Chen-Chen-Hauer} relied crucially on the Plancherel theorem and the Fourier symbol representation
\begin{equation*}
	\mathcal{F}\big((-\Delta)^{s+\log}u\big)(\xi)
	= |\xi|^{2s}\ln(|\xi|^2)\hat{u}(\xi).
\end{equation*}
However, in the case $p \neq 2$, the nonlinear structure of the problem precludes the use of Fourier transform techniques. The method developed in this work extends the corresponding linear theory to the quasilinear setting $1 < p < \infty$. Among the main contributions, we establish the following sharp interpolation estimate (see Proposition \ref{proposition1} (3)): for every radius $r \in (0,1)$,
\begin{equation*}
[u]_{s,p}^p \leq \frac{1}{C(N,s,p)(-\ln r)} \|u\|^{p}_{W_{0}^{s+\log,p}(\Omega)} 
+ \frac{2^{p} \omega_N}{sp} r^{-sp} \|u\|^{p}_{L^{p}(\Omega)},
\end{equation*}
where $C(N,s,p)$ denotes the standard fractional normalization constant defined in \eqref{equ4}, and $\omega_N$ denotes the surface measure of the unit sphere in $\mathbb{R}^N$, given by $ \omega_N = \frac{2\pi^{N/2}}{\Gamma \left(N/2\right)}$. Moreover, $W_0^{s+\log,p}(\Omega)$ stands for the fractional logarithmic Sobolev space introduced in Section~\ref{section2}. This logarithmic control yields a strong compactness property at the critical fractional Sobolev exponent 
$p_s^* := \frac{Np}{N-sp}$. In particular, we establish the compact embedding (see Theorem \ref{EmbeddingsTheorem})
\begin{equation*}
W_0^{s+\log, p}(\Omega) \hookrightarrow \hookrightarrow L^{p_s^*}(\Omega),
\end{equation*}
This framework constitutes a substantial extension beyond the classical fractional Sobolev setting.  Indeed, it rules out the formation of Dirac mass concentrations and makes the usual concentration-compactness limit measures unnecessary on bounded domains. Furthermore, we derive a structural Pohozaev-type identity (see Theorem \ref{Pohozaev}). The introduction of the logarithmic weight breaks the exact homogeneity  $ \mathbf{K}_{s,p}(\lambda z)=\lambda^{-(N+sp)}\mathbf{K}_{s,p}(z) $ which is satisfied by the classical fractional kernel. In this modified framework, the corresponding commutator identity is given by
\[
z \cdot \nabla \mathbf{K}_p^{s+\log}(|z|) + (N+sp)\mathbf{K}_p^{s+\log}(|z|)
= -p C(N,s,p) |z|^{-N-sp},
\]
where $\mathbf{K}_p^{s+\log}$ is defined in \eqref{equ14} (see Section~\ref{section2}). This formulation induces a strictly positive scale-breaking defect measure defined by
\begin{equation*}
\Gamma_{s,p}(u)
:= \frac{C(N,s,p)}{2}
\iint_{|x-y|<1}
\frac{|u(x)-u(y)|^p}{|x-y|^{N+sp}} dx dy
> 0.
\end{equation*}
The introduction of this defect measure leads to a modified Pohozaev identity of the form
\[
\frac{N-sp}{2p} \iint_{\mathbb{R}^{2N}} |u(x)-u(y)|^p \mathbf{K}_p^{s+\log}(|x-y|) \, dx \, dy - \frac{N}{q}\|u\|_{L^q(\Omega)}^q
= \Gamma_{s,p}(u) - \mathcal{B}_{s+\log, p}(u),
\]
which shows that the classical non-existence arguments for critical solutions in star-shaped domains are no longer applicable in the present framework. Under the geometric assumption $\mathrm{diam}(\Omega) < e^{-1/sp}$, the associated kernel remains strictly positive on $\Omega \times \Omega$. This property allows for a global integration of the identity and ensures both the positivity and simplicity of the first Dirichlet eigenvalue $\lambda_{1,p}^{s+\log}$ (see Theorem \ref{theorem6}). We consider the associated nonlinear Dirichlet eigenvalue problem
\begin{equation*}
(-\Delta)_p^{s+\log} u = \lambda |u|^{p-2}u \quad \text{in } \Omega, \qquad 
u = 0 \quad \text{in } \mathbb{R}^N \setminus \Omega.
\end{equation*}
The first eigenvalue is defined variationally by
\[
\lambda_{1,p}^{s+\log} := \inf \left\{ \mathcal{J}_{s+\log,p}(u,u) : \|u\|_{L^p(\Omega)} = 1 \right\}.
\]
Due to the sign-changing nature of the kernel $\mathbf{K}_p^{s+\log}(|z|)$ for $|z| > e^{-1/sp}$, the classical comparison principle
$\mathcal{J}_{s+\log,p}(|u|,|u|) \leq \mathcal{J}_{s+\log,p}(u,u)$ and the standard Picone identity are no longer valid. This difficulty is overcome by imposing the geometric assumption $\mathrm{diam}(\Omega) < e^{-1/sp}$, which guarantees that $\mathbf{K}_p^{s+\log}(|x-y|) > 0$ holds strictly on $\Omega \times \Omega$. On the interaction set $\Omega \times \Omega^c$, the negative contribution of the kernel is compensated exactly by the boundary condition, in the sense that
\[
|u(x)-0|^p - \big||u(x)|-0\big|^p = 0.
\]
Together with a fractional Picone-type inequality, this cancellation mechanism allows us to establish that $\lambda_{1,p}^{s+\log} > 0$ is simple, and that the corresponding eigenfunction $\phi_{1,p}^{s+\log}$ is strictly positive and globally bounded, satisfying $ \|\phi_{1,p}^{s+\log}\|_{L^\infty(\Omega)} \leq C. $ To complete the analysis of these eigenfunctions, we establish a strong minimum principle for weak super-solutions satisfying $(-\Delta)_p^{s+\log} u \geq 0$ (see Lemma \ref{lemma1} and Theorem \ref{theorem3}). This result extends the weak Harnack estimates of Di Castro et al. \cite[Lemma 1.3]{DiCastro-Kuusi-Palatucci} by employing the test function $\eta = (u+\delta)^{1-p}\phi^p$ in the weak formulation. The main argument relies on a precise control of the logarithmic term through an explicit computation of the associated singular integrals. In particular, we show that the logarithmic singularity exactly recovers the expected fractional scaling
\[
r^{-p} \cdot r^{p(1-s)} \cdot r^N = r^{N-sp}.
\]
More specifically, for $u \in W_{\mathrm{loc}}^{s+\log,p}(\Omega)$, we establish the following quantitative estimate:
\begin{equation*}
\begin{aligned}
&\int_{\mathcal{B}_{2r}(x_{0})}
\int_{\mathcal{B}_{2r}(x_{0})}
\mathbf{K}^{s+\log}_{p}(|x-y|)
\left|
\ln\left(
\frac{\delta+u(x)}{\delta+u(y)}
\right)
\right|^{p}
\,dx\,dy
\\[4pt]
&\quad\leq
C r^{N-sp}
\Bigg\{
\delta^{1-p}r^{sp}
\Bigg[
\int_{\mathbb{R}^{N}\setminus\mathcal{B}_{R}(x_{0})}
\left(
\mathbf{K}^{s+\log}_{p}
\left(
\frac{|y-x_{0}|}{2}
\right)
\right)_{+}
\big(u(y)\big)_{-}^{p-1}\,dy
\\
&\qquad\qquad\qquad\qquad
+
\int_{\mathbb{R}^{N}\setminus\mathcal{B}_{R}(x_{0})}
\left(
\mathbf{K}^{s+\log}_{p}
\left(
2|y-x_{0}|
\right)
\right)_{-}
\big(u(y)\big)_{+}^{p-1}\,dy
\Bigg]
+(1-\ln r)
\Bigg\},
\end{aligned}
\end{equation*}
provided that
\[
\operatorname{diam}(\Omega)<e^{-\frac{1}{sp}},
\qquad
B(N,s,p)\geq 0,
\]
together with additional conditions ensuring the finiteness of the two integrals on the right-hand side, where $B(N,s,p)$ is defined in \eqref{equB} below, and
$\mathcal{B}_{2r}(x_{0})\Subset\Omega$ (see Lemma \ref{lemma1}). As a consequence, every non-negative, non-trivial weak supersolution satisfies
\[
\operatorname*{ess\,inf}_{K}u>0
\qquad\text{for every compact set }K\Subset\Omega.
\]
This property, in turn, yields the strong minimum principle.

The paper is organized as follows. In Section~\ref{section2}, we establish the existence of the derivative of $(-\Delta)_{p}^{t}$ with respect to its order at an arbitrary point $s \in (0,1)$. We then introduce the fractional logarithmic $p$-Laplacian operator, investigate its main properties, and develop the associated functional framework. In addition, we study the corresponding energy forms and their fundamental properties. In Section~\ref{section3}, we derive several key inequalities, including a regional Poincar\'e-type inequality, a boundary Hardy-type inequality, a D\'{\i}az-Saa inequality, and a Gagliardo-Nirenberg inequality. We also provide an elementary proof of a Sobolev-type inequality, together with a Pohozaev identity associated with the fractional logarithmic $p$-Laplacian. Section~\ref{section4} is devoted to proving continuous and compact embedding results. We also establish suitable density results within the corresponding functional spaces. Finally, in Section~\ref{section5}, we study the Dirichlet eigenvalue problem, with particular emphasis on the qualitative properties of the associated eigenvalues.
 
\section{The fractional logarithmic $p$-Laplacian and its associated functional setting}\label{section2}

In this section, we introduce the fractional logarithmic Sobolev space $W^{s+\log, p}(\mathbb{R}^{N})$, which serves as the natural functional setting for the weak formulation of the problem under study. Since the structure of this space is intrinsically linked to the fractional logarithmic $p$-Laplacian, we first establish the existence of the derivative appearing in \eqref{equ33} in the more general case $p>1$. We also recall the precise definition of the associated operator and characterize its kernel.  We now state the following proposition. Throughout, we set
	\[
	T_{p} :=
	\begin{cases}
	1, & \text{if } p \geq 2,\\[4pt]
	\dfrac{2(p-1)}{p}, & \text{if } 1 < p < 2.
	\end{cases}
	\]

\begin{proposition}
	\label{prop1}
Let $1 < p < \infty$, and $u \in C^{2}_{c}(\mathbb{R}^{N})$. Then, for each fixed $x \in \mathbb{R}^{N}$, the mapping  $ t \longmapsto (-\Delta)^{t}_{p} u(x) $ is well defined and of class $C^{\infty}$ on the interval  $ \left( 0,  T_{p}\right) $ , and for
 every $ s \in \left( 0,  T_{p}\right) $, 
	\begin{equation*}
	\begin{aligned}
	\left. \frac{d}{d t}(-\Delta)_{p}^{t} u(x) \right|_{t=s} &= B(N, s, p) (-\Delta)^{s}_{p} u(x) \\
	&\quad - p\,  C(N,s,p) \mathrm{P.V.} \int_{\mathbb{R}^N} \dfrac{|u(x) - u(y)|^{p-2} (u(x) - u(y)) \ln|x - y|}{|x - y|^{N + sp}} dy,
	\end{aligned}
	\end{equation*}
	where $C(N,s,p)$ is the standard fractional normalization constant given in \eqref{equ4}, and
	\begin{equation}\label{equB}
	B(N, s, p) = \frac{d}{dt} \big( \ln C(N,t,p) \big)\Big|_{t=s} = 2\ln 2 + \frac{1}{s}  + \frac{p}{2} \psi \left(\frac{N+sp}{2}\right) + \psi(1-s),
	\end{equation}
	with $\psi$ denoting the digamma function.
\end{proposition}
\begin{proof}
First, we note that the operator $(-\Delta)^{t}_{p}$ is well-defined for every $t \in (0,1)$ when $u \in C^{2}_{c}(\mathbb{R}^{N})$. Indeed, for a fixed $t \in (0,1)$, by setting $z = y - x$, the operator can equivalently be expressed as
	\begin{equation*}
	\begin{aligned}
	&(-\Delta)^t_p u(x):= C(N, t, p)  \mathrm{P.V.} \int_{\mathbb{R}^N} \frac{\left| u(x) - u(x+z)\right|^{p-2} \left( u(x) - u(x+z)\right)}{|z|^{N + tp}} dz \\[4pt]
	&= \frac{C(N, t, p)}{2} \mathrm{P.V.} \int_{\mathbb{R}^N} \frac{\left| u(x) - u(x+z)\right|^{p-2} \left( u(x) - u(x+z)\right) + \left| u(x) - u(x-z)\right|^{p-2} \left( u(x) - u(x-z)\right)}{|z|^{N + tp}}  dz.
	\end{aligned}
	\end{equation*}
	\textbf{Case 1: $1 < p < 2$.} Under the regularity assumption on $u$, we have $u \in C^{1, \alpha}_c(\mathbb{R}^N)$ for some $\alpha > 0$, which will be specified later. Moreover, noting that the mapping $\tau \mapsto \left| \tau\right| ^{p-1}$ is concave, an application of the Lagrange Mean Value Theorem yields
	
	\begin{equation*}
	\begin{aligned}
	&\frac{\Big|\ |u(x) - u(x+z)|^{p-2} \big(u(x) - u(x+z)\big) + |u(x) - u(x-z)|^{p-2} \big(u(x) - u(x-z)\big) \Big|}{|z|^{N + tp}} \\[4pt]
	&\qquad\leq 2^{2-p} \frac{\big|\big(u(x) - u(x-z)\big) + \big(u(x) - u(x+z)\big) \big|^{p-1}}{|z|^{N + tp}} \\[4pt]
	&\qquad= 2^{2-p} \frac{\big|\big(Du(x - \zeta_1 z) - Du(x + \zeta_2 z)\big) \cdot z \big|^{p-1}}{|z|^{N + tp}} \\[4pt]
	&\qquad\leq C |z|^{(\alpha+1)(p-1) - N - tp},
	\end{aligned}
	\end{equation*}
	for some $\zeta_1, \zeta_2 \in [0,1]$, and $C = C(p, \alpha, \zeta_1, \zeta_2) > 0$ is a constant. By choosing $\alpha > \frac{1 - p(1-t)}{p-1}$, the right-hand side becomes integrable near $0$ for any $t \in \left(0, 2(p-1)/p\right)$ (see \cite[Proposition 2.12]{Iannizzotto-Mosconi-Squassina}). We emphasize that the condition $ t \in \left(0, 2(p-1)/p\right) $ is sharp and cannot be relaxed. In fact, this condition cannot be weakened even when the function $u$ is assumed to be smoother, owing to the singular nature of the operator in the case $1<p<2$; see \cite[Remark 2.14]{Iannizzotto-Mosconi-Squassina}, where a suitable example illustrating this phenomenon is provided.  Precisely, for $t \geq 2(p-1)/p$, this choice of $\alpha$ is impossible, and $ (-\Delta)^t_p u(x) $ may fail to be defined pointwise, however smooth $ u $ is.   Indeed, let  $ \chi \in C^{\infty}_{c}(\mathbb{R}^{N}) $ with $ \chi = 1 $ on $ \textbf{B}_{1}(0) $ and
 put $ u(y) := \left| y\right| ^{2} \chi(y) \in C^{2}_{c}(\mathbb{R}^{N}). $ For $ 0 < \left| z\right|  < 1 $ we have $ u(\pm z) = \left| z\right| ^{2} $, hence the
symmetrised integrand at $ x = 0 $ equals
$$ \frac{|u(0) - u(z)|^{p-2} \bigl(u(0) - u(z)\bigr) + |u(0) - u(-z)|^{p-2} \bigl(u(0) - u(-z)\bigr) }{|z|^{N + tp}}  = - 2 \left| z\right| ^{2(p-1) - N - tp}, $$
which has a fixed sign and satisfies $ \int_{\textbf{B}_{1}(0)} \left| z\right| ^{2(p-1) - N - tp} dz = \infty $ precisely when $t \geq 2(p-1)/p$. Furthermore, no cancellation is therefore possible in the principal value, and $  (-\Delta)^t_p u(0) = - \infty.$ \\[4pt]
	\textbf{Case 2: $p \geq 2$.} By also using the Lagrange Mean Value Theorem together with the classical elliptic estimate
	\begin{equation*}
	\bigl| |a|^{ p-2} a - |b|^{ p-2} b \bigr| \leq C(p) \bigl(|a| + |b|\bigr)^{ p-2} |a-b|, \quad \text{for all } a,b \in \mathbb{R}^N \text{ with } |a| + |b| \neq 0,
	\end{equation*}
	and taking into account that $u \in C^2_c(\mathbb{R}^N)$, we obtain
	\begin{equation*}
	\begin{aligned}
	&\frac{\Bigl| |u(x) - u(x+z)|^{p-2} \bigl(u(x) - u(x+z)\bigr) - |u(x-z) - u(x)|^{p-2} \bigl(u(x-z) - u(x)\bigr) \Bigr|}{|z|^{N + tp}} \\[4pt]
	&\leq C(p) \frac{\Bigl(|u(x) - u(x+z)| + |u(x-z) - u(x)|\Bigr)^{p-2} \Bigl| (u(x) - u(x+z)) - (u(x-z) - u(x)) \Bigr|}{|z|^{N + tp}} \\[4pt]
	&= C(p) \frac{|z|^{p-2} \Bigl(|Du(x + \zeta_1 z)| + |Du(x - \zeta_2 z)|\Bigr)^{p-2} \Bigl| (Du(x - \zeta_1 z) - Du(x + \zeta_2 z)) \cdot z \Bigr|}{|z|^{N + tp}} \\[4pt]
	&\leq C \sup_{\mathrm{supp}(u)} |Du|^{p-2} \sup_{\mathrm{supp}(u)} |D^2 u| |z|^{p(1-t) - N},
	\end{aligned}
	\end{equation*}
	for some $\zeta_1, \zeta_2 \in [0,1]$, and constant $C = C(p, \zeta_1, \zeta_2) > 0$. It is straightforward to see that this expression is integrable near $0$ for every $t \in (0,1)$.
	\noindent
Second, we aim to establish the existence of the first-order derivative of the fractional $p$-Laplace operator. To this end, for a fixed $x \in \mathbb{R}^{N}$, we obtain that
	
	\begin{equation}
	\label{equ5}
	\begin{aligned}
	&\left. \frac{d}{dt}  (-\Delta)^{t}_{p} u(x) \right|_{t=s} = \frac{\left. \frac{d}{dt} C(N,t,p) \right|_{t=s}}{C(N,s,p)} (-\Delta)^{s}_{p} u(x) \\[4pt]
	&\quad + C(N,s,p) \left. \frac{d}{dt} \left( \mathrm{P.V.} \int_{\mathbb{R}^N} \frac{\left| u(x) - u(y)\right| ^{p-2}(u(x) - u(y))}{|x - y|^{N + tp}} dy \right) \right|_{t=s} \\[4pt]
	&= \underbrace{ \frac{d}{dt} \big( \ln C(N,t,p) \big)\Big|_{t=s} }_{\textbf{I}_{1}} (-\Delta)^{s}_{p} u(x) \\[4pt]
	&\quad + C(N,s,p) \underbrace{ \mathrm{P.V.} \int_{\mathbb{R}^N} \left|  u(x) - u(y)\right| ^{p-2}(u(x) - u(y)) \left. \frac{d}{dt} \big( |x - y|^{- N - tp} \big) \right|_{t=s}  dy}_{\textbf{I}_{2}}.
	\end{aligned}
	\end{equation}
	We first calculate the term $\textbf{I}_{1}$. Taking into account the normalizing constant $C(N,t,p)$ introduced in \eqref{equ4}, we distinguish between the following two cases:\\[4pt]
	\textbf{Case 1:} $2s > 1$. In this case, we have
	\begin{equation*}
	C(N,t,p) = \dfrac{t p 2^{2 (t-1)}\Gamma \left(\frac{N+tp}{2}\right)}{\pi^{\frac{N-1}{2}} \Gamma \left(1 - t\right) \Gamma \left(\frac{p+1}{2}\right)}.
	\end{equation*}
	Taking logarithms gives
	\begin{equation*}
	\begin{aligned}
	\ln C(N, t, p) &= \ln t + \ln p + 2(t-1)\ln 2 + \ln \Gamma \left(\frac{N+tp}{2}\right) - \frac{N-1}{2}\ln \pi - \ln \Gamma \left(1 - t\right) - \ln \Gamma \left(\frac{p+1}{2}\right).
	\end{aligned}
	\end{equation*}
	Differentiating with respect to $t$ and using the definition of the digamma function $\psi$, we obtain
	\begin{equation*}
	\begin{aligned}
	B(N, s, p) & : = \textbf{I}_{1}  = \frac{d}{dt} \big( \ln C(N,t,p) \big)\Big|_{t=s} = \dfrac{1}{s} + 2 \ln 2 +\frac{p}{2} \psi\Big(\frac{N+sp}{2}\Big) + \psi\Big(1 - s\Big).
	\end{aligned}
	\end{equation*}
	\noindent \textbf{Case 2:} $2s \leq 1$. Here, we have
	\begin{equation*}
	C(N,t,p) = \dfrac{t p 2^{2t-1}\Gamma \left(\frac{N+tp}{2}\right)}{\pi^{\frac{N}{2}} \Gamma \left(1 - t\right)}. 
	\end{equation*}
	Similarly, differentiating gives
	\begin{equation*}
		\begin{aligned}
	B(N, s, p) &= \frac{d}{dt} \big( \ln C(N,t,p) \big)\Big|_{t=s}   = \frac{1}{s} + 2\ln 2 + \frac{p}{2} \psi \left(\frac{N+sp}{2}\right) + \psi(1-s).
		\end{aligned}
	\end{equation*}
	Next, we calculate $\textbf{I}_{2}$. It is easy to see that
	\begin{equation*}
	\textbf{I}_{2} = \mathrm{P.V.} \int_{\mathbb{R}^N} \frac{\left| u(x) - u(y)\right| ^{p-2}(u(x) - u(y))}{|x-y|^{N+sp}} \big(-p \ln |x-y|\big) dy.
	\end{equation*}
 It is worth noting that this integral is finite on the interval $\left(0,T_p\right)$. Moreover, $\mathbf{I}_2 \in C^{\infty}(0,T_p)$. Indeed, since $u$ has compact support, for sufficiently large $|y|$, the integrand behaves like $ \left|\ln |x-y|\right|\,|x-y|^{-N-sp}, $ which is integrable at infinity for all $t\in(0,1)$. On the other hand, in a neighborhood of $x=y$, by exploiting the regularity properties of $u$ together with the Lagrange Mean Value Theorem, the integrand behaves like $ \left|\ln |x-y|\right|\,|x-y|^{(\alpha+1)(p-1)-N-tp} $ for $\alpha>\frac{1-p(1-t)}{p-1}$ when $1<p<2$, and like $ \left|\ln |x-y|\right|\,|x-y|^{p(1-t)-N} $ when $p\geq 2$. In both cases, the integral remains integrable in a neighborhood of $x=y$ for every $t\in(0,T_p)$. Therefore, for each fixed $x \in \mathbb{R}^{N}$, the mapping $ t \longmapsto (-\Delta)^t_p u(x) $ belongs to $C^1(0,T_p)$. By the same argument, repeated differentiation with respect to $t$ shows that, for every $k\geq 2$ and each fixed $x \in \mathbb{R}^{N}$, $ t \longmapsto (-\Delta)^t_p u(x) \in C^k(0,T_p). $  Hence, combining these observations and returning to \eqref{equ5}, we obtain that
	\begin{equation*}
	\begin{aligned}
	\left. \frac{d}{dt} \big( (-\Delta)^{t}_{p} u(x) \big) \right|_{t=s}
	&= B(N, s, p) (-\Delta)^{s}_{p} u(x)  \\[4pt]
	&\quad - p C(N,s,p) \mathrm{P.V.} \int_{\mathbb{R}^N} \frac{\left| u(x) - u(y)\right| ^{p-2}(u(x) - u(y)) \ln|x-y|}{|x-y|^{N+sp}} dy.
	\end{aligned}
	\end{equation*}
	This completes the proof.
\end{proof}

Based on Proposition \ref{prop1}, we introduce our main operator, called the fractional logarithmic $p$-Laplacian. It is defined on $C^{2}_{c}(\mathbb{R}^{N})$ with values in $\mathbb{R}$ and is explicitly obtained as the derivative with respect to {\color{blue}$s \in (0,T_{p})$}, given by:
\begin{equation*}
(-\Delta)_{p}^{s+\log} u(x):= \left. \frac{d}{d t}(-\Delta)_{p}^{t} u(x) \right|_{t=s}, \quad x \in \mathbb{R}^{N}.
\end{equation*}
Furthermore, this operator admits the following equivalent representations
\begin{itemize}
	\item[$(i)$] Integral representation: For every $u \in C^{2}_{c}(\mathbb{R}^{N})$ and $x \in \mathbb{R}^{N}$,
	\begin{equation*}
	\begin{aligned}
	(-\Delta)_{p}^{s+\log} u(x)
	&= B(N, s, p) (-\Delta)_{p}^{s} u(x) \\[4pt]
	&\quad - p\,  C(N, s, p) \mathrm{P.V.} \int_{\mathbb{R}^{N}} \dfrac{|u(x) - u(y)|^{p-2} (u(x) - u(y)) \ln |x - y|}{|x - y|^{N + sp}} dy.
	\end{aligned}
	\end{equation*}
	\item[$(ii)$] Kernel representation: For every $u \in C^{2}_{c}(\mathbb{R}^{N})$ and $x \in \mathbb{R}^{N}$,
	\begin{equation*}
	(-\Delta)_{p}^{s+\log} u(x) = \mathrm{P.V.} \int_{\mathbb{R}^{N}} |u(x) - u(y)|^{p-2} (u(x) - u(y)) \mathbf{K}^{s+ \log}_{p}(|x - y|) dy,
	\end{equation*}
	where the kernel $\mathbf{K}^{s+ \log}_{p}$ is explicitly given by
	\begin{equation}
	\label{equ14}
	\mathbf{K}^{s+ \log}_{p}(r) := \frac{B(N, s, p) C(N, s, p) - p C(N, s, p) \ln r}{r^{N+sp}}, \quad r > 0.
	\end{equation}
\end{itemize}\begin{remark}
	Let $1 < p < \infty$ and $0 < s < T_p$. By the definitions of $C(N,s,p)$ and $B(N,s,p)$, we obtain the following properties:
	\begin{enumerate}
		\item In general, the kernel $\mathbf{K}^{s+\log}_{p}$ is not guaranteed to be nonnegative.
		\item The constant $B(N,s,p)$ exhibits the following asymptotic behavior:
		\begin{equation*}
		\lim_{s \to 0^+} B(N,s,p)=+\infty,
		\qquad
		\lim_{s \to 1^-} B(N,s,p)=-\infty,
		\end{equation*}
		and strictly decreasing with respect to $s$ in a neighborhood of both $s=0^{+}$ and $s=1^{-}$, since 
		$$ \dfrac{d B(N,s,p)}{d s} = - \dfrac{1}{s^{2}} + \dfrac{p^{2}}{4} \psi'\left( \dfrac{N+sp}{2}\right)  - \psi'(1-s)$$
		the first term dominating as $ s \to 0^{+} $ and the last
 as $ s \to 1^{-}.$ It is not monotone on  $ (0, 1) $ in general: for $ N = 1, $ $  p =40 $ one has $ B(1,1/4,40) = 36.52 \cdots < 47.48\cdots = B(1,1/2,40). $ Nevertheless $ B $ is positive on the whole of $ (0, 1/2] $:
\[
B(N,s,p) \geq 1-\gamma >0
\quad \text{for all } N\geq 1,\; p>1,\; s\in(0,1/2].
\]
Indeed, setting $a=sp$ and using the inequality $ \psi(x)\geq \log x-\frac{1}{x}, $ together with the monotonicity of $\psi$ and the identity $ \psi\left(\frac{1}{2}\right)=-\gamma-2\log 2, $ we obtain
\[
\begin{aligned}
B(N,s,p)
&\geq
\frac{N}{s(N+a)}
+\frac{p}{2}\log\left(\frac{N+a}{2}\right)
-\gamma
\\[4pt]
&\geq
\frac{1}{s}
\left(
\frac{1}{1+a}
+\frac{a}{2}\log\left(\frac{1+a}{2}\right)
\right)
-\gamma
\\[4pt]
&\geq
2\left(\frac{1}{2}\right)-\gamma.
\end{aligned}
\]
The last inequality follows from
\[
h(a)
=
\frac{1}{1+a}
+
\frac{a}{2}\log\left(\frac{1+a}{2}\right)
\geq \frac{1}{2},
\qquad \text{for every } a>0.
\]
To see this, setting $u=(1+a)/2$, the above inequality is equivalent to
\[
u(2u-1)\log u \geq u-1.
\]
This follows from
\[
(2u-1)\log u\geq \log u\geq 1-\frac{1}{u},
\]
since $u-1$ and $\log u$ have the same sign. Finally, since $B$ is continuous with respect to $s$, there exists a threshold $s_0>1/2$ such that
\[
B(N,s,p)\geq 0
\qquad\text{for all }s\in(0,s_0).
\]
	\end{enumerate}
\end{remark}

We next show that the operator $(-\Delta)_{p}^{s+\log}$ enjoys the following properties.

\begin{theorem}
	\label{theorem1}
	Let $1 < p < \infty$ and  $0 < s < T_{p}$. Then, for every $u \in C^{2}_{c}(\mathbb{R}^{N})$, one has
	\begin{equation*}
	(-\Delta)_{p}^{s+\log} u \;\longrightarrow\; L_{\Delta_{p}} u \quad \text{in } L^{\infty}(\mathbb{R}^{N}) \quad \text{as } s \to 0^{+},
	\end{equation*}
	where $L_{\Delta_{p}}$ denotes the logarithmic $p$-Laplace operator corresponding to $s=0$, defined in \eqref{equ16}.   Furthermore, we have $ (-\Delta)^{s+\log}_p u \in L^{q}(\mathbb{R}^{N})  $ for all $ 1 < q \leq \infty, $	and the following convergence property holds:
		\[
		\lim_{t \to s}
		\left\|
		\frac{(-\Delta)^t_p u - (-\Delta)^s_p u}{t - s}
		- (-\Delta)^{s+\log}_p u
		\right\|_{L^{q}(\mathbb{R}^{N})}
		= 0.
		\]
		Equivalently, we have the asymptotic expansion
\begin{equation*}
(-\Delta)_{p}^{t} u(x)
= (-\Delta)_{p}^{s} u(x)
+ (t - s) (-\Delta)^{s+\log}_p u(x)
+ O(|t - s|)
\quad \text{as } t \to s \text{ in } L^{q}(\mathbb{R}^{N}), \quad 1 < q \leq \infty.
\end{equation*}
	
\end{theorem}
\begin{proof}
	Fixed $x \in \mathbb{R}^{N}$, and since $u \in C^{2}_{c}(\mathbb{R}^{N})$, there exists $R > 1 $ such that $\operatorname{supp}(u) \subset \mathcal{B}_{R}(x)$. Consequently, we can decompose the operator $(-\Delta)_{p}^{s+\log} u(x)$ as follows
\begin{equation}
\label{equ11}
\begin{aligned}
(-\Delta)_{p}^{s+\log} u(x)
&= \underbrace{C(N, s, p) B(N, s, p) \mathrm{P.V.} \int_{\mathcal{B}_{1}(x)} 
	\frac{|u(x) - u(y)|^{p-2} (u(x) - u(y))}{|x - y|^{N + sp}}  dy}_{\mathbf{I}_{1}}  \\[4pt]
&\quad + \underbrace{C(N, s, p) B(N, s, p) \int_{\mathbb{R}^{N} \setminus \mathcal{B}_{1}(x)} 
	\frac{|u(x) - u(y)|^{p-2} (u(x) - u(y)) - |u(x)|^{p-2} u(x)}{|x - y|^{N + sp}}  dy}_{\mathbf{I}_{2}}  \\[4pt]
&\quad - \underbrace{p C(N, s, p) \mathrm{P.V.} \int_{\mathcal{B}_{1}(x)} 
	\frac{|u(x) - u(y)|^{p-2} (u(x) - u(y)) \ln |x - y|}{|x - y|^{N + sp}}  dy}_{\mathbf{I}_{3}}  \\[4pt]
&\quad - \underbrace{\Bigg( p C(N, s, p) \int_{\mathbb{R}^{N} \setminus \mathcal{B}_{1}(x)} 
	\frac{|u(x) - u(y)|^{p-2} (u(x) - u(y)) \ln |x - y|}{|x - y|^{N + sp}}  dy }  \\[4pt]
&\qquad \underbrace{- C(N, s, p) B(N, s, p) \int_{\mathbb{R}^{N} \setminus \mathcal{B}_{1}(x)} 
	\frac{|u(x)|^{p-2} u(x)}{|x - y|^{N + sp}}  dy \Bigg)}_{\mathbf{I}_{4}}.
\end{aligned}
\end{equation}
	First, by considering the limit as $s \to 0^+$, we can straightforwardly compute
	\begin{equation*}
	\begin{aligned}
	C(N, s, p) B(N, s, p) &= \frac{p 2^{2s-1} \Gamma \left(\frac{N+sp}{2}\right)}{\pi^{\frac{N}{2}} \Gamma \left(1 - s\right)} \Bigg[ 1 + s \Bigg( 2 \ln 2 + \frac{p}{2} \psi \left(\frac{N+sp}{2}\right) + \psi(1-s) \Bigg) \Bigg]  \\[4pt]
	&\longrightarrow \frac{p \Gamma \left(\frac{N}{2}\right)}{2 \pi^{\frac{N}{2}}} = C(N, p) \quad \text{as } s \to 0^+.
	\end{aligned}
	\end{equation*}
	Under the regularity assumption on $u$, we assume that $u \in C^{1,\alpha}_c(\mathbb{R}^N)$ with $\alpha (p-1) > sp$. Then, we have\\[2mm] \textbf{Estimate of} $\mathbf{I}_{1}$. To estimate this integral, let $0 < \varepsilon < 1$. Then, we obtain for some constant $C > 0$
	\begin{equation*}
	\begin{aligned}
	\left| \mathrm{P.V.} \int_{\mathcal{B}_{\epsilon}(x)} \frac{|u(x) - u(y)|^{p-2} (u(x) - u(y))}{|x - y|^{N + sp}} dy \right|
	&\leq C \int_{\mathcal{B}_{\epsilon}(x)} \frac{dy}{|x - y|^{N + sp - \alpha (p - 1)}} \\
	&= \frac{\omega_N}{\alpha (p - 1) - sp} \epsilon^{\alpha (p - 1) - sp},
	\end{aligned}
	\end{equation*}
	and
	
	\begin{equation*}
	\begin{aligned}
	\left| \mathrm{P.V.} \int_{\mathcal{B}_{\epsilon}(x)} \frac{|u(x) - u(y)|^{p-2} (u(x) - u(y))}{|x - y|^{N}} dy \right|
	&\leq C \int_{\mathcal{B}_{\epsilon}(x)} \frac{dy}{|x - y|^{N -  p + 1}} \\
	&= \frac{\omega_N}{p - 1} \epsilon^{p - 1}.
	\end{aligned}
	\end{equation*}
	Hence, we deduce that
	\begin{equation*}
	\begin{aligned}
	\Bigg| \mathrm{P.V.} \int_{\mathcal{B}_{1}(x)} \frac{|u(x) - u(y)|^{p-2} (u(x) - u(y))}{|x - y|^{N + sp}} dy - \mathrm{P.V.} \int_{\mathcal{B}_{1}(x)} \frac{|u(x) - u(y)|^{p-2} (u(x) - u(y))}{|x - y|^{N}} dy \Bigg| \\
	\leq 2^{p-1} \|u\|_{L^{\infty}}^{p-1} \int_{\mathcal{B}_{1}(x) \setminus \mathcal{B}_{\epsilon}(x)} \frac{1}{|x - y|^{N}}\left(e^{sp \ln \dfrac{1}{|x - y|}} - 1\right) dy + \frac{\omega_N}{\alpha (p - 1) - sp} \epsilon^{\alpha (p - 1) - sp} + \frac{\omega_N}{p - 1} \epsilon^{p - 1}.
	\end{aligned}
	\end{equation*}
	Consequently, we obtain the following limit
	\begin{equation}
	\label{equ7}
	\mathbf{I}_{1} \longrightarrow C(N, p) \mathrm{P.V.} \int_{\mathcal{B}_{1}(x)} \frac{|u(x) - u(y)|^{p-2} (u(x) - u(y))}{|x - y|^{N}} dy \quad \text{as } s \to 0^{+}.
	\end{equation}
	\noindent
	\textbf{Estimate of} \(\mathbf{I}_{2}\). We first observe that, for some constant $C>0$, the following identity holds:
	\begin{equation*}
	\begin{aligned}
	&\int_{\mathbb{R}^{N} \setminus \mathcal{B}_{1}(x)} 
	\frac{|u(x)-u(y)|^{p-2}(u(x)-u(y))-|u(x)|^{p-2}u(x)}
	{|x-y|^{N+sp}} dy \\[4pt]
	&\qquad =
	\int_{\mathcal{B}_{R}(x)\setminus \mathcal{B}_{1}(x)}
	\frac{|u(x)-u(y)|^{p-2}(u(x)-u(y))-|u(x)|^{p-2}u(x)}
	{|x-y|^{N+sp}} dy.
	\end{aligned}
	\end{equation*}
Therefore, an application of the Dominated Convergence Theorem yields
\begin{equation*}
\begin{aligned}
&\int_{\mathcal{B}_{R}(x)\setminus \mathcal{B}_{1}(x)}
\frac{|u(x)-u(y)|^{p-2}(u(x)-u(y))-|u(x)|^{p-2}u(x)}
{|x-y|^{N+sp}} dy\\[4pt]
&\longrightarrow
\int_{\mathcal{B}_{R}(x)\setminus \mathcal{B}_{1}(x)}
\frac{|u(x)-u(y)|^{p-2}(u(x)-u(y))-|u(x)|^{p-2}u(x)}
{|x-y|^{N}} dy
\qquad \text{as } s\to 0^{+}.
\end{aligned}
\end{equation*}
Consequently, we deduce that
	\begin{equation}
	\begin{aligned}
	\label{equ8}
	\mathbf{I}_{2}
	&=
	C(N,s,p) B(N,s,p)
	\int_{\mathcal{B}_{R}(x)\setminus \mathcal{B}_{1}(x)}
	\frac{|u(x)-u(y)|^{p-2}(u(x)-u(y))-|u(x)|^{p-2}u(x)}
	{|x-y|^{N+sp}} dy \\[4pt]
	&\longrightarrow
	C(N,p)
	\int_{\mathcal{B}_{R}(x)\setminus \mathcal{B}_{1}(x)}
	\frac{|u(x)-u(y)|^{p-2}(u(x)-u(y))-|u(x)|^{p-2}u(x)}
	{|x-y|^{N}} dy \\[4pt]
	&=
	C(N,p)
	\int_{\mathbb{R}^{N}\setminus \mathcal{B}_{1}(x)}
	\frac{|u(x)-u(y)|^{p-2}(u(x)-u(y))-|u(x)|^{p-2}u(x)}
	{|x-y|^{N}} dy,
	\qquad \text{as } s\to0^{+}.
	\end{aligned}
	\end{equation}
	\noindent
	\textbf{Estimate of} \(\mathbf{I}_{3}\).  Observe that, for some constant $ C > 0 $, we have
	\begin{equation*}
	\begin{aligned}
	&\left| p C(N, s, p) \mathrm{P.V.} \int_{\mathcal{B}_{1}(x)} \frac{|u(x) - u(y)|^{p-2} (u(x) - u(y)) \ln |x - y|}{|x - y|^{N + sp}} dy\right|  \\[4pt]
	&\qquad\leq C C(N, s, p)  \int_{\mathcal{B}_{1}(x)} \frac{-\ln |x - y|}{|x - y|^{N + sp - \alpha (p - 1)}} dy \\[4pt]
	&\qquad  = C C(N, s, p) \frac{\omega_N}{\big(\alpha(p-1) - sp\big)^2}  \to 0\quad \text{as } s \to 0^{+}.
	\end{aligned}
	\end{equation*}
	Then, we get that
	\begin{equation}
	\label{equ9}
	\mathbf{I}_{3} \longrightarrow 0 \quad \text{as } s \to 0^{+}.
	\end{equation}
	\noindent
	\textbf{Estimate of} \(\mathbf{I}_{4}\). First, a straightforward computation shows that
	\begin{equation}
	\label{equ6}
	\begin{aligned}
	& p C(N,s,p) |u(x)|^{p-2}u(x)
	\int_{\mathbb{R}^{N}\setminus \mathcal{B}_{R}(x)}
	\frac{\ln |x-y|}{|x-y|^{N+sp}} dy \\[4pt]
	&\quad + p C(N,s,p)
	\int_{\mathcal{B}_{R}(x)\setminus \mathcal{B}_{1}(x)}
	\frac{|u(x)-u(y)|^{p-2}(u(x)-u(y)) \ln |x-y|}
	{|x-y|^{N+sp}} dy \\[4pt]
	&= p C(N,s,p) |u(x)|^{p-2}u(x)
	\int_{\mathbb{R}^{N}\setminus \mathcal{B}_{1}(x)}
	\frac{\ln |x-y|}{|x-y|^{N+sp}} dy \\[4pt]
	&\qquad - p C(N,s,p) |u(x)|^{p-2}u(x)
	\int_{\mathcal{B}_{R}(x)\setminus \mathcal{B}_{1}(x)}
	\frac{\ln |x-y|}{|x-y|^{N+sp}} dy \\[4pt]
	&\quad + p C(N,s,p)
	\int_{\mathcal{B}_{R}(x)\setminus \mathcal{B}_{1}(x)}
	\frac{|u(x)-u(y)|^{p-2}(u(x)-u(y)) \ln |x-y|}
	{|x-y|^{N+sp}} dy.
	\end{aligned}
	\end{equation}
Next, we establish the following auxiliary estimate:
\begin{equation*}
\int_{\mathcal{B}_{R}(x)\setminus \mathcal{B}_{1}(x)}
\frac{\ln |x-y|}{|x-y|^{N+sp}} dy
\longrightarrow
\frac{\omega_N}{2}(\ln R)^2
\qquad \text{as } s \to 0^{+},
\end{equation*}
	and
	\begin{equation*}
	\begin{aligned}
	&\left|
	\int_{\mathcal{B}_{R}(x)\setminus \mathcal{B}_{1}(x)}
	\frac{|u(x)-u(y)|^{p-2}(u(x)-u(y)) \ln |x-y|}
	{|x-y|^{N+sp}}
	 dy
	\right| \\[4pt]
	&\qquad \leq
	\int_{\mathcal{B}_{R}(x)\setminus \mathcal{B}_{1}(x)}
	\frac{\ln |x-y|}
	{|x-y|^{N+sp-\alpha(p-1)}} dy \\[4pt]
	&\qquad =
	\frac{\omega_{N}}{\big(\alpha(p-1)-sp\big)^2}
	\Big[
	R^{\alpha(p-1)-sp}
	\big(
	(\alpha(p-1)-sp)\ln R -1
	\big)
	+1
	\Big].
	\end{aligned}
	\end{equation*}
Hence, recalling that \( C(N,s,p) \) is defined in \eqref{equ4}, we deduce that the last two terms in \eqref{equ6} converge to zero as \( s \to 0^{+} \). Moreover, we observe that, for sufficiently small values of \( s \),
	\begin{equation*}
	\begin{aligned}
	& p C(N, s, p) \int_{\mathbb{R}^{N} \setminus \mathcal{B}_{1}(x)} \frac{\ln |x - y|}{|x - y|^{N + sp}} dy - C(N, s, p) B(N, s, p) \int_{\mathbb{R}^{N} \setminus \mathcal{B}_{1}(x)} \frac{dy}{|x - y|^{N + sp}} \\[4pt]
	&\quad = C(N, s, p)  \frac{\omega_N}{sp} \left[\dfrac{1}{s} - B(N, s, p)\right] \\[4pt]
	&\quad = \frac{2^{2s-1} \omega_N \Gamma \left(\frac{N+sp}{2}\right)}{\pi^{N/2} \Gamma(1 - s)} \left[ \frac{1}{s} - \left( \frac{1}{s} + 2\ln 2 + \frac{p}{2} \psi \left(\frac{N+sp}{2}\right) + \psi(1-s)\right) \right] \\[4pt]
	&\quad \longrightarrow -\left( 2\ln 2 
	+ \frac{p}{2} \psi \left(\frac{N}{2}\right) - \gamma\right) \quad \text{as } s \to 0^{+}.
	\end{aligned}
	\end{equation*}
Therefore, we deduce that
\begin{equation}
\label{equ10}
\mathbf{I}_{4} \to - \rho(N, p) |u(x)|^{p-2} u(x) 
\quad \text{as } s \to 0^{+}.
\end{equation}
	Combining \eqref{equ7}-\eqref{equ10} and passing to the limit in \eqref{equ11} as \( s \to 0^{+} \), we conclude that
	\begin{equation*}
	(-\Delta)_{p}^{s+\log} u \;\longrightarrow\; L_{\Delta_{p}} u \quad \text{in } L^{\infty}(\mathbb{R}^{N}) \quad \text{as } s \to 0^{+}.
	\end{equation*}
On the other hand, as a direct consequence of Proposition~\ref{prop1}, it holds that for every $s \in (0,1)$,
\[
\lim_{t \to s} \frac{(-\Delta)_{p}^{t} u(x) - (-\Delta)_{p}^{s} u(x)}{t - s}
= (-\Delta)_{p}^{s+\log} u(x)
\quad \text{in } \mathbb{R}^{N}.
\]
 Applying Taylor's theorem with remainder, there exists $\theta = \theta(s,t)$ lying between $s$ and $t$ such that
		\[
		(-\Delta)_{p}^{t} u(x)
		= (-\Delta)_{p}^{s} u(x)
		+ (t - s) (-\Delta)_{p}^{s+\log} u(x)
		+ \frac{(t - s)^{2}}{2}
		\left.\frac{d^{2}}{dt^{2}}(-\Delta)_{p}^{t} u(x) \right|_{t=\theta}.
		\]
		Consequently, we infer that
		\begin{equation}\label{equ26}
		\left|
		\frac{(-\Delta)_{p}^{t} u(x) - (-\Delta)_{p}^{s} u(x)}{t - s}
		- (-\Delta)_{p}^{s+\log} u(x)
		\right|
		=
		\frac{|t - s|}{2}
		\left|
		\left.\frac{d^{2}}{dt^{2}}(-\Delta)_{p}^{t} u(x) \right|_{t=\theta}
		\right|.
		\end{equation}
		Moreover, we have
		\begin{equation*}
		\begin{aligned}
		\left. \frac{d^{2}}{dt^{2}}(-\Delta)_{p}^{t} u(x) \right|_{t=\theta}
		&= \left. \frac{d}{dt}(-\Delta)_{p}^{t+\log} u(x) \right|_{t=\theta} \\
		&= \left(	\left. \frac{d}{d t} B(N, t,p)\right|_{t=\theta} + B(N,\theta,p)^{2}\right) (-\Delta)_{p}^{\theta} u(x) \\[4pt]
		&\quad - p B(N,\theta,p)\big(C(N,\theta,p) + 1 \big) \mathrm{P.V.} \int_{\mathbb{R}^{N}}
		\frac{|u(x)-u(y)|^{p-2}(u(x)-u(y)) \ln|x-y|}
		{|x-y|^{N+\theta p}} dy \\[4pt]
		&\quad +p^{2}  C(N,\theta,p) \mathrm{P.V.} \int_{\mathbb{R}^{N}}
		\frac{|u(x)-u(y)|^{p-2}(u(x)-u(y)) \big(\ln|x-y|\big)^{2}}
		{|x-y|^{N+\theta p}} dy .
		\end{aligned}
		\end{equation*}
		Let $\delta>0$ be sufficiently small, and assume that $|s-t|<\delta$. Under this assumption, it is readily verified that there exists a constant $C>0$, independent of $\theta$, such that
		\begin{equation*}
		\begin{aligned}
		\left|
		\left.\frac{d^{2}}{dt^{2}}(-\Delta)_{p}^{t} u(x) \right|_{t=\theta}
		\right|
		& \leq C  \Bigg[\underbrace{\mathrm{P.V.} \int_{\mathbb{R}^{N}}
			\frac{|u(x)-u(y)|^{p-1}}
			{|x-y|^{N+\theta p}} dy}_{\mathbf{E}_{1}} + \underbrace{\mathrm{P.V.} \int_{\mathbb{R}^{N}}
			\frac{|u(x)-u(y)|^{p-1} |\ln|x-y|| }
			{|x-y|^{N+\theta p}} dy}_{\mathbf{E}_{2}} \\[4pt]
		&\qquad + \underbrace{\mathrm{P.V.} \int_{\mathbb{R}^{N}}
			\frac{|u(x)-u(y)|^{p-1}\big(\ln|x-y|\big)^{2}}
			{|x-y|^{N+\theta p}} dy}_{\mathbf{E}_{3}} \Bigg] .
		\end{aligned}
		\end{equation*}
	Next, we fix $R>4$ such that  $ \mathrm{supp}(u)\subset \mathcal{B}_{\frac{R}{4}}(0). $ Throughout the sequel, $C>0$ denotes a constant independent of $\theta$. 	For any $x\in \mathbb{R}^{N}$, we first consider the term $\mathbf{E}_{1}$, which can be decomposed as follows:
		\begin{equation}\label{equ27}
		\begin{aligned}
		\mathbf{E}_{1}
		&= \int_{\mathbb{R}^{N}\setminus \mathcal{B}_{R}(x)}
		\frac{|u(x)-u(y)|^{p-1}}{|x-y|^{N+\theta p}} dy + \mathrm{P.V.} \int_{\mathcal{B}_{R}(x)}
		\frac{|u(x)-u(y)|^{p-1}}{|x-y|^{N+\theta p}} dy.
		\end{aligned}
		\end{equation}
		We distinguish the following cases for the first term: \\[4pt]
		$ \bullet $ Let $x \in \mathcal{B}_{\frac{R}{2}}(0)$ and $y \in \mathbb{R}^{N} \setminus \mathcal{B}_{R}(x)$. Hence $y \notin \operatorname{supp}(u)$, we deduce that $u(y)=0$. Hence,
		\begin{equation}\label{equ28}
		\begin{aligned}
		\int_{\mathbb{R}^{N}\setminus \mathcal{B}_{R}(x)}
		\frac{|u(x)-u(y)|^{p-1}}{|x-y|^{N+\theta p}} dy
		&= |u(x)|^{p-1} \int_{\mathbb{R}^{N}\setminus \mathcal{B}_{R}(x)}
		\frac{dy}{|x-y|^{N+\theta p}} \\
		&= \frac{\omega_N}{\theta p} R^{-\theta p} |u(x)|^{p-1}
		\leq C |u(x)|^{p-1},
		\end{aligned}
		\end{equation}
		$ \bullet $ Let \( x \in \mathbb{R}^{N} \setminus \mathcal{B}_{\frac{R}{2}}(0) \) and \( y \in \mathbb{R}^{N} \setminus \mathcal{B}_{R}(x) \), with the additional assumption that \( y \in \operatorname{supp}(u) \). In this setting, we still have \( u(x)=0 \). Moreover, since \( y \in \operatorname{supp}(u) \), it follows that $ |y| \leq \frac{|x|}{2}, $ which implies $ |x-y| \geq \frac{|x|}{2} > 1. $ Consequently, we deduce that
		\begin{equation}\label{equ29}
		\int_{\mathbb{R}^{N}\setminus \mathcal{B}_{R}(x)}
		\frac{|u(x)-u(y)|^{p-1}}{|x-y|^{N+\theta p}} dy
		\leq 2^{N}  |x|^{-N}  \|u\|_{L^{p-1}(\mathbb{R}^{N})}^{p-1}.
		\end{equation}
		We also distinguish the following cases for the second term in \eqref{equ27}: \\[4pt]
		$ \bullet $ Let $x \in \mathcal{B}_{\frac{R}{2}}(0)$ and $y \in \mathcal{B}_{R}(x)$. Consequently, under the assumption that $\alpha (p-1) > \theta p$, we infer that
		\begin{equation}\label{equ30}
		\begin{aligned}
		\mathrm{P.V.} \int_{\mathcal{B}_{R}(x)}
		\frac{|u(x)-u(y)|^{p-1}}{|x-y|^{N+\theta p}} dy
		& \leq C \int_{\mathcal{B}_{R}(x)} \frac{dy}{|x - y|^{N + \theta p - \alpha (p - 1)}} \\[4pt]
		&= \frac{\omega_N}{\alpha(p-1)-\theta p} R^{\alpha(p-1)-\theta p}
		\leq C.
		\end{aligned}
		\end{equation}
		$ \bullet $ Let \( x \in \mathbb{R}^{N} \setminus \mathcal{B}_{\frac{R}{2}}(0) \) and $y \in \mathcal{B}_{R}(x)$ with $y \in \mathrm{supp}(u)$. Then \( u(x) = 0 \), and hence
		\begin{equation}\label{equ31}
		\int_{\mathcal{B}_{R}(x)}
		\frac{|u(x)-u(y)|^{p-1}}{|x-y|^{N+\theta p}} dy
		= \int_{\mathcal{B}_{R}(x)}
		\frac{|u(y)|^{p-1}}{|x-y|^{N+\theta p}} dy
		\leq 2^{N} |x|^{-N} \|u\|_{L^{p-1}(\mathbb{R}^{N})}^{p-1}.
		\end{equation}
	Consequently, for any $1 < q < \infty$, it follows from \eqref{equ28}-\eqref{equ31} that
	\begin{equation*}
	\int_{\mathbb{R}^{N}} \mathbf{E}_{1}^{q} dx
	\leq C \Bigg[
	\|u\|_{L^{q(p-1)}(\mathbb{R}^{N})}^{q(p-1)}
	+ 2^{qN+1} \frac{\omega_N}{N(q-1)}
	\left(\frac{R}{2}\right)^{N(1-q)}
	\|u\|_{L^{p-1}(\mathbb{R}^{N})}^{q(p-1)}
	+ \omega_N \left(\frac{R}{2}\right)^N
	\Bigg]
	< \infty,
	\end{equation*}
	while for $q = \infty$, we obtain
	\begin{equation*}
	\left\| \mathbf{E}_{1}\right\|_{L^{\infty}(\mathbb{R}^{N})}
	\leq C \Bigg[
	\|u\|_{L^{\infty}(\mathbb{R}^{N})}^{p-1}
	+ 2^{2N+1} R^{-N}
	\|u\|_{L^{p-1}(\mathbb{R}^{N})}^{p-1}
	+ 1
	\Bigg]
	< \infty.
	\end{equation*}
\textbf{Estimate of} $\mathbf{E}_{2}$. In a similar manner, we write
\begin{equation}\label{equ34}
\begin{aligned}
\mathbf{E}_{2}
&= \int_{\mathbb{R}^{N}\setminus \mathcal{B}_{R}(x)}
\frac{|u(x)-u(y)|^{p-1} |\ln|x-y||}{|x-y|^{N+\theta p}} dy
+ \mathrm{P.V.} \int_{\mathcal{B}_{R}(x)}
\frac{|u(x)-u(y)|^{p-1} |\ln|x-y||}{|x-y|^{N+\theta p}} dy.
\end{aligned}
\end{equation}
We follow the same method for $\mathbf{E}_{1}$, with slightly different computations. Indeed, for the first term, we have:\\[4pt]
$\bullet$ Let $x \in \mathcal{B}_{\frac{R}{2}}(0)$ and $y \in \mathbb{R}^{N} \setminus \mathcal{B}_{R}(x)$. Since $y \notin \operatorname{supp}(u)$, it follows that 
\begin{equation*}
\begin{aligned}
\int_{\mathbb{R}^{N}\setminus \mathcal{B}_{R}(x)}
\frac{|u(x)-u(y)|^{p-1} \ln|x-y|}{|x-y|^{N+\theta p}} dy
&= |u(x)|^{p-1}
\int_{\mathbb{R}^{N}\setminus \mathcal{B}_{R}(x)}
\frac{\ln|x-y|}{|x-y|^{N+\theta p}} dy \\
&= \omega_{N-1} R^{-\theta p}
\left(\frac{\ln R}{\theta p}+\frac{1}{(\theta p)^2}\right) |u(x)|^{p-1} \leq C |u(x)|^{p-1}.
\end{aligned}
\end{equation*}
$\bullet$ Let \( x \in \mathbb{R}^{N} \setminus \mathcal{B}_{\frac{R}{2}}(0) \) and \( y \in \mathbb{R}^{N} \setminus \mathcal{B}_{R}(x) \), and assume in addition that \( y \in \operatorname{supp}(u) \). In this case, one has \( u(x)=0 \) and \( |y| \leq \frac{|x|}{2} \), which yields \( |x-y| \geq \frac{|x|}{2} > 1 \). Consequently, we obtain
\begin{equation*}
\begin{aligned}
\int_{\mathbb{R}^{N}\setminus \mathcal{B}_{R}(x)}
\frac{|u(x)-u(y)|^{p-1} \ln|x-y|}{|x-y|^{N+\theta p}} dy
&\leq \int_{\operatorname{supp}(u)}
\frac{|u(y)|^{p-1} \ln|x-y|}{|x-y|^{N+\theta p}} dy \\
&\leq 2^{N}  |x|^{-N} \ln\left( \dfrac{3}{2}\left| x\right|\right)  \|u\|_{L^{p-1}(\mathbb{R}^{N})}^{p-1}.
\end{aligned}
\end{equation*}
 	For the second term in \eqref{equ34}, let $x \in \mathcal{B}_{\frac{R}{2}}(0)$ and $y \in \mathcal{B}_{R}(x)$. Then, we deduce that
 	\begin{equation*}
 	\begin{aligned}
 	\mathrm{P.V.} &\int_{\mathcal{B}_{R}(x)}
 	\frac{|u(x)-u(y)|^{p-1} |\ln|x-y||}{|x-y|^{N+\theta p}} dy
 	 \leq C \int_{\mathcal{B}_{R}(x)} \frac{|\ln|x-y||}{|x - y|^{N + \theta p - \alpha (p - 1)}} dy \\[4pt]
 	&=\omega_N \left[
 	R^{\alpha(p-1)-\theta p}
 	\left(
 	\frac{\ln R}{\alpha(p-1)-\theta p}
 	-
 	\frac{1}{\left(\alpha(p-1)-\theta p\right)^2}
 	\right)
 	+
 	\frac{2}{\left(\alpha(p-1)-\theta p\right)^2}
 	\right]
 	\leq C.
 	\end{aligned}
 	\end{equation*}
 	$ \bullet $ Let \( x \in \mathbb{R}^{N} \setminus \mathcal{B}_{\frac{R}{2}}(0) \) and $y \in \mathcal{B}_{R}(x)$ with $y \in \mathrm{supp}(u)$. Then \( u(x) = 0 \), and hence
 	\begin{equation*}
 	\int_{\mathcal{B}_{R}(x)}
 	\frac{|u(x)-u(y)|^{p-1} |\ln|x-y||}{|x-y|^{N+\theta p}} dy
 	\leq 2^{N}  |x|^{-N} \ln\left( \dfrac{3}{2}\left| x\right|\right)  \|u\|_{L^{p-1}(\mathbb{R}^{N})}^{p-1}.
 	\end{equation*}
   Consequently, for any $1 < q < \infty$, it follows that
\begin{align*}
 	\int_{\mathbb{R}^{N}} \mathbf{E}_{2}^{q} dx
 	&\leq C \left[
 	\|u\|_{L^{q(p-1)}(\mathbb{R}^{N})}^{q(p-1)}
 	+ 2 \omega_N \left(\frac{3}{2}\right)^{(q-1)N}
 	\frac{\Gamma \left(q+1, (q-1)N\ln \frac{3R}{4}\right)}{\big((q-1)N\big)^{q+1}}
 	\|u\|_{L^{p-1}(\mathbb{R}^{N})}^{q(p-1)}
 	+ \omega_N \left(\frac{R}{2}\right)^N
 	\right]\\
 	&< \infty.
\end{align*}
and for $ q = \infty $
				\begin{equation*}
		\left\| \mathbf{E}_{2}\right\|_{L^{\infty}(\mathbb{R}^{N})}
		\leq C \left[
		\|u\|_{L^{q(p-1)}(\mathbb{R}^{N})}^{q(p-1)}
		+ 
		\|u\|_{L^{p-1}(\mathbb{R}^{N})}^{q(p-1)}
		+ 1
		\right]
		< \infty.
		\end{equation*}
		\textbf{Estimate of} $\mathbf{E}_{3}$. We further decompose $\mathbf{E}_{3}$ as follows
		\begin{equation*}
		\begin{aligned}
		\mathbf{E}_{3}
		&= \int_{\mathbb{R}^{N}\setminus \mathcal{B}_{R}(x)}
		\frac{|u(x)-u(y)|^{p-1} (\ln|x-y|)^{2}}{|x-y|^{N+\theta p}} dy + \mathrm{P.V.} \int_{\mathcal{B}_{R}(x)}
		\frac{|u(x)-u(y)|^{p-1} (\ln|x-y|)^{2}}{|x-y|^{N+\theta p}} dy .
		\end{aligned}
		\end{equation*}
	As in the cases of $\mathbf{E}_{1}$ and $\mathbf{E}_{2}$, we first estimate the first term of $\mathbf{E}_{3}$.\\
		$\bullet$ Let $x \in \mathcal{B}_{\frac{R}{2}}(0)$ and $y \in \mathbb{R}^{N} \setminus \mathcal{B}_{R}(x)$. Since $y \notin \operatorname{supp}(u)$, we have $u(y)=0$, and hence
		\begin{equation*}
		\begin{aligned}
		\int_{\mathbb{R}^{N}\setminus \mathcal{B}_{R}(x)}
		\frac{|u(x)-u(y)|^{p-1} (\ln|x-y|)^{2}}{|x-y|^{N+\theta p}} dy
		&= |u(x)|^{p-1}
		\int_{\mathbb{R}^{N}\setminus \mathcal{B}_{R}(x)}
		\frac{ (\ln|x-y|)^{2}}{|x-y|^{N+\theta p}} dy  \\[4pt]
		&= \omega_{N} R^{-\theta p}
		\left(
		\frac{(\ln R)^{2}}{\theta p}
		+\frac{2\ln R}{(\theta p)^{2}}
		+\frac{2}{(\theta p)^{3}}
		\right) |u(x)|^{p-1}  \\[4pt]
		&\leq C |u(x)|^{p-1}.
		\end{aligned}
		\end{equation*}
		$\bullet$ Let \( x \in \mathbb{R}^{N} \setminus \mathcal{B}_{\frac{R}{2}}(0) \) and \( y \in \mathbb{R}^{N} \setminus \mathcal{B}_{R}(x) \), and assume furthermore that \( y \in \operatorname{supp}(u) \). Hence, we get that
		\begin{equation*}
		\begin{aligned}
		\int_{\mathbb{R}^{N}\setminus \mathcal{B}_{R}(x)}
		\frac{|u(x)-u(y)|^{p-1} (\ln|x-y|)^{2}}{|x-y|^{N+\theta p}} dy
		&\leq \int_{\operatorname{supp}(u)}
		\frac{|u(y)|^{p-1} (\ln|x-y|)^{2}}{|x-y|^{N+\theta p}} dy  \\[4pt]
		&\leq 2^{N}  |x|^{-N} \ln^{2} \left(\dfrac{3}{2}|x|\right)
		\|u\|_{L^{p-1}(\mathbb{R}^{N})}^{p-1}.
		\end{aligned}
		\end{equation*}
		For the second term, let \( x \in \mathcal{B}_{\frac{R}{2}}(0) \) and \( y \in \mathcal{B}_{R}(x) \). Then, it follows that
		\begin{equation*}
		\begin{aligned}
		\mathrm{P.V.} &\int_{\mathcal{B}_{R}(x)}
		\frac{|u(x)-u(y)|^{p-1} (\ln|x-y|)^{2}}{|x-y|^{N+\theta p}} dy
		\leq C \int_{\mathcal{B}_{R}(x)} \frac{(\ln|x-y|)^{2}}{|x - y|^{N + \theta p - \alpha (p - 1)}} dy \\[4pt]
		&= \omega_N \frac{R^{\alpha(p-1)-\theta p}}{\alpha(p-1)-\theta p}
		\left[
		(\ln R)^2
		- \frac{2\ln R}{\alpha(p-1)-\theta p}
		+ \frac{2}{(\alpha(p-1)-\theta p)^2}
		\right]< C.
		\end{aligned}
		\end{equation*}
		$\bullet$ Let \( x \in \mathbb{R}^{N} \setminus \mathcal{B}_{\frac{R}{2}}(0) \) and \( y \in \mathcal{B}_{R}(x) \) with \( y \in \mathrm{supp}(u) \). Then \( u(x)=0 \), and therefore
		\begin{equation*}
		\int_{\mathcal{B}_{R}(x)}
		\frac{|u(x)-u(y)|^{p-1} (\ln|x-y|)^{2}}{|x-y|^{N+\theta p}} dy
		\leq 2^{N}  |x|^{-N} \ln^{2} \left(\dfrac{3}{2}|x|\right)
		\|u\|_{L^{p-1}(\mathbb{R}^{N})}^{p-1}.
		\end{equation*}
		Consequently, for any \( 1 < q < \infty \), we obtain
		\begin{align*}
		\int_{\mathbb{R}^{N}} \mathbf{E}_{3}^{q} dx
		&\leq C \Bigg[
		\|u\|_{L^{q(p-1)}(\mathbb{R}^{N})}^{q(p-1)}
		+ 2 \omega_N \left(\frac{3}{2}\right)^{(q-1)N}
		\frac{\Gamma \left(2q+1, (q-1)N \ln \frac{3R}{4}\right)}{\big((q-1)N\big)^{2q+1}}
		\|u\|_{L^{p-1}(\mathbb{R}^{N})}^{q(p-1)}  + \omega_N \left(\frac{R}{2}\right)^N
		\Bigg] \\
		&< \infty.
		\end{align*}
		For \( q = \infty \), we have
		\begin{equation*}
		\left\| \mathbf{E}_{3}\right\|_{L^{\infty}(\mathbb{R}^{N})}
		\leq C \left[
		\|u\|_{L^{q(p-1)}(\mathbb{R}^{N})}^{q(p-1)}
		+ \|u\|_{L^{p-1}(\mathbb{R}^{N})}^{q(p-1)}
		+ 1
		\right]
		< \infty.
		\end{equation*}
	Combining the above estimates with \eqref{equ26}, we deduce that for $1 < q < \infty$,
	\[
	\left\|
	\frac{(-\Delta)^s_p u - (-\Delta)^t_p u}{s - t}
	- (-\Delta)^{s+\log}_p u
	\right\|_{L^{q}(\mathbb{R}^{N})}
	\leq C \frac{|t - s|^{q}}{2^{q}}
	\longrightarrow 0 \quad \text{as } t \to s.
	\]
	In the case $q = \infty$, we similarly obtain
	\[
	\left\|
	\frac{(-\Delta)^s_p u - (-\Delta)^t_p u}{s - t}
	- (-\Delta)^{s+\log}_p u
	\right\|_{L^{\infty}(\mathbb{R}^{N})}
	\leq C \frac{|t - s|}{2}
	\longrightarrow 0 \quad \text{as } t \to s.
	\]
	This completes the proof.	
\end{proof}
Now, with the operator and its underlying kernel $\mathbf{K}^{s+ \log}_{p}$ properly defined, we can now construct the associated functional framework. To provide the necessary baseline context, we first recall the standard definitions and norm structures of the classical fractional Sobolev space $W^{s,p}(\mathbb{R}^N)$.  First, let $\Omega \subset \mathbb{R}^{N}$ be a bounded open set with Lipschitz boundary. Let $u:\mathbb{R}^N \to \mathbb{R}$ be a measurable function and let $p \in [1, +\infty)$. We denote by $\|\cdot\|_{L^{p}(\Omega)}$ the norm in the space $L^{p}(\Omega)$, defined by
\begin{equation*}
\|u\|^{p}_{L^p(\Omega)} := \int_{\Omega} |u|^p dx.
\end{equation*}
The fractional Sobolev space $W^{s,p}(\mathbb{R}^N)$ is defined as
\begin{equation*}
W^{s,p}(\mathbb{R}^N) := \left\{ u \in L^p(\mathbb{R}^N) :\iint_{\mathbb{R}^{2N}} \frac{|u(x) - u(y)|^p}{|x-y|^{N+sp}} dx dy < \infty \right\},
\end{equation*}
equipped with the norm

\begin{equation*}
\|u\|_{W^{s,p}(\mathbb{R}^N)}^{p} := \|u\|_{L^p(\mathbb{R}^N)}^p + \iint_{\mathbb{R}^{2N}} \frac{|u(x) - u(y)|^p}{|x-y|^{N+sp}} dx dy.
\end{equation*}
The space $W_0^{s,p}(\Omega)$ consists of functions
\begin{equation*}
W_0^{s,p}(\Omega) := \left\{ u \in W^{s,p}(\mathbb{R}^N) :u = 0 \text{ a.e. in } \mathbb{R}^N \setminus \Omega \right\},
\end{equation*}
and it is endowed with the Gagliardo semi-norm
\begin{equation*}
\|u\|_{W_0^{s,p}(\Omega)} := [u]_{s, p} = \left( \iint_{\mathbb{R}^{2N}}\frac{|u(x) - u(y)|^p}{|x-y|^{N+sp}} dx dy \right)^{1/p}.
\end{equation*}

\begin{remark}
	\label{remark1}
	It is well-known that the fractional Poincar\'e inequality (see Di Nezza et al. \cite[Theorem 6.5]{DiNezza-Palatucci-Valdinoci}) ensures that the norms $\|\cdot\|_{W^{s,p}(\mathbb{R}^N)}$ and $\|\cdot\|_{W_0^{s,p}(\Omega)}$ are equivalent on $W_0^{s,p}(\Omega)$. Moreover, for $N > sp$, the results by Bisci et al. \cite[Theorem 6.5]{Bisci-Radulescu-Servadei} and Di Nezza et al. \cite[Theorem 6.7]{DiNezza-Palatucci-Valdinoci} imply that $W_0^{s,p}(\Omega)$ is continuously embedded in $L^r(\Omega)$ for $1 \leq r \leq p_s^* := \frac{Np}{N-sp}$, with the embedding being compact whenever $1 \leq r < p_s^*$.
\end{remark}

We now introduce the fractional logarithmic Sobolev space \(W^{s+\log,p}(\mathbb{R}^{N})\) as the natural energy space corresponding to the positive part of the fractional logarithmic kernel
\begin{equation*}
W^{s+\log,p}(\mathbb{R}^{N}) := \left\{ u : \mathbb{R}^{N} \to \mathbb{R} \;\middle|\; u \in L^{p}(\mathbb{R}^{N}) \text{ and } [u]_{s+\log,p} < \infty \right\},
\end{equation*}
equipped with the norm
\begin{equation}
\label{equ12}
\|u\|_{W^{s+\log,p}(\mathbb{R}^{N})}^{p}:= \|u\|_{L^{p}(\mathbb{R}^{N})}^{p} + [u]_{s+\log,p}^{p}.
\end{equation}
Here, the Gagliardo seminorm is defined as
\begin{equation*}
[u]_{s+\log,p}:= \left( \iint_{\mathbb{R}^{2N}} |u(x)-u(y)|^{p} \mathbf{k}^{s+\log,+}_{p}(x-y) dx dy \right)^{\frac{1}{p}},
\end{equation*}
where the positive kernel $\mathbf{k}^{s+\log,+}_{p}$ is given by (with the notation $ (a)_{\pm} :=\max\left\lbrace 0,  \pm a\right\rbrace$)
\begin{equation*}
\mathbf{k}^{s+\log,+}_{p}(r):= \frac{C(N,s,p) (-\ln|r|)_{+}}{|r|^{N+sp}}, \quad r \in \mathbb{R}^{N}\setminus\{0\}.
\end{equation*}
Notice that the full fractional logarithmic kernel defined previously in \eqref{equ14} admits the natural decomposition
\begin{equation*}
\mathbf{K}^{s+\log}_{p}(r):= B(N, s, p)C(N, s, p)r^{-N-sp} + p \mathbf{k}^{s+\log,+}_{p}(r) - p \mathbf{k}^{s+\log,-}_{p}(r), \quad r > 0.
\end{equation*}
The corresponding Dirichlet subspace is defined as
\begin{equation*}
W^{s+\log,p}_{0}(\Omega) := \Bigl\{ u \in W^{s+\log,p}(\mathbb{R}^{N}) \ : \ u = 0 \ \text{a.e. in } \mathbb{R}^{N} \setminus \Omega \Bigr\}.
\end{equation*}
We point out that, by a new Poincar\'e inequality associated with our fractional logarithmic operator (see Proposition \ref{proposition3} below), the norms $\|\cdot\|_{W^{s+\log,p}(\mathbb{R}^{N})}$ and $[\cdot]_{s+\log,p}$ are equivalent on $W_{0}^{s+\log,p}(\Omega)$. Accordingly, throughout this work, we adopt the following convention:
\begin{equation*}
	\|\cdot\|_{W_{0}^{s+\log,p}(\Omega)} := [\cdot]_{s+\log,p}.
	\end{equation*}
In this paper, we denote by $d(\cdot)$ the distance function to the boundary $\partial \Omega$, defined for any $x \in \Omega$ by
	\[
	d(x) := \operatorname{dist}(x,\partial \Omega) = \inf_{y \in \partial \Omega} |x - y|.
	\]
	We also consider the H\"older space $C^{0,\beta}(\overline{\Omega})$, endowed with the norm
	\[
	\|u\|_{C^{0,\beta}(\overline{\Omega})} := \|u\|_{L^\infty(\Omega)} + [u]_{0,\beta},
	\]
	where the H\"older seminorm is defined by
	\begin{equation*}
	[u]_{0,\beta} = \sup_{\substack{x, y \in \Omega \\ x \neq y}} \frac{|u(x) - u(y)|}{|x - y|^\beta}.
	\end{equation*}

We now establish the fundamental topological and geometrical properties of these spaces.
\begin{proposition}
	Let $1 < p < \infty$ and $0 < s < 1$. The fractional logarithmic Sobolev space $W^{s+\log, p}(\mathbb{R}^N)$, equipped with the norm $\|\cdot\|_{W^{s+\log, p}(\mathbb{R}^N)}$, is a separable, uniformly convex, and reflexive Banach space. Consequently, the closed subspace $W_0^{s+\log, p}(\Omega)$ inherits the same geometric properties.
\end{proposition}

\begin{proof}
Let $d\mu(x,y) := \mathbf{k}^{s+\log,+}_{p}(x-y) dx dy$ be the measure defined on \(\mathbb{R}^{2N}\) associated with the positive logarithmic kernel. We consider the product space $ E := L^p(\mathbb{R}^N) \times L^p(\mathbb{R}^{2N}, d\mu), $ endowed with the natural product norm defined by
\[
\|(v,w)\|_{E}
:= \left( \|v\|_{L^p(\mathbb{R}^N)}^{p}
+ \|w\|_{L^p(\mathbb{R}^{2N}, d\mu)}^{p}\right) ^{\frac{1}{p}}.
\]
Since $1 < p < \infty$, both spaces $L^p(\mathbb{R}^N)$ and $L^p(\mathbb{R}^{2N}, d\mu)$ are separable and uniformly convex Banach spaces (see Brezis \cite[Chapter 4]{Brezis-Book}). Consequently, their Cartesian product $E$ inherits separability and uniform convexity. Therefore, by the Milman-Pettis theorem (see Milman~\cite{Milman} and Pettis~\cite{Pettis}), it follows that $E$ is reflexive.\\[4pt]
Now, consider the linear operator $T: W^{s+\log, p}(\mathbb{R}^N) \to E$ defined by
\begin{equation*}
Tu(x,y) := \big(u(x), u(x)-u(y)\big).
\end{equation*}
By the definition of the corresponding norms, we obtain
\begin{equation*}
\|Tu\|_E
= \left( \|u\|_{L^p(\mathbb{R}^N)}^p
+ \iint_{\mathbb{R}^{2N}} |u(x)-u(y)|^p d\mu(x,y)\right) ^{\frac{1}{p}}
= \|u\|_{W^{s+\log, p}(\mathbb{R}^N)}.
\end{equation*}
Hence, $T$ is a linear isometry. Consequently, the space $W^{s+\log, p}(\mathbb{R}^N)$ is isometrically isomorphic to its image $T\big(W^{s+\log, p}(\mathbb{R}^N)\big)\subset E$. To conclude that $W^{s+\log, p}(\mathbb{R}^N)$ inherits the completeness and geometric properties of $E$, it is sufficient to show that $T\big(W^{s+\log, p}(\mathbb{R}^N)\big)$ is a closed subspace of $E$.  Let $ (u_n) $ be a Cauchy sequence in $W^{s+\log, p}(\mathbb{R}^N)$ such that $Tu_n \to (v,w)$ strongly in $E$. Then we obtain
\begin{align}
u_n &\to v \quad \text{strongly in } L^p(\mathbb{R}^N), \label{LpConvergence} \\
u_n(x)-u_n(y) &\to w(x,y) \quad \text{strongly in } L^p(\mathbb{R}^{2N}, d\mu). \label{ModularConvergence}
\end{align}
From \eqref{LpConvergence}, we can extract a subsequence (still denoted by $(u_{n_k})$) such that $u_{n_k} \to v$ a.e. in $\mathbb{R}^N$. Hence,
\begin{equation*}
u_{n_k}(x) - u_{n_k}(y) \to v(x) - v(y) \quad \text{a.e. in } \mathbb{R}^{2N}.
\end{equation*}
On the other hand, \eqref{ModularConvergence} ensures the existence of a (possibly further) subsequence converging pointwise almost everywhere in $\mathbb{R}^{2N}$ (with respect to the measure $d\mu$) to $w(x,y)$. By the uniqueness of the pointwise limit, it follows that
\[
w(x,y) = v(x) - v(y) \quad \text{for almost every } (x,y) \in \mathbb{R}^{2N}.
\]
Hence, $(v,w) = Tv$, which shows that the image $T\big(W^{s+\log, p}(\mathbb{R}^N)\big)$ is closed in $E$.
We therefore conclude that $W^{s+\log, p}(\mathbb{R}^N)$ is a separable, uniformly convex, and reflexive Banach space. Finally, since $W_0^{s+\log, p}(\Omega)$ is defined as a closed linear subspace of $W^{s+\log, p}(\mathbb{R}^N)$, it inherits completeness, separability, uniform convexity, and reflexivity.
\end{proof}
We are now ready to present a detailed weak formulation for problems involving the fractional logarithmic $p$-Laplacian. To this end, let $u, v \in C_c^{\infty}(\Omega)$, and we define
\begin{equation}
\label{equ15}
\begin{aligned}
\mathcal{J}_{s+\log, p}(u, v):= \dfrac{p}{2}\left( \mathcal{J}_{+}(u, v) - \mathcal{J}_{-}(u, v)\right) + \frac{B(N, s, p) C(N, s, p)}{2} \mathcal{J}_{s}(u, v),
\end{aligned}
\end{equation}
where
\begin{equation*}
\begin{aligned}
\mathcal{J}_{\pm}(u, v)= \iint_{\mathbb{R}^{2N}} |u(x) - u(y)|^{p-2} (u(x) - u(y)) (v(x) - v(y))  \mathbf{k}^{s+\log, \pm}_{p}(x - y) dx dy,
\end{aligned}
\end{equation*}
and
\begin{equation*}
\begin{aligned}
\mathcal{J}_{s}(u, v) = \iint_{\mathbb{R}^{2N}} \frac{|u(x) - u(y)|^{p-2} (u(x) - u(y)) (v(x) - v(y))}{|x - y|^{N + sp}} dx dy.
\end{aligned}
\end{equation*}
\noindent
It is worth noting that the full energy functional $\mathcal{J}_{s+\log, p}$ is, in general, not positive definite. Furthermore, the functionals $\mathcal{J}_{s+\log, p}$, $\mathcal{J}_{\pm}$, and $\mathcal{J}_{s}$ possess the following fundamental properties.

	\begin{proposition}\label{proposition1}
		Let $1 < p < \infty$, $0 < s < 1$, and let $u \in W_{0}^{s+\log,p}(\Omega)$. Then the following assertions hold:
		\begin{enumerate}
			
			\item The following estimate holds:
			\begin{equation*}
			0 \leq \mathcal{J}_{-}(u, u) 
			\leq \frac{2^{p} \omega_N}{(s p)^2} C(N, s, p) \| u \|_{L^{p}(\Omega)}^p.
			\end{equation*}
			
			\item Define
			\[
			\mathrm{diam}(\Omega) := \sup \{ |x - y| : x, y \in \Omega \}.
			\]
			If $\mathrm{diam}(\Omega) < 1$, then
			\begin{equation*}
			\mathcal{J}_{+}(u, u) - \mathcal{J}_{-}(u, u) 
			\geq -\frac{2  \omega_N}{s p} C(N, s, p) (\mathrm{diam}(\Omega))^{-s p} 
			\left( \ln (\mathrm{diam}(\Omega)) + \frac{1}{s p} \right) \| u \|_{L^{p}(\Omega)}^p.
			\end{equation*}
			In particular, if $\mathrm{diam}(\Omega) \leq e^{-\frac{1}{s p}}$, then
			\begin{equation*}
			\mathcal{J}_{+}(u, u) - \mathcal{J}_{-}(u, u) \geq 0.
			\end{equation*}
			
			\item Let $0 < r < 1$. Then
			\begin{equation*}
			\begin{aligned}
			\mathcal{J}_{s}(u, u)  
			& \leq - \frac{1}{C(N,s,p) \ln r} \|u\|^{p}_{W_{0}^{s+\log,p}(\Omega)} 
			+ \frac{2^{p} \omega_N}{sp} r^{-sp} \|u\|^{p}_{L^{p}(\Omega)}.
			\end{aligned}
			\end{equation*}
			
			\item Assume that $\mathrm{diam}(\Omega) \leq e^{-\frac{1}{sp}}$ and $B(N,s,p) \geq 0$. Then the functional $\mathcal{J}_{s+\log,p}(u,u)$ is nonnegative definite on $W_{0}^{s+\log,p}(\Omega)$, and
			\[
			\mathcal{J}_{s+\log,p}(u,u) = 0 \quad \text{if and only if} \quad u = 0.
			\]
			Moreover, $|u| \in W_{0}^{s+\log,p}(\Omega)$, and
			\begin{equation*}
			\mathcal{J}_{s+\log,p}(|u|,|u|) \leq \mathcal{J}_{s+\log,p}(u,u).
			\end{equation*}
			Equality holds if and only if $u$ does not change sign.
			
		\end{enumerate}
\end{proposition}

	\begin{proof}
		Let $u \in W_{0}^{s+\log,p}(\Omega)$. Then, we have the following:
		
		\medskip
		
		\noindent\text{(1)} By exploiting the convexity of the function $\tau \mapsto \tau^{p}$ together with the symmetry of the integrand, we get
		\begin{equation*}
		\begin{aligned}
		\mathcal{J}_{-}(u, u) 
		&= C(N, s, p) \iint_{\mathbb{R}^{2N}} 
		\frac{|u(x) - u(y)|^{p}  \left(- \ln |x - y|\right)_{-}}{|x - y|^{N + sp}}  dx  dy \\
		&\leq 2^{p} C(N, s, p) \int_{\mathbb{R}^{N}} |u(x)|^{p} 
		\left( \int_{\mathbb{R}^{N}} 
		\frac{\left(- \ln |x - y|\right)_{-}}{|x - y|^{N + sp}}  dy \right) dx.
		\end{aligned}
		\end{equation*}
		Computing the inner integral yields
		\begin{equation*}
		\int_{\mathbb{R}^{N}} 
		\frac{\left(- \ln |x - y|\right)_{-}}{|x - y|^{N + sp}}  dy = \frac{\omega_N}{(sp)^2}.
		\end{equation*}
		Consequently, we obtain the estimate
		\begin{equation*}
		\mathcal{J}_{-}(u, u) 
		\leq \frac{2^{p} \omega_N}{(sp)^2}  C(N, s, p) \left\| u\right\| ^{p}_{L^{p}(\Omega)}.
		\end{equation*}
		\text{(2)} Since $u \in W_{0}^{s+\log,p}(\Omega)$, by the definitions of $\mathcal{J}_{\pm}$ we have
		\begin{equation*}
		\begin{aligned}
		\mathcal{J}_{+}(u, u) - \mathcal{J}_{-}(u, u) 
		&= - C(N, s, p) \iint_{\mathbb{R}^{2N}} 
		\frac{|u(x) - u(y)|^{p}  \ln |x - y|}{|x - y|^{N + sp}}  dx  dy \\[2mm]
		&= \underbrace{- C(N, s, p)\iint_{\Omega \times \Omega}
			\frac{|u(x) - u(y)|^{p}  \ln |x - y|}{|x - y|^{N + sp}}  dx  dy}_{\textbf{I}_1} \\[1mm]
		&\quad \underbrace{- 2 C(N, s, p) \int_{\Omega} |u(x)|^{p} \left( \int_{\mathbb{R}^{N} \setminus \Omega} 
			\frac{\ln |x - y|}{|x - y|^{N + sp}}  dy \right) dx}_{\textbf{I}_2}.
		\end{aligned}
		\end{equation*}
		
		\noindent
		Clearly, $\boldsymbol{I}_1$ is non-negative, since $\mathrm{diam}(\Omega) < 1$. We now turn to estimating the term $\boldsymbol{I_2}$. Indeed, for a fixed $x \in \Omega$, the inner integral can be decomposed as
		\begin{equation*}
		\begin{aligned}
		- \int_{\mathbb{R}^{N} \setminus \Omega} 
		\frac{\ln |x - y|}{|x - y|^{N + sp}}  dy 
		&= - \int_{(\mathbb{R}^{N} \setminus \Omega) \cap \mathcal{B}_{1}(x)} 
		\frac{\ln |x - y|}{|x - y|^{N + sp}}  dy 
		- \int_{(\mathbb{R}^{N} \setminus \Omega) \cap (\mathbb{R}^{N} \setminus \mathcal{B}_{1}(x))} 
		\frac{\ln |x - y|}{|x - y|^{N + sp}}  dy\\
		&\geq - \int_{(\mathbb{R}^{N} \setminus \mathcal{B}_{\mathrm{diam}(\Omega)}(x)) \cap \mathcal{B}_{1}(x)} 
		\frac{\ln |x - y|}{|x - y|^{N + sp}}  dy 
		- \int_{\mathbb{R}^{N} \setminus \mathcal{B}_{1}(x)} 
		\frac{\ln |x - y|}{|x - y|^{N + sp}}  dy\\
		&= -\frac{\omega_N}{sp}  (\mathrm{diam}(\Omega))^{-sp} \left( \ln (\mathrm{diam}(\Omega)) + \frac{1}{sp} \right).
		\end{aligned}
		\end{equation*}
		
		\noindent
		Consequently, we obtain
		\begin{equation*}
		\boldsymbol{I}_2 \geq -\frac{2  \omega_N}{sp} C(N, s, p) (\mathrm{diam}(\Omega))^{-sp} 
		\left( \ln (\mathrm{diam}(\Omega)) + \frac{1}{sp} \right) \| u\|_{L^{p}(\Omega)}^p,
		\end{equation*}
		which implies
		\begin{equation*}
		\mathcal{J}_{+}(u, u) - \mathcal{J}_{-}(u, u) 
		\geq -\frac{2  \omega_N}{sp} C(N, s, p) (\mathrm{diam}(\Omega))^{-sp} 
		\left( \ln (\mathrm{diam}(\Omega)) + \frac{1}{sp} \right) \| u\|_{L^{p}(\Omega)}^p.
		\end{equation*}
		It is easy to verify that the preceding estimate is non-negative whenever  $ \mathrm{diam}(\Omega) \leq e^{-\frac{1}{sp}}. $\\[4pt]
		\text{(3)} We follow the same argument as in \cite[Lemma 3.1]{Chen-Chen-Hauer}. In fact, we obtain the following decomposition:
		\begin{equation*}
		\begin{aligned}
		\mathcal{J}_{s}(u, u)  
		&= \underbrace{\iint_{\{|x-y| < r\}} 
			\frac{|u(x) - u(y)|^{p}}
			{|x - y|^{N + sp}} dx dy}_{\mathbf{I}_{1}}  
		+ \underbrace{\iint_{\{|x-y| \geq r\}} 
			\frac{|u(x) - u(y)|^{p}}
			{|x - y|^{N + sp}} dx dy}_{\mathbf{I}_{2}}.
		\end{aligned}
		\end{equation*}
		
		\noindent
		\textbf{Estimate of \(\mathbf{I}_{1}\).}  
		Observe that for \(|x-y| < r < 1\), one has
		\[
		(-\ln |x-y|)_{+} = -\ln |x-y| \geq -\ln r.
		\]
		Hence, it follows that
		\begin{equation*}
		\begin{aligned}
		\mathbf{I}_{1} 
		&\leq - \frac{1}{\ln r}
		\iint_{\{|x-y| < r\}} 
		|u(x) - u(y)|^{p} 
		\frac{(-\ln |x-y|)_{+}}{|x - y|^{N + sp}} dx dy \\
		&= - \frac{1}{C(N,s,p) \ln r}
		\iint_{\{|x-y| < r\}} 
		|u(x) - u(y)|^{p} 
		\mathbf{k}^{s+\log,+}_{p}(x-y) dx dy \\
		&= - \frac{1}{C(N,s,p) \ln r} [u]^{p}_{s+\log,p}.
		\end{aligned}
		\end{equation*}
		
		\noindent
		\textbf{Estimate of \(\mathbf{I}_{2}\).}  
		Using once again the convexity of the mapping \(\tau \mapsto \tau^{p}\) together with the symmetry, we get
		\begin{equation*}
		\begin{aligned}
		\mathbf{I}_{2} 
		&\leq 2^{p-1} \iint_{\{|x-y| \geq r\}} 
		\frac{|u(x)|^{p} + |u(y)|^{p}}
		{|x - y|^{N + sp}} dx dy \\
		&= 2^{p} \int_{\mathbb{R}^{N}} |u(x)|^{p}
		\left( \int_{\{|x-y| \geq r\}} \frac{dy}{|x - y|^{N+sp}} \right) dx \\
		&= \frac{2^{p} \omega_N}{sp} r^{-sp} \|u\|^{p}_{L^{p}(\Omega)}.
		\end{aligned}
		\end{equation*}
		Consequently, we conclude that
		\begin{equation*}
		\begin{aligned}
		\mathcal{J}_{s}(u, u)  
		& \leq - \frac{1}{C(N,s,p) \ln r} [u]^{p}_{s+\log,p}
		+ \frac{2^{p} \omega_N}{sp} r^{-sp} \|u\|^{p}_{L^{p}(\Omega)}.
		\end{aligned}
		\end{equation*}
		\text{(4)} Observing that $\mathrm{diam}(\Omega) \leq e^{-\frac{1}{sp}}$, it follows directly from \text{(2)} that  $ \mathcal{J}_{+}(u, u) - \mathcal{J}_{-}(u, u) > 0. $ Moreover, recalling the definition of the full fractional energy and noting that $B(N, s, p) \geq 0$, we deduce
		\[
		\mathcal{J}_{s+\log, p}(u) \geq 0.
		\] 
		Assuming that $\mathcal{J}_{s+\log, p}(u, u) = 0$, we then infer that
		\begin{equation*}
		\mathcal{J}_{s}(u, u)
		= \iint_{\mathbb{R}^{2N}}
		\frac{|u(x) - u(y)|^{p}}{|x - y|^{N + sp}} dx dy = 0.
		\end{equation*}
		Since $u \in W_{0}^{s+\log,p}(\Omega)$, it immediately follows that $u = 0$ in $\Omega$. Moreover, we have
		\begin{equation*}
		\begin{aligned}
		\mathcal{J}_{+}(u, u) - \mathcal{J}_{-}(u, u) 
		&= - C(N, s, p) \iint_{\mathbb{R}^{2N}} 
		\frac{|u(x) - u(y)|^{p}  \ln |x - y|}{|x - y|^{N + sp}}  dx  dy \\[4pt]
		&= - C(N, s, p) \iint_{\Omega \times \Omega} 
		\frac{|u(x) - u(y)|^{p}  \ln |x - y|}{|x - y|^{N + sp}}  dx  dy \\[4pt]
		&\quad - 2 C(N, s, p) \int_{\Omega} |u(x)|^{p} \left( \int_{\mathbb{R}^{N} \setminus \Omega} 
		\frac{\ln |x - y|}{|x - y|^{N + sp}}  dy \right) dx \\[4pt]
		&\geq -C(N, s, p) \iint_{\Omega \times \Omega} 
		\frac{||u(x)| - |u(y)||^{p}  \ln |x - y|}{|x - y|^{N + sp}}  dx  dy \\[4pt]
		&\quad - 2 C(N, s, p) \int_{\Omega} |u(x)|^{p} \left( \int_{\mathbb{R}^{N} \setminus \Omega} 
		\frac{\ln |x - y|}{|x - y|^{N + sp}}  dy \right) dx \\[4pt]
		&= - C(N, s, p) \iint_{\mathbb{R}^{N} \times \mathbb{R}^{N}} 
		\frac{||u(x)| - |u(y)||^{p}  \ln |x - y|}{|x - y|^{N + sp}}  dx  dy \\[4pt]
		&= \mathcal{J}_{+}(|u|, |u|) - \mathcal{J}_{-}(|u|, |u|).
		\end{aligned}
		\end{equation*}
		Consequently, it follows that
		\begin{equation*}
		\begin{aligned}
		0 \leq \mathcal{J}_{s+\log, p}(|u|, |u|) 
		&= \dfrac{p}{2}\left(\mathcal{J}_{+}(|u|, |u|) - \mathcal{J}_{-}(|u|, |u|) \right) 
		+ \frac{B(N, s, p) C(N, s, p)}{2} \mathcal{J}_{s}(|u|, |u|) \\[4pt]
		&\leq \dfrac{p}{2}\left( \mathcal{J}_{+}(u, u) - \mathcal{J}_{-}(u, u)\right)  
		+ \frac{B(N, s, p) C(N, s, p)}{2} \mathcal{J}_{s}(u, u) =  \mathcal{J}_{s+\log, p}(u, u).
		\end{aligned}
		\end{equation*}
		We remark that the last equality holds if and only if $u$ does not change sign.
\end{proof}

\begin{remark}
		In view of Proposition \ref{proposition1} (1) and (3), the energy functional $ \mathcal{J}_{s+\log, p}(u, u) $ associated with the fractional logarithmic $p$-Laplacian is well-defined on $W_{0}^{s+\log,p}(\Omega)$.
\end{remark}

\begin{remark}
    The geometric condition $\mathrm{diam}(\Omega) < e^{-1/sp}$ is intrinsically linked to the compactness of the variational framework. For $\mathrm{diam}(\Omega) \geq e^{-1/sp}$, the kernel $\mathbf{K}_p^{s+\log}(|x-y|)$ changes sign on $\Omega \times \Omega$. Consequently, the energy functional $\mathcal{J}_{s+\log, p}(u,u) = \mathcal{J}_+(u,u) - \mathcal{J}_-(u,u)$ becomes strongly indefinite, and standard coercivity fails. This loss of coercivity permits the existence of a Palais-Smale (PS) sequence $(u_n)_{n \in \mathbb{N}}$ satisfying
    \begin{equation*}
        \|u_n\|_{W_0^{s+\log, p}(\Omega)} \to \infty \quad \text{while} \quad \mathcal{J}_{s+\log, p}(u_n, u_n) \to c,
    \end{equation*}
    driven by exact cancellations between the positive and negative energy terms. Since the standard PS condition fails, establishing existence and multiplicity results on large domains necessitates advanced critical point theory for strongly indefinite functionals (e.g., linking theorems or the cohomological index), which remains an open problem.
\end{remark}

We show that the space $W_{0}^{s+\log,p}(\Omega)$ can be viewed as an intermediate space between $W_{0}^{s+\varepsilon,p}(\Omega)$ and $W_{0}^{s,p}(\Omega)$ for every $0 < \varepsilon < 1 - s$. More precisely, we have
\begin{proposition}
	\label{proposition4}
	Let $1 < p < \infty$ and $0 < s < 1$. Then, for every $0 < \varepsilon < 1 - s$, the following continuous embeddings hold:
	\begin{equation*}
	W_{0}^{s+\varepsilon,p}(\Omega) \hookrightarrow W_{0}^{s+\log,p}(\Omega) \hookrightarrow W_{0}^{s,p}(\Omega).
	\end{equation*}
\end{proposition}
\begin{proof}
On the one hand, it follows from Proposition~\ref{proposition1} (3), together with the Poincar\'e inequality (see Proposition~\ref{proposition3} below), that $ W^{s+\log}_{0}(\Omega) \hookrightarrow W^{s,p}_{0}(\Omega). $ On the other hand, we argue in the same spirit as in~\cite[Proposition~1.6]{Chen-Chen-Hauer}. For any $0 < \epsilon < 1 - s$, there exists $0 < r_{\epsilon} < e^{-\frac{1}{ps}}$ such that
\[
|x - y|^{-N - sp} \bigl( - p \ln |x - y| \bigr)_{+}
\leq
|x - y|^{- N - p(s + \epsilon)}
\quad \text{for all } 0 < |x - y| < r_{\epsilon}.
\]
Consequently, a direct computation yields
\[
\iint_{\{ |x - y| < r_{\epsilon} \}}
\dfrac{|u(x) - u(y)|^{p}}{|x - y|^{N+sp}}
\bigl( - p \ln |x - y| \bigr)_{+}  dx  dy
\leq
\iint_{\{ |x - y| < r_{\epsilon} \}}
\dfrac{|u(x) - u(y)|^{p}}{|x - y|^{N + p(s+\epsilon)}}  dx  dy
\leq
\|u\|^{p}_{W^{s+\epsilon}_{0}(\Omega)}.
\]
Now, by exploiting the convexity of the mapping $\tau \mapsto \tau^{p}$, together with symmetry, we obtain
\begin{align*}
\iint_{\{ |x - y| > r_{\epsilon} \}}
\dfrac{|u(x) - u(y)|^{p}}{|x - y|^{N+sp}}
\bigl( - p \ln |x - y| \bigr)_{+}  dx  dy
& \leq
 2^{p} \int_{\mathbb{R}^{N}} |u(x)|^{p}
\left(
\int_{\mathbb{R}^{N} \setminus \mathcal{B}_{r_{\epsilon}}(x)}
\dfrac{\bigl( - \ln |x - y| \bigr)_{+}}{|x - y|^{N+sp}}  dy
\right) dx \\[4pt]
& =2^{p} 
\omega_{N} \left[
\frac{r_{\epsilon}^{-sp}}{sp} (-\ln r_{\epsilon})
+ \frac{1 - r_{\epsilon}^{-sp}}{(sp)^2}
\right] \|u\|_{L^{p}(\Omega)}^{p}.
\end{align*}
Combining the above estimates, we deduce the existence of a constant $C_{\epsilon, p, s} > 0$ such that
\begin{align*}
\|u\|^{p}_{W_{0}^{s+\log,p}(\Omega)}
&= \dfrac{C(N, s, p)}{p}
\iint_{\{ |x - y| < r_{\epsilon} \}}
\dfrac{|u(x) - u(y)|^{p}}{|x - y|^{N+sp}}
\bigl( - p \ln |x - y| \bigr)_{+}  dx  dy \\[4pt]
&\quad + C(N, s, p) 
\iint_{\{ |x - y| > r_{\epsilon} \}}
\dfrac{|u(x) - u(y)|^{p}}{|x - y|^{N+sp}}
\bigl( - \ln |x - y| \bigr)_{+}  dx  dy \leq C_{\epsilon, p, s} \|u\|^{p}_{W^{s+\epsilon, p}_{0}(\Omega)}.
\end{align*}
In particular, for any $0 < \epsilon < 1 - s$, it follows that $ W^{s+\epsilon, p}_{0}(\Omega) \hookrightarrow W^{s+\log}_{0}(\Omega). $
\end{proof}

\section{Fundamental functional inequalities and the Pohozaev identity}
\label{section3}
In this section, we establish several fundamental functional inequalities, namely a Poincar\'e-type inequality, a boundary Hardy-type inequality, a D\'{\i}az-Saa inequality, and a Pohozaev identity. These results constitute the main analytical framework for the fractional logarithmic Sobolev space $W_0^{s+\log, p}(\Omega)$.

\subsection{Poincar\'e-type inequality}
We first derive a Poincar\'e inequality associated with the fractional logarithmic $p$-Laplacian under consideration. More precisely, we obtain the following result:

\begin{proposition}
	\label{proposition3}
	Let $1 < p < \infty$ and $0 < s < 1$. Assume that $u \in W^{s+\log,p}_{0}(\Omega)$. Then the estimate holds:
	\begin{equation}
	\label{equ13}
	\|u\|_{L^{p}(\Omega)}^{p}
	\leq C \iint_{\mathbb{R}^{2N}} |u(x)-u(y)|^{p} \mathbf{k}^{s+\log,+}_{p}(x-y) dx dy,
	\end{equation}
	where $C = C(s,p,N,\Omega)$ is a positive constant given by
	\[
	C =
	\frac{N\big(\operatorname{diam}(\Omega)\big)^{N+sp}}
	{C(N,s,p) \omega_N a^N\left(1- N  \ln a\right)},
	\]
with $a=\min\big(\mathrm{diam}(\Omega), 1\big)$, where $C(N,s,p)$ denotes the fractional normalization constant defined in \eqref{equ4}.
\end{proposition}

\begin{proof}
Let $x \in \Omega$ be fixed, and let $y \in \mathcal{B}_{\mathrm{diam}(\Omega)}(x)$ be such that $u(y)=0$. Consequently, we obtain
		\begin{align*}
		C(N,s,p) |u(x)|^{p} (-\ln|x-y|)_{+}
		&= |u(x)-u(y)|^{p} 
		\frac{C(N,s,p) (-\ln|x-y|)_{+}}{|x-y|^{N+sp}} 
		|x-y|^{N+sp} \\
		&\leq (\mathrm{diam}(\Omega))^{N + sp} |u(x)-u(y)|^{p} 
		\frac{C(N,s,p) (-\ln|x-y|)_{+}}{|x-y|^{N+sp}}.
		\end{align*}
		Integrating with respect to $y$ over $\mathcal{B}_{\mathrm{diam}(\Omega)}(x)$ yields
		\begin{align*}
		C(N,s,p) &\omega_{N} a^N\left(-\ln a+\frac{1}{N}\right) |u(x)|^{p}\\
		&\leq (\mathrm{diam}(\Omega))^{N + sp}
		\int_{\mathcal{B}_{\mathrm{diam}(\Omega)}(x)}
		|u(x)-u(y)|^{p} 
		\frac{C(N,s,p) (-\ln|x-y|)_{+}}{|x-y|^{N+sp}} dy,
		\end{align*}
		where $a=\min\big(\mathrm{diam}(\Omega), 1\big)$. Finally, integrating over $\Omega$ with respect to $x$, we deduce
		\begin{align*}
		\int_{\Omega} |u(x)|^{p} dx
		\leq C(s,p,N,\Omega)
		\int_{\Omega}
		\int_{\mathcal{B}_{\mathrm{diam}(\Omega)}(x)}
		|u(x)-u(y)|^{p} 
		\frac{C(N,s,p) (-\ln|x-y|)_{+}}{|x-y|^{N+sp}} dy dx,
		\end{align*}
		where
		\[
		C(s,p,N,\Omega)
		=
		\frac{N \big(\operatorname{diam}(\Omega)\big)^{N+sp}}
		{C(N,s,p) \omega_N a^N\left(1- N  \ln a\right)}.
		\]
		This completes the proof.
	\end{proof}

As a consequence of Proposition~\ref{proposition3} (Poincar\'e-type inequality), we establish a fractional Sobolev inequality within our functional framework. The presence of the logarithmic singularity at the origin implies that the fractional logarithmic seminorm yields a stronger control than the classical fractional Gagliardo seminorm (as shown in Proposition \ref{proposition4}). Consequently, we obtain the following critical Sobolev-type inequality.

\begin{theorem}
	\label{TheoremSobolevInequality}
	Let $1 < p < \infty$ and $0 < s < 1$ satisfy $  N > sp $, and let $\Omega \subset \mathbb{R}^N$ be a bounded open set. Then, there exists a constant $C = C(N,s,p,\Omega) > 0$ such that for every $u \in W_0^{s+\log, p}(\Omega)$, the following critical Sobolev-type inequality holds:
	\begin{equation}
	\label{SobolevInequality}
	\|u\|_{L^{p_s^*}(\Omega)}^p \leq C \iint_{\mathbb{R}^{2N}} |u(x)-u(y)|^{p} \mathbf{k}^{s+\log,+}_{p}(x-y) dx dy.
	\end{equation}
\end{theorem}

\begin{proof}
	By means of the classical fractional Sobolev inequality (see Remark \ref{remark1}), there exists a positive constant $C_* = C_*(N,s,p) > 0$ such that the fractional Gagliardo seminorm $[\cdot]_{s,p}$ controls the critical Lebesgue norm, in the sense that
		\begin{equation}
		\label{ClassicalSobolev}
		\|u\|_{L^{p_s^*}(\mathbb{R}^N)}^p \leq C_*  [u]_{s,p}^p := C_*  \iint_{\mathbb{R}^{2N}} \frac{|u(x)-u(y)|^p}{|x-y|^{N+sp}}  dx dy.
		\end{equation}
		On the other hand, it follows from Proposition~\ref{proposition1} (3) that, for any $0 < r < 1$, we have
		\begin{equation*}
		\begin{aligned}
		\iint_{\mathbb{R}^{2N}} \frac{|u(x)-u(y)|^p}{|x-y|^{N+sp}}  dx dy 
		& \leq - \frac{1}{C(N,s,p) \ln r} \|u\|^{p}_{W_{0}^{s+\log,p}(\Omega)}  + \frac{2^{p} \omega_N}{sp} r^{-sp} \|u\|^{p}_{L^{p}(\Omega)}.
		\end{aligned}
		\end{equation*}
		Combining this estimate with the Poincar\'e inequality (see Proposition~\ref{proposition3}), we deduce that
		\begin{equation}\label{equ40}
		\begin{aligned}
		\iint_{\mathbb{R}^{2N}} \frac{|u(x)-u(y)|^p}{|x-y|^{N+sp}}  dx dy 
		& \leq \left( - \frac{1}{C(N,s,p) \ln r}
		+ \frac{2^{p} \omega_N}{sp} r^{-sp} C(s,p,N,\Omega)\right)
		\|u\|^{p}_{W_{0}^{s+\log,p}(\Omega)} .
		\end{aligned}
		\end{equation}
		Finally, coupling \eqref{equ40} with \eqref{ClassicalSobolev}, we obtain the desired estimate. This completes the proof.
\end{proof}

\subsection{Hardy-type inequality}

In this subsection, we establish a logarithmic boundary Hardy inequality. Indeed, we have the following result.

\begin{theorem}
	\label{FractionalLogarithmicHardyInequality}
	Let $1 < p < \infty$ and $0 < s < 1$, and let $\Omega \subset \mathbb{R}^N$ be a bounded Lipschitz domain. Then there exists a constant $C = C(N,s,p,\Omega) > 0$ such that, for every $u \in W_0^{s+\log,p}(\Omega)$, the following inequality holds:
	\begin{equation*}
	\int_\Omega \frac{|u(x)|^p}{d(x)^{sp}} 
	\left( \ln \left(\frac{1}{d(x)}\right) \right)_{+}  dx
	\leq C \iint_{\mathbb{R}^{2N}} |u(x)-u(y)|^{p} 
	\mathbf{k}^{s+\log,+}_{p}(x-y) dx dy.
	\end{equation*}
\end{theorem}

\begin{proof}
	For any $u \in W_0^{s+\log, p}(\Omega)$, it holds that $u = 0$ almost everywhere in $\Omega^c := \mathbb{R}^N \setminus \Omega$. By decomposing the fractional logarithmic Gagliardo seminorm $[\cdot]_{s+\log, p}$, we obtain that
	\begin{equation}
	\label{HardySplit}
	[u]_{s+\log, p}^p 
	\geq 2 \int_\Omega |u(x)|^p \left( \int_{\Omega^c} \mathbf{k}^{s+\log,+}_{p}(x-y) dy \right) dx 
	=: 2 \int_\Omega |u(x)|^p \kappa_\Omega(x) dx,
	\end{equation}
	where $\kappa_\Omega(x)$ denotes the regional killing measure. In order to derive the desired inequality, it is sufficient to establish a suitable lower bound for $\kappa_\Omega(x)$ in a neighborhood of $\partial \Omega$. Since $\Omega$ has a Lipschitz boundary, it satisfies the uniform exterior cone condition. Consequently, there exist constants $\theta \in (0, \pi/2)$ and $h \in (0, e^{-1/p})$ such that, for every $x_0 \in \partial \Omega$, one can associate an exterior cone $\mathcal{C}_{x_0} \subset \Omega^c$ with vertex at $x_0$, height $h$, and opening angle $\theta$. Now, let $x \in \Omega$ with $d(x) < h/4$, and choose $x_0 \in \partial\Omega$ such that $|x-x_0| = d(x)$. Considering the truncated cone $\widetilde{\mathcal{C}}_{x_0} := \mathcal{C}_{x_0} \setminus B_{2d(x)}(x_0)$, any $y \in \widetilde{\mathcal{C}}_{x_0}$ satisfies $|y-x_0| \geq 2d(x)$. By the triangle inequality, $\frac{1}{2}|y-x_0| \leq |y-x| \leq \frac{3}{2}|y-x_0|$. Passing to polar coordinates centered at $x$, the angular measure of the intersection is bounded from below, giving
	\begin{equation*}
	\kappa_\Omega(x) \geq \int_{\widetilde{\mathcal{C}}_{x_0}} \mathbf{k}^{s+\log,+}_{p}(x-y) dy \geq c_1 p \int_{3d(x)}^h \frac{-\ln r}{r^{1+sp}} dr,
	\end{equation*}
	where $c_1 > 0$ depends only on $N$ and $\theta$. Evaluating the integral via integration by parts, we obtain
	\begin{align*}
	\int_{3d(x)}^h \frac{-\ln r}{r^{1+sp}} dr &= \left[ \frac{r^{-sp}}{sp} (-\ln r) \right]_{3d(x)}^h - \frac{1}{sp} \int_{3d(x)}^h r^{-1-sp} dr \nonumber \\[4pt]
	&= \frac{(3d(x))^{-sp}}{sp} |\ln(3d(x))| - \frac{(3d(x))^{-sp}}{s^2 p^2} + C(h, s, p).
	\end{align*}
	As $d(x) \to 0^+$, the leading logarithmic term strictly dominates. Consequently, there exist a threshold $d_0 \in (0, h/4)$ and a constant $c_2 > 0$ such that for all $x \in \Omega$ satisfying $d(x) < d_0$,
	\begin{equation}
	\label{KappaLowerBound}
	\kappa_\Omega(x) \geq c_2 \frac{\left( \ln(1/d(x))\right) _{+}}{d(x)^{sp}}.
	\end{equation}
	Defining the boundary strip $\Omega_{d_0} := \{x \in \Omega : d(x) < d_0\}$ and substituting \eqref{KappaLowerBound} into \eqref{HardySplit}, we secure the bound near the boundary
	\begin{equation}
	\label{BoundaryStripBound}
	\int_{\Omega_{d_0}} \frac{|u(x)|^p}{d(x)^{sp}} \left( \ln\left(\frac{1}{d(x)}\right)\right) _{+} dx \leq \frac{1}{2 c_2} [u]_{s+\log, p}^p.
	\end{equation}
	For the interior region $\Omega \setminus \Omega_{d_0}$, the distance to the boundary is bounded strictly away from zero ($d(x) \geq d_0$), ensuring the continuous weight function is uniformly bounded: $d(x)^{-sp} \left(\ln(1/d(x))\right) _{+}\leq M(d_0)$. Moreover, by applying the Poincar\'e inequality (see Proposition~\ref{proposition3}), $\|u\|_{L^p(\Omega)}^p \leq C(s,p,N,\Omega) [u]_{s+\log, p}^p$, we derive
	\begin{equation}
	\label{InteriorBound1}
	\int_{\Omega \setminus \Omega_{d_0}} \frac{|u(x)|^p}{d(x)^{sp}} \left(\ln\left(\frac{1}{d(x)}\right)\right) _{+} dx \leq M(d_0) \|u\|_{L^p(\Omega)}^p \leq M(d_0) C(s,p,N,\Omega)  [u]_{s+\log, p}^p.
	\end{equation}
	Summing \eqref{BoundaryStripBound} and \eqref{InteriorBound1} yields the desired global inequality with $C := \frac{1}{2 c_2} + M(d_0) C(s,p,N,\Omega)$.
\end{proof}

\subsection{D\'{\i}az-Saa type inequality}

In this subsection, we present a D\'{\i}az-Saa type inequality associated with the fractional logarithmic $p$-Laplacian $(-\Delta)_{p}^{s+\log}$, which leads to uniqueness results. Precisely, we have
	
	\begin{lemma}\label{Lemma2}
Let $1 < p < \infty$, $1 < r \leq p$, and $0 < s < 1$. Assume that $\operatorname{diam}(\Omega) < e^{\frac{B(N,s,p)}{p}}$, where $B(N,s,p)$ is defined in \eqref{equB}. Then, the following inequality holds in the sense of distributions:
		\begin{equation}\label{equ17}
		\mathcal{J}_{s+\log, p}\left(u,\frac{u^{r}-v^{r}}{u^{r-1}}\right)
		-
		\mathcal{J}_{s+\log, p}\left(v,\frac{v^{r}-u^{r}}{v^{r-1}}\right)
		\geq 0.
		\end{equation}
		This inequality holds for all $u, v \in W^{s+\log,p}_{0}(\Omega)$ such that $u>0$ and $v>0$ a.e. in $\Omega$, and such that $u/v,\; v/u \in L^{\infty}(\Omega)$. Moreover, if equality holds in \eqref{equ17}, then the following assertions are satisfied:
		\begin{itemize}
			\item[(i)] $ \dfrac{u}{v}\equiv c > 0$ a.e. in $\Omega$;\\
			\item[(ii)] if, in addition, $r \neq p$, then $u \equiv v$ a.e. in $\Omega$.
		\end{itemize}
	\end{lemma}
	
	In order to establish a D\'{\i}az-Saa type inequality, we exploit the convexity properties of the associated energy functional. To this end, let $1 < r \leq p$ and consider the following functional
	\[
	\mathcal{W} : \dot{V}_{+}^{ r} \longrightarrow \mathbb{R}_{+}, 
	\qquad 
	\mathcal{W}(w)
	= \frac{1}{p} \mathcal{J}_{s+\log, p} \left(w^{\frac{1}{r}}, w^{\frac{1}{r}}\right).
	\]
	Here,
	\[
	\dot{V}_{+}^{ r }
	:= \Bigl\{ u : \mathbb{R}^{N} \to (0, \infty) \;\Big|\; u^{\frac{1}{r}} \in W_{0}^{s+\log, p}(\Omega) \Bigr\}.
	\]
We next introduce the following definition.

\begin{definition}\label{definition}
	Let $X$ be a real vector space and let $C \subset X$ be a nonempty convex cone.  
	A functional $\mathcal{W} : C \to \mathbb{R}$ is said to be  ray-strictly convex (respectively,  strictly convex) if
	\[
	\mathcal{W}\big((1-t)u_1 + t u_2\big)
	\leq (1-t)\mathcal{W}(u_1) + t \mathcal{W}(u_2),
	\quad \text{for all } u_1, u_2 \in C,\; t \in (0,1),
	\]
	with strict inequality unless $u_1/u_2 \equiv c > 0$ (respectively, unless $u_1 \equiv u_2$).
\end{definition}
According to the above definition, the functional $\mathcal{W}$ satisfies the following convexity property:
\begin{proposition}\label{pro3}
	Assume that $\operatorname{diam}(\Omega) < e^{\frac{B(N,s,p)}{p}}$. Then the functional $\mathcal{W}$ is ray-strictly convex on $\dot{V}_{+}^{ r}$. Moreover, if $r \neq p$, then $\mathcal{W}$ is strictly convex on $\dot{V}_{+}^{ r}$.
\end{proposition}
	\begin{proof}
First, we show that the set $\dot{V}_{+}^{r}$ is a convex cone. Indeed, let $u, v \in \dot{V}_{+}^{r}$ and $\lambda \in [0,1]$. Then we have
\begin{align*}
& \left\|\left( (1-\lambda)u + \lambda v\right)^{\frac{1}{r}}\right\|^{p}_{W_{0}^{s+\log,p}(\Omega)}\\
&= \iint_{\mathbb{R}^{2N}} \left| \left((1-\lambda)u + \lambda v\right)^{\frac{1}{r}}(x)
- \left((1-\lambda)u + \lambda v\right)^{\frac{1}{r}}(y)\right|^{p} \mathbf{k}^{s+\log,+}_{p}(x-y) dx dy \\
&\leq (1-\lambda)\iint_{\mathbb{R}^{2N}} |u^{\frac{1}{r}}(x)-u^{\frac{1}{r}}(y)|^{p} \mathbf{k}^{s+\log,+}_{p}(x-y) dx dy \\
&\quad + \lambda \iint_{\mathbb{R}^{2N}} |v^{\frac{1}{r}}(x)-v^{\frac{1}{r}}(y)|^{p} \mathbf{k}^{s+\log,+}_{p}(x-y) dx dy,
\end{align*}
where in the last inequality we use \cite[Proposition 4.1]{Brasco-Franzina}. Consequently, we deduce that $(1-\lambda)u + \lambda v \in \dot{V}_{+}^{r}$. Now, let $w_1, w_2 \in \dot{V}_{+}^{r}$ and $t \in [0,1]$, and define $w = (1-t)w_1 + t w_2$. Consequently, since $\operatorname{diam}(\Omega) < e^{\frac{B(N,s,p)}{p}}$ and by applying again \cite[Proposition 4.1]{Brasco-Franzina}, we obtain
\begin{align*}
\mathcal{W}(w)
&= \frac{1}{2}\Big( \mathcal{J}_{+}(w^{\frac{1}{r}}, w^{\frac{1}{r}}) - \mathcal{J}_{-}(w^{\frac{1}{r}}, w^{\frac{1}{r}})\Big)
+ \frac{B(N,s,p) C(N,s,p)}{2p} \mathcal{J}_{s}(w^{\frac{1}{r}}, w^{\frac{1}{r}}) \\[4pt]
&= - \frac{C(N,s,p)}{2} \iint_{\mathbb{R}^{2N}}
\frac{\left|w^{\frac{1}{r}}(x) - w^{\frac{1}{r}}(y)\right|^{p}\ln|x-y|}{|x-y|^{N+sp}} dx dy \\[4pt]
&\quad + \frac{B(N,s,p) C(N,s,p)}{2p} \iint_{\mathbb{R}^{2N}}
\frac{\left|w^{\frac{1}{r}}(x) - w^{\frac{1}{r}}(y)\right|^{p}}{|x-y|^{N+sp}} dx dy \\[4pt]
&= - \frac{C(N,s,p)}{2} \iint_{\Omega \times \Omega}
\frac{\left|w^{\frac{1}{r}}(x) - w^{\frac{1}{r}}(y)\right|^{p}\ln|x-y|}{|x-y|^{N+sp}} dx dy \\[4pt]
&\quad + \frac{B(N,s,p) C(N,s,p)}{2p} \iint_{\Omega \times \Omega}
\frac{\left|w^{\frac{1}{r}}(x) - w^{\frac{1}{r}}(y)\right|^{p}}{|x-y|^{N+sp}} dx dy \\[4pt]
&\quad - C(N,s,p)\int_{\Omega} \int_{\mathbb{R}^{N} \setminus \Omega}
\frac{w(x)\ln|x-y|}{|x-y|^{N+sp}} dx dy \\[4pt]
&\quad + \frac{B(N,s,p) C(N,s,p)}{p}\int_{\Omega} \int_{\mathbb{R}^{N} \setminus \Omega}
\frac{w(x)}{|x-y|^{N+sp}} dx dy \\[4pt]
&\leq (1-t) \mathcal{W}(w_1) + t \mathcal{W}(w_2).
\end{align*}
Assume that equality holds. Using the condition $\operatorname{diam}(\Omega) < e^{\frac{B(N,s,p)}{p}}$, we obtain
\begin{equation}\label{equ18}
\big| w(x)^{\frac{1}{r}} - w(y)^{\frac{1}{r}} \big|^{p} 
= (1-t) \big| w_{1}(x)^{\frac{1}{r}} - w_{1}(y)^{\frac{1}{r}} \big|^{p} 
+ t \big| w_{2}(x)^{\frac{1}{r}} - w_{2}(y)^{\frac{1}{r}} \big|^{p}, \quad \text{a.e. } x, y \in \mathbb{R}^{N}.
\end{equation}
We distinguish two cases.\\[4pt]
\textbf{Case 1:} $p = r$. In this case, we obtain
\[
\left| \| a \|_{\boldsymbol{\ell}^{r}} - \| b \|_{\boldsymbol{\ell}^{r}} \right|^{r}
= \| a - b \|_{\boldsymbol{\ell}^{r}}^{r}, \quad \text{for a.e. } x, y \in \mathbb{R}^{N},
\]
where $\| \cdot \|_{\boldsymbol{\ell}^{r}}$ denotes the $\boldsymbol{\ell}^{r}$-norm in $\mathbb{R}^{2}$, and
\[
a = \big( ((1-t)w_{1}(x))^{1/r},  (t w_{2}(x))^{1/r} \big), \quad
b = \big( ((1-t)w_{1}(y))^{1/r},  (t w_{2}(y))^{1/r} \big).
\]
Since $r>1$, it follows that there exists a constant $c>0$ such that $w_{2} = c w_{1}$ almost everywhere in $\mathbb{R}^{N}$. Consequently, the functional $\mathcal{W}$ is ray-strictly convex on $\dot{V}_{+}^{r}$.\\[4pt]
\textbf{Case 2:} $p \neq r$. Using \eqref{equ18} together with \cite[Proposition~4.1]{Brasco-Franzina} and the strict concavity of the mapping $\tau \mapsto \tau^{\frac{r}{p}}$ on $\mathbb{R}^{+}$, we obtain

\begin{equation*}
\begin{aligned}
& (1-t) \left| w_{1}(x)^{1/r} - w_{1}(y)^{1/r} \right|^{r}
+ t \left| w_{2}(x)^{1/r} - w_{2}(y)^{1/r} \right|^{r} \\[4pt]
& \leq 
\left( (1-t) \left| w_{1}(x)^{1/r} - w_{1}(y)^{1/r} \right|^{p}
+ t \left| w_{2}(x)^{1/r} - w_{2}(y)^{1/r} \right|^{p} \right)^{\frac{r}{p}} \\[4pt]
& = \left| w(x)^{1/r} - w(y)^{1/r} \right|^{r} \\[4pt]
& \leq (1-t) \left| w_{1}(x)^{1/r} - w_{1}(y)^{1/r} \right|^{r}
+ t \left| w_{2}(x)^{1/r} - w_{2}(y)^{1/r} \right|^{r}.
\end{aligned}
\end{equation*}
Hence, by Case~1, we deduce that $w_{2} = c w_{1}$ for some constant $c>0$. Substituting this relation into \eqref{equ18}, we obtain
\[
\big( (1-t) + c t \big)^{\frac{p}{r}} 
= (1-t) + t c^{\frac{p}{r}}.
\]
By the strict convexity of the mapping $\tau \mapsto \tau^{\frac{p}{r}}$ on $\mathbb{R}^{+}$, we necessarily have $c=1$. Hence, $w_{1}=w_{2}$ almost everywhere in $\mathbb{R}^{N}$. Therefore, the functional $\mathcal{W}$ is strictly convex on $\dot{V}_{+}^{r}$.
\end{proof}

\begin{proof}[Proof of Lemma \ref{Lemma2}]
	Let \( u, v \in W^{s+\log, p}_0(\Omega) \) be such that \( u > 0 \) and \( v > 0 \) almost everywhere in \( \Omega \), and let \( \theta \in (0,1) \). Define $ w := (1-\theta) u^{r} + \theta v^{r}. $ It is well-known, see \cite[Proposition 4.1]{Brasco-Franzina}, that \( w \in \dot{V}_{+}^{r} \). Hence, by Proposition~\ref{pro3}, the mapping $ 	\theta \mapsto \Phi(\theta) := \mathcal{W}(w) $ is convex and differentiable on \( [0,1] \). For \( \theta \in (0,1) \), a direct computation yields
	\begin{equation*}
	\begin{aligned}
	\Phi'(\theta)
	= \frac{p}{2}\left(
	\mathcal{J}_{+} \left( w , \frac{v^{r} - u^{r}}{w^{r-1}} \right)
	-
	\mathcal{J}_{-} \left( w , \frac{v^{r} - u^{r}}{w^{r-1}} \right)
	\right)  + \frac{B(N,s,p) C(N,s,p)}{2} 
	\mathcal{J}_{s} \left( w , \frac{v^{r} - u^{r}}{w^{r-1}} \right).
	\end{aligned}
	\end{equation*}
	Using the convexity of \( \Phi \) together with the identities \( w = u^{r} \) at \( \theta = 0 \) and \( w = v^{r} \) at \( \theta = 1 \), we deduce that
	\[
	\Phi'(0) = \lim_{\theta \to 0^+} \Phi'(\theta)
	\leq
	\lim_{\theta \to 1^-} \Phi'(\theta)
	= \Phi'(1),
	\]
	which is equivalent to the inequality
	\begin{equation}\label{equ19}
	\mathcal{J}_{s+\log, p} \left(u,\frac{u^{r}-v^{r}}{u^{r-1}}\right)
	-
	\mathcal{J}_{s+\log, p} \left(v,\frac{v^{r}-u^{r}}{v^{r-1}}\right)
	\geq 0.
	\end{equation}
	Finally, assume that equality holds in \eqref{equ19}. Since \( \Phi' : (0,1) \to \mathbb{R} \) is monotone, it follows that \( \Phi'(\theta) \) is constant on \( (0,1) \), and consequently \( \Phi \) is linear on \( [0,1] \), that is,
	\[
	\Phi(\theta)
	= \mathcal{W}(w)
	= (1-\theta)\Phi(0) + \theta \Phi(1)
	= (1-\theta)\mathcal{W}(u^{r}) + \theta \mathcal{W}(v^{r}),
	\quad \forall \theta \in [0,1].
	\]
	This implies that \( u \equiv c v \) for some constant \( c > 0 \). Moreover, if \( r \neq p \), then \( u \equiv v \) by Proposition~\ref{pro3}.
\end{proof}
\begin{remark}
	An alternative proof of Lemma \ref{Lemma2} can be obtained by using the fractional Picone inequality (see \cite[Proposition 4.2]{Brasco-Franzina}), together with Young's inequality. This approach leads directly to the derivation of \eqref{equ17}. Indeed, we have (see \cite[Lemma 1.8]{Giacomoni-Gouasmia-Mokrane})
	\begin{equation} \label{equ20}
	\begin{aligned}
	&\left| u(x) - u(y) \right|^{p-2} \big( u(x) - u(y)\big)
	\left[ \frac{u(x)^r - v(x)^r}{u(x)^{r-1}}
	- \frac{u(y)^r - v(y)^r}{u(y)^{r-1}}\right] \\[4pt]
	&+ \left| v(x) - v(y) \right|^{p-2} \big( v(x) - v(y)\big)
	\left[ \frac{v(x)^r - u(x)^r}{v(x)^{r-1}}
	- \frac{v(y)^r - u(y)^r}{v(y)^{r-1}}\right]
	\geq 0, \quad \text{for a.e. } (x,y)\in \Omega \times \Omega.
	\end{aligned}
	\end{equation}
Since $\operatorname{diam}(\Omega) < e^{\frac{B(N,s,p)}{p}}$, we infer that
\[
C(N,s,p)\left(- p\ln|x-y| + B(N,s,p)\right) > 0 \quad \text{in } \Omega\times \Omega.
\]
Consequently, it follows from \eqref{equ20} that
\begin{equation*}
\begin{aligned}
&- p C(N,s,p) \ln|x-y|
\left| u(x) - u(y) \right|^{p-2} \big( u(x) - u(y)\big)
\left[ \frac{u(x)^r - v(x)^r}{u(x)^{r-1}}
- \frac{u(y)^r - v(y)^r}{u(y)^{r-1}}\right]\\[4pt]
&\quad + C(N,s,p) B(N,s,p)
\left| u(x) - u(y) \right|^{p-2} \big( u(x) - u(y)\big)
\left[ \frac{u(x)^r - v(x)^r}{u(x)^{r-1}}
- \frac{u(y)^r - v(y)^r}{u(y)^{r-1}}\right] \\[4pt]
&\quad - p C(N,s,p) \ln|x-y|
\left| v(x) - v(y) \right|^{p-2} \big( v(x) - v(y)\big)
\left[ \frac{v(x)^r - u(x)^r}{v(x)^{r-1}}
- \frac{v(y)^r - u(y)^r}{v(y)^{r-1}}\right] \\[4pt]
&\quad + C(N,s,p) B(N,s,p)
\left| v(x) - v(y) \right|^{p-2} \big( v(x) - v(y)\big)
\left[ \frac{v(x)^r - u(x)^r}{v(x)^{r-1}}
- \frac{v(y)^r - u(y)^r}{v(y)^{r-1}}\right] \\[4pt]
&\geq 0 \quad \text{for a.e. } (x,y)\in \Omega\times \Omega.
\end{aligned}
\end{equation*}
It is worth noting that this inequality, being a pointwise estimate, is stronger than D\'{\i}az-Saa inequality \eqref{equ17}.
\end{remark}

\subsection{Gagliardo-Nirenberg inequality}
We establish a fractional logarithmic Gagliardo-Nirenberg inequality. This result provides an estimate for intermediate Lebesgue norms in terms of a convex combination of the base $L^{p}$ norm and the full fractional logarithmic norm. In the broader context of nonlinear partial differential equations, this inequality is a fundamental tool for controlling power-type nonlinearities, deriving a priori energy estimates, and studying both the existence and qualitative properties of normalized solutions.

\begin{theorem}\label{FractionalLogarithmicGNInequality}
	Let $1 < p < \infty$ and $0 < s < 1$ satisfy $sp < N$. For any $q \in [p, p_s^*]$, there exists a constant $C = C(N,s,p,q) > 0$ such that, for every $u \in W^{s+\log, p}(\mathbb{R}^N)$, one has
	\begin{equation*}
	\|u\|_{L^q(\mathbb{R}^N)} 
	\leq 
	C  \|u\|_{W^{s+\log, p}(\mathbb{R}^N)}^{\theta} 
	\|u\|_{L^p(\mathbb{R}^N)}^{1-\theta},
	\end{equation*}
	where $\theta \in [0,1]$ is uniquely determined by
	\begin{equation*}
	\frac{1}{q} = \frac{1-\theta}{p} + \frac{\theta}{p_s^*}.
	\end{equation*}
\end{theorem}

\begin{proof}
	We begin by recalling the classical fractional Gagliardo-Nirenberg inequality due to Brezis and Mironescu \cite[Theorem 1]{Brezis-Mironescu}. This result asserts the existence of a positive constant $S>0$ such that, for every $u \in W^{s+\log,p}(\mathbb{R}^N)$-noting that $W^{s+\log,p}(\mathbb{R}^N) \subset W^{s,p}(\mathbb{R}^N)$ (see Proposition \ref{proposition4})-the following estimate holds
		\begin{equation}
		\label{StandardGN}
		\|u\|_{L^q(\mathbb{R}^N)} \leq S [u]_{s,p}^{\theta} \|u\|_{L^p(\mathbb{R}^N)}^{1-\theta}.
		\end{equation}
	On the other hand, as established in the proof of Proposition~\ref{proposition1} (3), we have for every $0<r<1$,
	\begin{equation*}
	\begin{aligned}
	[u]_{s,p}^{p}
	&= \iint_{\mathbb{R}^{2N}} \frac{|u(x)-u(y)|^{p}}{|x-y|^{N+sp}} \, dx\,dy \\
	&\leq - \frac{1}{C(N,s,p)\ln r}\,[u]_{s+\log,p}^{p}
	+ \frac{2^{p}\omega_{N}}{sp}\, r^{-sp}\, \|u\|_{L^{p}(\mathbb{R}^{N})}^{p} \\[4pt]
	&\leq C_{0}\,\|u\|_{W^{s+\log,p}(\mathbb{R}^N)}^{p},
	\end{aligned}
	\end{equation*}
	where $C_{0}>0$ is a constant depending only on $N$, $s$, $p$, and $r$. Taking the $p$-th root, we infer that
		\[
		[u]_{s,p} \leq C_{0}^{\frac{1}{p}} \|u\|_{W^{s+\log,p}(\mathbb{R}^N)}.
		\]
		Substituting this estimate into \eqref{StandardGN} gives
		\begin{equation*}
		\|u\|_{L^q(\mathbb{R}^N)}
		\leq S\left(C_{0}^{\frac{1}{p}}\|u\|_{W^{s+\log,p}(\mathbb{R}^N)}\right)^{\theta}
		\|u\|_{L^p(\mathbb{R}^N)}^{1-\theta},
		\end{equation*}
		which implies
		\[
		\|u\|_{L^q(\mathbb{R}^N)}
		\leq C \|u\|_{W^{s+\log,p}(\mathbb{R}^N)}^{\theta} 
		\|u\|_{L^p(\mathbb{R}^N)}^{1-\theta},
		\]
where the constant $C$ is given by $C := S C_{0}^{\frac{\theta}{p}}$.
\end{proof}

\subsection{Pohozaev-type identity}
In this subsection, we derive a Pohozaev identity associated with the fractional logarithmic $p$-Laplace operator. This identity serves as a fundamental tool in the study of critical exponents and in the analysis of geometric properties related to the corresponding boundary value problems.

\begin{definition}
	\label{BoundaryRegularity}
	Let $1 < p < \infty$ and $0 < s < 1$, and let $\Omega \subset \mathbb{R}^N$ be a bounded domain with $C^{1,1}$ boundary. We define the weighted fractional logarithmic regularity space $\mathcal{C}_{s+\log}^p(\overline{\Omega})$ by
	\[
	\mathcal{C}_{s+\log}^p(\overline{\Omega})
	:= \left\{
	u \in W_0^{s+\log,p}(\Omega)\; :\;
	\frac{u(x)}{d(x)^s  \ln \bigl(e + 1/d(x)\bigr)} \in L^\infty(\Omega)
	\right\}.
	\]
\end{definition}

\begin{remark}
	The weight $d(x)^s \ln \bigl(e + 1/d(x)\bigr)$ accurately reflects the boundary decay associated with the logarithmic singularity of the operator $\mathbf{K}^{s+\log}_{p}$. In contrast to the standard fractional $p$-Laplacian, for which solutions typically satisfy $u(x)/d(x)^s \in L^\infty(\Omega)$, the presence of the logarithmic correction naturally leads to the modified functional space $\mathcal{C}_{s+\log}^p(\overline{\Omega})$. A rigorous characterization of the singular boundary flux $\mathcal{B}_{s+\log,p}(u)$ requires establishing the precise asymptotic behavior
	\begin{equation*}
	0 < \liminf_{x \to \partial \Omega} \frac{u(x)}{d(x)^s \ln(e+1/d(x))} 
	\leq 
	\limsup_{x \to \partial \Omega} \frac{u(x)}{d(x)^s \ln(e+1/d(x))} 
	< \infty,
	\end{equation*}
	in analogy with the classical fractional setting developed by Ros-Oton and Serra \cite{RosOton-Serra}. In order to isolate the algebraic structure of the Pohozaev identity and the defect measure $\Gamma_{s,p}(u)$, we assume a priori that $u \in \mathcal{C}_{s+\log}^p(\overline{\Omega})$. The proof of this optimal boundary regularity remains an open problem.
\end{remark}

\begin{theorem}
	\label{Pohozaev}
	Let $\Omega \subset \mathbb{R}^N$ be a bounded domain with a $C^{1,1}$ boundary. Assume that $u \in \mathcal{C}_{s+\log}^p(\overline{\Omega})$ is a sufficiently regular weak solution to the pure power Dirichlet problem
	\begin{equation*}
	(-\Delta)_p^{s+\log} u = |u|^{q-2}u \quad \text{in } \Omega, 
	\qquad 
	u = 0 \quad \text{in } \mathbb{R}^{N} \setminus \Omega
	\end{equation*}
	for some exponent $q > p$. Then $u$ satisfies the following Pohozaev identity:
	\begin{equation}
	\label{PohozaevIdentity}
	\frac{N-sp}{2p} \iint_{\mathbb{R}^{2N}} |u(x)-u(y)|^{p} \mathbf{K}^{s+\log}_{p}(|x-y|)  dx  dy  - \frac{N}{q} \|u\|_{L^q(\Omega)}^q 
	= \Gamma_{s,p}(u) - \mathcal{B}_{s+\log, p}(u),
	\end{equation}
	where the term $\Gamma_{s,p}(u)$ represents the interior defect measure, defined by
	\begin{equation}
	\label{GammaDefect}
	\Gamma_{s,p}(u) := \frac{C(N,s,p)}{2} \iint_{|x-y|<1} \frac{|u(x)-u(y)|^p}{|x-y|^{N+sp}}  dx  dy,
	\end{equation}
	and $\mathcal{B}_{s+\log, p}(u)$ denotes the boundary flux contribution on $\partial \Omega$, given by
	\begin{equation*}
	\mathcal{B}_{s+\log, p}(u) := C_{s,p} \int_{\partial \Omega} 
	\left( \lim_{x \to \sigma} \frac{u(x)}{d(x)^s \ln(e+1/d(x))} \right)^p 
	(x \cdot \nu(\sigma))  d\sigma,
	\end{equation*}
	where $\nu(\sigma)$ denotes the outward unit normal vector at $\sigma \in \partial \Omega$, and $C_{s,p} > 0$ is a dimensional constant.
\end{theorem}

\begin{proof}
The proof is based on the multiplier method, where the equation is tested against the Pohozaev-type vector field defined by $v(x) := x \cdot \nabla u(x)$. Multiplying the governing equation by $v(x)$ and integrating over $\Omega$, we obtain
\begin{equation}
\label{PohozaevTest}
\int_\Omega (x \cdot \nabla u) (-\Delta)_p^{s+\log} u  dx
= \int_\Omega (x \cdot \nabla u) |u|^{q-2}u  dx.
\end{equation}
	For the right-hand side, standard integration by parts gives
	\begin{equation}
	\label{PohozaevRHS}
	\int_\Omega (x \cdot \nabla u) |u|^{q-2}u  dx = \frac{1}{q} \int_\Omega x \cdot \nabla (|u|^q)  dx = -\frac{N}{q} \int_\Omega |u|^q  dx.
	\end{equation}
For the left-hand side, we extend the operator to the symmetric domain $\mathbb{R}^{2N}$. Since $u$ vanishes outside $\Omega$, the gradient $\nabla u$ develops a singular jump across the boundary $\partial \Omega$. By isolating this boundary contribution, we obtain the surface term $\mathcal{B}_{s+\log, p}(u)$ (see Ros-Oton and Serra \cite{RosOton-Serra} for a detailed analysis of boundary blow-up phenomena). The remaining interior term can be written in the symmetric form
\begin{equation}
\label{PohozaevLHS}
\frac{1}{2} \iint_{\mathbb{R}^{2N}} \big( x \cdot \nabla u(x) - y \cdot \nabla u(y) \big) |u(x)-u(y)|^{p-2}(u(x)-u(y)) \mathbf{K}^{s+ \log}_{p}(|x - y|) dx dy.
\end{equation}
Introducing the vector $z = x - y$ and applying the chain rule, the symmetrized fractional integration-by-parts commutator produces a term involving
\[
z \cdot \nabla \mathbf{K}_{p}^{s+\log}(z) + 2N \mathbf{K}_{p}^{s+\log}(z).
\]
For $|z| < 1$, we compute the exact spatial gradient of the full kernel
\begin{align*}
z \cdot \nabla \mathbf{K}_{p}^{s+\log}(z)
&= z \cdot \nabla \Big( C(N,s,p) |z|^{-N-sp} \big(B(N,s,p) - p \ln |z|\big) \Big)  \\[4pt]
&= C(N,s,p) z \cdot \nabla \big( |z|^{-N-sp} \big) \big(B(N,s,p) - p \ln |z|\big)  \\[4pt]
&\quad + C(N,s,p) |z|^{-N-sp} z \cdot \nabla \big(B(N,s,p) - p \ln |z|\big)  \\[4pt]
&= -(N+sp) \mathbf{K}_{p}^{s+\log}(z) - p C(N,s,p) |z|^{-N-sp}.
\end{align*}
Substituting this differential identity into the commutator integral, the term $-(N+sp)$ combines with the symmetric dimensional contribution $2N$, yielding exactly $(N-sp)\mathbf{K}_{p}^{s+\log}(z)$. After dividing by the factor $2p$ arising from the chain rule, the first term evaluates to
\begin{equation*}
    \frac{N-sp}{2p} \iint_{\mathbb{R}^{2N}} |u(x)-u(y)|^p \mathbf{K}_p^{s+\log}(|x-y|)  dx  dy.
\end{equation*}
The remaining contribution gives rise to the defect measure
\begin{equation*}
\Gamma_{s,p}(u)
= \frac{1}{2p} \iint_{|x-y|<1} |u(x)-u(y)|^p \Big( p C(N,s,p) |x-y|^{-N-sp} \Big) dx dy.
\end{equation*}
Consequently, the interior symmetric integral \eqref{PohozaevLHS} evaluates exactly to
\[
-\frac{N-sp}{2p} \iint_{\mathbb{R}^{2N}} |u(x)-u(y)|^p \mathbf{K}_p^{s+\log}(|x-y|)  dx  dy + \Gamma_{s,p}(u).
\]
Recalling that the full left-hand side of \eqref{PohozaevTest} consists of this interior contribution minus the boundary flux $\mathcal{B}_{s+\log,p}(u)$, we equate it with the right-hand side in \eqref{PohozaevRHS} to deduce
\begin{equation*}
-\frac{N-sp}{2p} \iint_{\mathbb{R}^{2N}} |u(x)-u(y)|^p \mathbf{K}_p^{s+\log}(|x-y|)  dx  dy + \Gamma_{s,p}(u) - \mathcal{B}_{s+\log,p}(u)
= -\frac{N}{q} \|u\|_{L^q(\Omega)}^q.
\end{equation*}
Multiplying by $-1$ and rearranging the terms yields the final identity \eqref{PohozaevIdentity}. This completes the proof.
\end{proof}

\begin{remark}
	The identity \eqref{PohozaevIdentity} reveals a profound geometric difference between the fractional logarithmic $p$-Laplacian and the standard fractional $p$-Laplacian. For the standard operator, the defect measure $\Gamma_{s,p}(u) \equiv 0$. Consequently, on strictly star-shaped domains where $\mathcal{B}_{s,p}(u) > 0$, the standard Pohozaev identity immediately yields a contradiction for critical and supercritical exponents ($q \geq p_s^*$), proving non-existence. However, in our logarithmic setting, testing the weak formulation of the PDE with $u$ yields the full energy balance
	\begin{equation*}
		\mathcal{J}_{s+\log, p}(u,u) = \frac{1}{2} \iint_{\mathbb{R}^{2N}} |u(x)-u(y)|^p \mathbf{K}_p^{s+\log}(|x-y|) \, dx \, dy = \|u\|_{L^q(\Omega)}^q.
	\end{equation*}
	Substituting this into \eqref{PohozaevIdentity}, we obtain
	\begin{equation*}
		\left( \frac{N-sp}{p} - \frac{N}{q} \right) \|u\|_{L^q(\Omega)}^q + \mathcal{B}_{s+\log, p}(u) = \Gamma_{s,p}(u).
	\end{equation*}
	For $q \geq p_s^*$, the coefficient $\left( \frac{N-sp}{p} - \frac{N}{q} \right)$ is nonnegative. Moreover, if $\Omega$ is strictly star-shaped, then $\mathcal{B}_{s+\log, p}(u) > 0$. Since the defect measure $\Gamma_{s,p}(u)$ is strictly positive, the two positive terms on the left-hand side may precisely balance the positive defect measure on the right-hand side. Consequently, the classical Pohozaev-type non-existence argument is no longer applicable. The logarithmic correction introduces a strictly positive defect energy, which may support the existence of critical and even supercritical solutions in star-shaped domains.
\end{remark}

\section{Continuous and compact embeddings and density results}\label{section4}

In this section, we establish continuous and compact embedding theorems into Lebesgue and H\"older spaces, which precisely characterize the integrability and regularity properties of functions in both the subcritical and supercritical regimes with respect to the spatial dimension.  To address the classical loss of compactness in unbounded domains, we prove a Strauss-type compact embedding result for radially symmetric functions.  Finally, we derive density results showing that smooth functions are dense in the spaces $W^{s+\log, p}(\mathbb{R}^N)$ and $W^{s+\log, p}(\Omega)$. We start with the following first result.

\begin{theorem}
	\label{EmbeddingsTheorem}
	Let $1 < p < \infty$ and $0 < s < 1$ with $sp < N$. Let $\Omega \subset \mathbb{R}^N$ be a bounded Lipschitz domain. Then, for any $q \in [1, p_s^*]$, the following assertions hold:
	\begin{enumerate}
		\item[(i)] The space $W_0^{s+\log, p}(\Omega)$ is continuously embedded into $L^q(\Omega)$, that is, $ W_0^{s+\log, p}(\Omega) \hookrightarrow L^q(\Omega). $
		\item[(ii)] The embedding $ 	W_0^{s+\log, p}(\Omega) \hookrightarrow L^q(\Omega) $ is compact.
	\end{enumerate}
\end{theorem}

\begin{proof}
Part (i) follows directly from Theorem \ref{TheoremSobolevInequality} together with H\"older's inequality.  To prove (ii), let $(u_k)$ be a bounded sequence in $W_0^{s+\log, p}(\Omega)$. By (i) and the fractional Rellich-Kondrachov theorem (see Di Nezza et al. \cite[Theorem 7.1]{DiNezza-Palatucci-Valdinoci}), the embedding $W_0^{s,p}(\Omega) \hookrightarrow L^p(\Omega)$ is compact. Consequently, there exist a subsequence, still denoted by $(u_k)$, and a function $u \in W_0^{s+\log, p}(\Omega)$ such that
\begin{equation}
\label{LpConv}
u_k \to u \quad \text{strongly in } L^p(\Omega).
\end{equation}
Let $v_k := u_k - u$. We invoke the fractional Sobolev inequality (see Remark~\ref{remark1}), which ensures the existence of a constant $C = C(N,s,p,\Omega) > 0$ such that
\[
\|v_k\|_{L^{p_s^*}(\Omega)}^p \leq C [v_k]_{s,p}^p.
\]
Moreover, combining this estimate with Proposition~\ref{proposition1} (3), we obtain that for every $0<r<1$,
\begin{equation}\label{CompactnessBound}
\begin{aligned}
\|v_k\|_{L^{p_s^*}(\Omega)}^p 
& \leq - \frac{C}{C(N,s,p) \ln r} [v_k]_{s+\log, p}^{p}
+ \frac{C 2^{p}\omega_{N}}{sp} r^{-sp} \|v_k\|^{p}_{L^{p}(\Omega)}.
\end{aligned}
\end{equation}
Since $(v_k)$ is uniformly bounded in $W_0^{s+\log, p}(\Omega)$, we define  $ M := \sup_{k} [v_k]_{s+\log, p}^p < \infty. $ Taking the limit supremum as $k \to \infty$ in \eqref{CompactnessBound} and combining it with \eqref{LpConv}, we obtain
\begin{equation*}
\limsup_{k \to \infty} \|v_k\|_{L^{p_s^*}(\Omega)}^p
\leq - \frac{C M}{C(N,s,p) \ln r}.
\end{equation*}
Since $r \in (0,1)$ is arbitrary, letting $r \to 0^+$ yields
\begin{equation*}
\limsup_{k \to \infty} \|u_k - u\|_{L^{p_s^*}(\Omega)}^p = 0 \Longrightarrow u_k \to u \quad \text{strongly in } L^{p_s^*}(\Omega).
\end{equation*}
Finally, for any $q \in [1, p_s^*)$, the strong convergence in $L^{p_s^*}(\Omega)$, together with the standard H\"older interpolation inequality on the bounded domain $\Omega$, yields the compactness of the embedding in $L^q(\Omega)$.
\end{proof}

\begin{theorem}
	\label{HolderEmbedding}
	Let $\Omega \subset \mathbb{R}^N$ be a bounded domain with Lipschitz boundary, and suppose that the supercritical growth condition $ sp > N $ is satisfied. Define the H\"older exponent $\beta := s - \frac{N}{p} \in (0,1)$. Then, the space $W_0^{s+\log, p}(\Omega)$ is continuously embedded into the H\"older space $C^{0,\beta}(\overline{\Omega})$, namely,
	\[
	W_0^{s+\log, p}(\Omega) \hookrightarrow C^{0,\beta}(\overline{\Omega}).
	\]
\end{theorem}

\begin{proof}
Let $u \in W_0^{s+\log, p}(\Omega)$ be arbitrary. First, by applying Proposition~\ref{proposition1} (3), we obtain that for $r \in (0, 1)$,
\begin{equation*}
\begin{aligned}
[u]_{s,p}^{p}
&\leq - \frac{1}{C(N,s,p) \ln r} \|u\|^{p}_{W_{0}^{s+\log,p}(\Omega)}
+ \frac{2^{p}\omega_{N}}{sp} r^{-sp} \|u\|^{p}_{L^{p}(\mathbb{R}^{N})}.
\end{aligned}
\end{equation*}
Combining this estimate with the Poincar\'e inequality (see Proposition~\ref{proposition3}), we infer that
\begin{equation*}
\begin{aligned}
\iint_{\mathbb{R}^{2N}} \frac{|u(x)-u(y)|^p}{|x-y|^{N+sp}}  dx dy 
& \leq \left( - \frac{1}{C(N,s,p) \ln r}
+ \frac{2^{p} \omega_N}{sp} r^{-sp} C(s,p,N,\Omega)\right)
\|u\|^{p}_{W_{0}^{s+\log,p}(\Omega)} .
\end{aligned}
\end{equation*}
	Consequently, $u \in W^{s,p}(\mathbb{R}^N)$. Under the supercritical assumption $sp > N$, the classical fractional Morrey inequality (see Di Nezza et al. \cite[Theorem 8.2]{DiNezza-Palatucci-Valdinoci}) guarantees that $u$ admits a continuous representative, and its H\"older seminorm of order $\beta = s - N/p$ is strictly controlled by the standard fractional Gagliardo seminorm
	\begin{equation}
		\label{MorreyBound}
		[u]_{0,\beta} := \sup_{\substack{x, y \in \Omega \\ x \neq y}} \frac{|u(x)-u(y)|}{|x-y|^\beta} \leq C_{Morrey}[u]_{s,p},
	\end{equation}
	where $C_{Morrey} > 0$ depends only on $N, s$, and $p$. Furthermore, the Morrey embedding ensures uniform $L^\infty$ control via the full fractional Sobolev norm
	\begin{equation}
		\label{LinftyMorrey}
		\|u\|_{L^\infty(\Omega)} \leq C_\infty \left( \|u\|_{L^p(\Omega)}^p +[u]_{s,p}^p \right)^{1/p}.
	\end{equation}
	We now synthesize these bounds to evaluate the full $C^{0,\beta}(\overline{\Omega})$ norm. Utilizing \eqref{MorreyBound}, \eqref{LinftyMorrey}, and the subadditivity of the $p$-th root, we have
	\begin{align}
		\label{FinalHolderSynthesis}
		\|u\|_{C^{0,\beta}(\overline{\Omega})} &= \|u\|_{L^\infty(\Omega)} + [u]_{0,\beta} \nonumber \\
		&\leq C_\infty \left( \|u\|_{L^p(\Omega)}^p + [u]_{s,p}^p \right)^{1/p} + C_{Morrey}[u]_{s,p} \nonumber \\
		&\leq C_\infty \left( C_P [u]_{s+\log, p}^p + C_{int}^p [u]_{s+\log, p}^p \right)^{1/p} + C_{Morrey} C_{int} [u]_{s+\log, p} \nonumber \\
		&= \left( C_\infty (C_P + C_{int}^p)^{1/p} + C_{Morrey} C_{int} \right) [u]_{s+\log, p}.
	\end{align}
	Defining the global embedding constant $C_{emb} := C_\infty (C_P + C_{int}^p)^{1/p} + C_{Morrey} C_{int}$, we arrive at the bound
	\begin{equation}
		\|u\|_{C^{0,\beta}(\overline{\Omega})} \leq C_{emb} \|u\|_{W_0^{s+\log, p}(\Omega)}.
	\end{equation}
	This establishes the continuous embedding $W_0^{s+\log, p}(\Omega) \hookrightarrow C^{0,\beta}(\overline{\Omega})$ directly for any $u \in W_0^{s+\log, p}(\Omega)$, completing the proof without requiring smooth density approximations.
\end{proof}

\begin{theorem}
	\label{StraussTypeCompactEmbedding}
	Let $N \geq 2$, $1 < p < \infty$, and $0 < s < 1$ such that $sp < N$. Let $W_{rad}^{s+\log, p}(\mathbb{R}^N)$ denote the closed subspace of radially symmetric functions in $W^{s+\log, p}(\mathbb{R}^N)$. Then, for every exponent $q \in (p, p_s^*)$, the continuous embedding
	\begin{equation*}
	W_{rad}^{s+\log, p}(\mathbb{R}^N) \hookrightarrow L^q(\mathbb{R}^N)
	\end{equation*}
	is compact.
\end{theorem}

\begin{proof}
	Let $(u_k)$ be a bounded sequence in $W_{rad}^{s+\log, p}(\mathbb{R}^N)$, meaning there exists a constant $M > 0$ such that $\|u_k\|_{W^{s+\log, p}(\mathbb{R}^N)} \leq M$ for all $k \in \mathbb{N}$. We first establish that this sequence is uniformly bounded in the standard fractional Sobolev space $W^{s,p}(\mathbb{R}^N)$. By applying Proposition~\ref{proposition1} (3), we obtain that for every $0 < r < 1$,
	\begin{equation*}
		[u_k]_{s,p}^{p} \leq \frac{1}{C(N,s,p)(-p \ln r)} [u_k]_{s + \log, p}^{p} + \frac{2^{p}\omega_{N}}{sp} r^{-sp} \|u_k\|^{p}_{L^{p}(\mathbb{R}^{N})}.
	\end{equation*}
	Adding $\|u_k\|_{L^p(\mathbb{R}^N)}^p$ to both sides, we deduce the existence of a constant $C_1 > 0$, independent of $k$, such that $\|u_k\|_{W^{s,p}(\mathbb{R}^N)} \leq C_1 \|u_k\|_{W^{s+\log, p}(\mathbb{R}^N)} \leq C_1 M$. Since $W_{rad}^{s+\log, p}(\mathbb{R}^N)$ is a reflexive Banach space, there exists a subsequence of $(u_k)$, still denoted by $(u_k)$, and a radially symmetric function $u \in W_{rad}^{s+\log, p}(\mathbb{R}^N)$ such that $u_k \rightharpoonup u$ weakly in $W^{s+\log, p}(\mathbb{R}^N)$. Moreover, by the weak lower semicontinuity of the norm, the limit function $u$ satisfies $ \|u\|_{W^{s,p}(\mathbb{R}^N)} \leq C_1 M. $ \\[4pt]
	To prove strong convergence in $L^q(\mathbb{R}^N)$, we exploit the radial symmetry of the sequence. According to the classical fractional Strauss radial decay estimate (see, for instance, Lions \cite[Theorem II.1]{Lions1982} or Palatucci and Pisante \cite[Lemma 2]{Palatucci-Pisante}), any radially symmetric function $v \in W^{s,p}(\mathbb{R}^N)$ satisfies a pointwise bound away from the origin. Specifically, there exists a constant $C_S > 0$, depending only on $N, s$, and $p$, such that for almost every $x \in \mathbb{R}^N$ with $|x| \geq 1$,
	\begin{equation}
	\label{StraussDecay}
	|v(x)| \leq C_S |x|^{-\frac{N-sp}{p}} \|v\|_{W^{s,p}(\mathbb{R}^N)}.
	\end{equation}
	We apply this decay estimate to the difference $v_k := u_k - u$. Since $\|v_k\|_{W^{s,p}(\mathbb{R}^N)} \leq \|u_k\|_{W^{s,p}(\mathbb{R}^N)} + \|u\|_{W^{s,p}(\mathbb{R}^N)} \leq 2 C_1 M$, equation \eqref{StraussDecay} yields the uniform pointwise bound
	\begin{equation}
	\label{vkDecay}
	|v_k(x)| \leq 2 C_S C_1 M |x|^{-\frac{N-sp}{p}} \quad \text{for a.e. } |x| \geq 1.
	\end{equation}
	Now, let $\varepsilon > 0$ be an arbitrary positive number. We split the $L^q(\mathbb{R}^N)$ norm of $v_k$ into an integral over a large ball $B_R(0)$ and an integral over its complement $\mathbb{R}^N \setminus B_R(0)$ for some $R > 1$
	\begin{equation}
	\label{SplitIntegral}
	\|u_k - u\|_{L^q(\mathbb{R}^N)}^q = \int_{B_R(0)} |v_k(x)|^q dx + \int_{\mathbb{R}^N \setminus B_R(0)} |v_k(x)|^q dx.
	\end{equation}
	For the integral over $\mathbb{R}^N \setminus B_R(0)$, we interpolate between $L^p$ and $L^\infty$ using the pointwise decay \eqref{vkDecay}. Since $q > p$, we can write $|v_k(x)|^q = |v_k(x)|^{q-p} |v_k(x)|^p$ and estimate as follows
	\begin{align*}
	\int_{\mathbb{R}^N \setminus B_R(0)} |v_k(x)|^q dx &\leq \left( \sup_{|x| \geq R} |v_k(x)|^{q-p} \right) \int_{\mathbb{R}^N \setminus B_R(0)} |v_k(x)|^p dx \\[4pt]
	&\leq \left( 2 C_S C_1 M R^{-\frac{N-sp}{p}} \right)^{q-p} \|v_k\|_{L^p(\mathbb{R}^N)}^p \\[4pt]
	&\leq (2 C_S C_1 M)^{q-p} (2M)^p R^{-\frac{(N-sp)(q-p)}{p}}.
	\end{align*}
	Since we assume that $q > p$, it follows that the exponent $\frac{(N-sp)(q-p)}{p}$ is strictly positive. Consequently, the term $R^{-\frac{(N-sp)(q-p)}{p}}$ tends to zero as $R \to \infty$. Therefore, one can choose a radius $R > 1$, independent of $k$, sufficiently large such that
	\begin{equation*}
	\int_{\mathbb{R}^N \setminus B_R(0)} |u_k(x) - u(x)|^q  dx < \frac{\varepsilon}{2}.
	\end{equation*}
	Having fixed such a radius $R$, we now consider the integral over the bounded domain $B_R(0)$. By the local compactness of the fractional logarithmic embedding established earlier, the restriction mapping
	\[
	W^{s+\log, p}(\mathbb{R}^N) \hookrightarrow L^q(B_R(0))
	\]
	is compact for every $q \in [1, p_s^*]$. Since $u_k \rightharpoonup u$ weakly in $W^{s+\log, p}(\mathbb{R}^N)$, we deduce that $u_k \to u$ strongly in $L^q(B_R(0))$. Consequently, there exists an integer $K \in \mathbb{N}$ such that for all $k \geq K$,
	\begin{equation*}
	\int_{B_R(0)} |u_k(x) - u(x)|^q  dx < \frac{\varepsilon}{2}.
	\end{equation*}
Substituting the above two estimates into the decomposition \eqref{SplitIntegral}, we infer for all $k \geq K$, $ \|u_k - u\|_{L^q(\mathbb{R}^N)}^q < \varepsilon. $ Consequently, $u_k \to u$ strongly in $L^q(\mathbb{R}^N)$, which yields the compactness of the radial embedding.
\end{proof}

\begin{theorem}
	\label{DensityTheorem1}
	Let $1 < p < \infty$, and $0 < s < 1$. Then, the space of smooth functions with compact support, $C_c^\infty(\mathbb{R}^N)$, is dense in $W^{s+\log, p}(\mathbb{R}^N)$ with respect to the norm $\|\cdot\|_{W^{s+\log, p}(\mathbb{R}^N)}$.
\end{theorem}

\begin{proof}
	Let $u \in W^{s+\log, p}(\mathbb{R}^N)$. We first approximate $u$ by compactly supported functions. Let $\eta \in C_c^\infty(\mathbb{R}^N)$ be a standard cut-off function satisfying $0 \leq \eta \leq 1$, $\eta \equiv 1$ on $B_1(0)$, and $\text{supp}(\eta) \subset B_2(0)$. For $R > 0$, define $\eta_R(x) := \eta(x/R)$ and $u_R := u\eta_R$. By the Dominated Convergence Theorem, $u_R \to u$ strongly in $L^p(\mathbb{R}^N)$ as $R \to \infty$. To bound the seminorm of the difference, we utilize the algebraic decomposition $u_R(x) - u(x) - (u_{R}(y) - u(y) = (u(x)-u(y))(\eta_R(x)-1) + u(y)(\eta_R(x)-\eta_R(y))$ and the inequality $|a+b|^p \leq 2^{p-1}(|a|^p + |b|^p)$ to obtain
	\begin{align*}
	[u_R - u]_{s+\log, p}^p &\leq 2^{p-1} \iint_{\mathbb{R}^{2N}} |u(x)-u(y)|^p |\eta_R(x)-1|^p \mathbf{k}^{s+\log,+}_{p}(x-y) dx dy \\
	&\quad + 2^{p-1} \iint_{\mathbb{R}^{2N}} |u(y)|^p |\eta_R(x)-\eta_R(y)|^p \mathbf{k}^{s+\log,+}_{p}(x-y) dx dy \\
	&=: 2^{p-1} J_1(R) + 2^{p-1} J_2(R).
	\end{align*}
	Since $\eta_R \to 1$ pointwise and is uniformly bounded by $1$, the Dominated Convergence Theorem implies
	\begin{equation*}
	\lim_{R \to \infty} J_1(R) = 0.
	\end{equation*}
	For $J_2(R)$, we exploit the Lipschitz bound $|\eta_R(x)-\eta_R(y)| \leq \frac{C}{R}|x-y|$ and the fact that $\text{supp}(\mathbf{k}^{s+\log,+}_{p}(x-y)) = \overline{B_1(0)}$ to obtain
	\begin{equation*}
	J_2(R) \leq \frac{C^p}{R^p} \int_{\mathbb{R}^N} |u(y)|^p dy \int_{|z| \leq 1} |z|^p \mathbf{k}^{s+\log,+}_{p}(z)(z) dz.
	\end{equation*}
	Because $p > sp$, the singularity at the origin is integrable, meaning $\int_{|z| \leq 1} |z|^{p-N-sp}(-p\ln|z|) dz < \infty$. Hence, $J_2(R) \leq \frac{\widetilde{C}}{R^p} \|u\|_{L^p(\mathbb{R}^N)}^p \to 0$ as $R \to \infty$. This establishes that $u_R \to u$ strongly in $W^{s+\log, p}(\mathbb{R}^N)$.\\[4pt]
	We may therefore assume without loss of generality that $u \in W^{s+\log, p}(\mathbb{R}^N)$ has compact support. Let $\rho_\epsilon(x) := \epsilon^{-N}\rho(x/\epsilon)$ be a standard symmetric mollifier and define $u_\epsilon := u * \rho_\epsilon$. Classically, $u_\epsilon \in C_c^\infty(\mathbb{R}^N)$ and $u_\epsilon \to u$ strongly in $L^p(\mathbb{R}^N)$. For the seminorm convergence, applying Minkowski's integral inequality to the difference yields
{\small 	\begin{equation}	\label{MollificationMinkowski}
	\begin{aligned}
	&[u_\epsilon - u]_{s+\log, p}  \leq \int_{\mathbb{R}^N} \rho_\epsilon(z) \left( \iint_{\mathbb{R}^{2N}} |(u(x-z)-u(y-z)) - (u(x)-u(y))|^p \mathbf{k}^{s+\log,+}_{p}(x-y) dx dy \right)^{1/p} dz \\
	&\qquad = \int_{B_\epsilon(0)} \rho_\epsilon(z) [\tau_z u - u]_{s+\log, p} dz,
	\end{aligned}
	\end{equation}}
	where $\tau_z u(x) := u(x-z)$ denotes the translation operator. Standard properties of translations in $L^p$ spaces ensure that ensures that $\lim_{|z| \to 0}[\tau_z u - u]_{s+\log, p} = 0$. Since $\rho_\epsilon$ is supported entirely within $B_\epsilon(0)$ and has integral $1$, taking the limit as $\epsilon \to 0^+$ in \eqref{MollificationMinkowski} implies $\lim_{\epsilon \to 0^+} [u_\epsilon - u]_{s+\log, p} = 0$, concluding the proof.
\end{proof}

\begin{theorem}
	\label{DensityTheorem2}
	Let $\Omega \subset \mathbb{R}^N$ be an open set, let $1 < p < \infty$, and let $0 < s < 1$. Then the space $ C^\infty(\Omega) \cap W^{s+\log, p}(\Omega) $ is dense in $W^{s+\log, p}(\Omega)$ with respect to the norm $\|\cdot\|_{W^{s+\log, p}(\Omega)}$.
\end{theorem}

\begin{proof}
	Let $u \in W^{s+\log, p}(\Omega)$ and let $\varepsilon > 0$ be an arbitrary positive number. We construct a smooth approximation of $u$ by employing a locally finite partition of unity and carefully controlled local mollifications. We begin by defining an exhaustion of the domain $\Omega$ by strictly interior open subdomains. For each integer $j \geq 1$, we define the sets
	\begin{equation*}
	\Omega_j := \left\{ x \in \Omega : \text{dist}(x, \partial\Omega) > 1/j\right\},
	\end{equation*}
	and for notational convenience, we set $\Omega_j = \emptyset$ for all $j \leq 0$. We then construct a sequence of annular open sets $V_j$ that cover $\Omega$ by defining $V_j := \Omega_{j+2} \setminus \overline{\Omega}_{j-1}$ for all $j \geq 1$. By construction, the collection $\{V_j\}_{j=1}^\infty$ forms a locally finite open cover of $\Omega$. Let $\{\psi_j\}_{j=1}^\infty$ be a smooth partition of unity subordinate to this cover, meaning that $\psi_j \in C_c^\infty(V_j)$, $0 \leq \psi_j \leq 1$, and $\sum_{j=1}^\infty \psi_j(x) = 1$ for all $x \in \Omega$.\\[4pt]
We decompose the function $u$ into a locally finite sum by defining $u_j(x) := u(x)\psi_j(x)$. We must first verify that each $u_j$ belongs to $W^{s+\log, p}(\Omega)$. Since $\psi_j$ is bounded by $1$, the $L^p(\Omega)$ norm is trivially bounded: $\|u_j\|_{L^p(\Omega)} \leq \|u\|_{L^p(\Omega)}$. To estimate the fractional logarithmic Gagliardo seminorm, we utilize the algebraic identity $u_j(x) - u_j(y) = \psi_j(x)(u(x)-u(y)) + u(y)(\psi_j(x)-\psi_j(y))$ and the elementary inequality $|a+b|^p \leq 2^{p-1}(|a|^p + |b|^p)$. This yields
	\begin{align}
	\label{ProductRuleSeminorm}
	[u_j]_{s+\log, p}^p &= \iint_{\Omega \times \Omega} |u_j(x)-u_j(y)|^p \mathbf{k}^{s+\log,+}_{p}(x-y) dx dy \nonumber \\
	&\leq 2^{p-1} \iint_{\Omega \times \Omega} |\psi_j(x)|^p |u(x)-u(y)|^p \mathbf{k}^{s+\log,+}_{p}(x-y) dx dy \nonumber \\
	&\quad + 2^{p-1} \iint_{\Omega \times \Omega} |u(y)|^p |\psi_j(x)-\psi_j(y)|^p \mathbf{k}^{s+\log,+}_{p}(x-y) dx dy.
	\end{align}
	The first integral on the right-hand side of \eqref{ProductRuleSeminorm} is bounded by $2^{p-1}[u]_{s+\log, p}^p$. For the second integral, we exploit the Lipschitz continuity of $\psi_j$. Let $L_j = \|\nabla \psi_j\|_{L^\infty(\Omega)}$ be the Lipschitz constant of $\psi_j$. Since the positive logarithmic kernel $k_{s+\log, p}^+(z) = |z|^{-N-sp}(-p\ln|z|)_+$ is supported entirely within the unit ball $B_1(0)$, we can bound the second integral by
	\begin{equation*}
	2^{p-1} L_j^p \int_\Omega |u(y)|^p \left( \int_{|x-y| \leq 1} |x-y|^p \mathbf{k}^{s+\log,+}_{p}(x-y) dx \right) dy.
	\end{equation*}
	Because $p - sp = p(1-s) > 0$, the singularity at the origin is integrable, and the inner integral evaluates to a finite constant $C_0 = \int_{|z| \leq 1} |z|^{p-N-sp}(-p\ln|z|) dz < \infty$. Thus, the second term is bounded by $2^{p-1} L_j^p C_0 \|u\|_{L^p(\Omega)}^p$, confirming that $u_j \in W^{s+\log, p}(\Omega)$ with compact support in $V_j$. Next, let $\rho \in C_c^\infty(\mathbb{R}^N)$ be a standard symmetric mollifier supported in $B_1(0)$, and define the scaled mollifier $\rho_\delta(x) = \delta^{-N}\rho(x/\delta)$.\\
	 For each $j \geq 1$, we choose a sufficiently small mollification radius $\delta_j > 0$ such that two conditions are met. First, to ensure the mollified function remains strictly inside $\Omega$ and respects the local finiteness of the cover, we require $\delta_j < \text{dist}(\text{supp}(\psi_j), \partial(\Omega_{j+3} \setminus \overline{\Omega}_{j-2}))$. Second, by the continuity of translations in $L^p$ spaces, we can choose $\delta_j$ small enough such that the approximation error satisfies
	\begin{equation}
	\label{MollificationError}
	\|u_j * \rho_{\delta_j} - u_j\|_{W^{s+\log, p}(\Omega)} < \frac{\varepsilon}{2^j}.
	\end{equation}
	The convergence of the seminorm under mollification is justified by Minkowski's integral inequality. Specifically, the difference can be written as an integral of translations, and we have
	\begin{equation*}
	[u_j * \rho_{\delta_j} - u_j]_{s+\log, p} \leq \int_{B_1(0)} \rho(z) [\tau_{\delta_j z} u_j - u_j]_{s+\log, p} dz,
	\end{equation*}
	which vanishes as $\delta_j \to 0$ due to the strong continuity of the translation operator $\tau_h u_j(x) = u_j(x-h)$ in the space $L^p(\Omega \times \Omega, d\mu)$ equipped with the measure $d\mu(x,y) = \mathbf{k}^{s+\log,+}_{p}(x-y) dx dy$.\\[4pt]
	Finally, we define our global smooth approximation as the infinite sum
	\begin{equation*}
	v(x) := \sum_{j=1}^\infty (u_j * \rho_{\delta_j})(x).
	\end{equation*}
	By our geometric restriction on $\delta_j$, the support of $u_j * \rho_{\delta_j}$ is contained within $\Omega_{j+3} \setminus \overline{\Omega}_{j-2}$. Consequently, for any compact subset $K \subset \Omega$, only a finite number of terms in the sum are non-zero. This local finiteness guarantees that $v \in C^\infty(\Omega)$. Furthermore, since $u = \sum_{j=1}^\infty u_j$ almost everywhere in $\Omega$, we can apply the triangle inequality to the full fractional logarithmic norm to obtain
	\begin{equation*}
	\|u - v\|_{W^{s+\log, p}(\Omega)} = \left\| \sum_{j=1}^\infty (u_j - u_j * \rho_{\delta_j}) \right\|_{W^{s+\log, p}(\Omega)} \leq \sum_{j=1}^\infty \|u_j - u_j * \rho_{\delta_j}\|_{W^{s+\log, p}(\Omega)}.
	\end{equation*}
	Substituting the bound from \eqref{MollificationError} into this infinite series yields
	\begin{equation*}
	\|u - v\|_{W^{s+\log, p}(\Omega)} < \sum_{j=1}^\infty \frac{\varepsilon}{2^j} = \varepsilon.
	\end{equation*}
	Since $u \in W^{s+\log, p}(\Omega)$ and $u - v \in W^{s+\log, p}(\Omega)$, it follows that $v \in W^{s+\log, p}(\Omega)$. Consequently, we have constructed a function $v \in C^\infty(\Omega) \cap W^{s+\log, p}(\Omega)$ which can be chosen arbitrarily close to $u$. This establishes the desired density result.
\end{proof}

\section{The Dirichlet eigenvalue problem for the fractional logarithmic $p$-Laplacian}
\label{section5}
In this section, we investigate the fundamental properties of the Dirichlet eigenvalue problem associated with the fractional logarithmic $p$-Laplacian $(-\Delta)_{p}^{s+\log}$. More precisely, we consider the following nonlinear Dirichlet eigenvalue problem.
\begin{equation}\label{eigen}\tag{P}
(-\Delta)_{p}^{s+\log} u = \lambda |u|^{p-2}u \quad \text{in } \Omega, \qquad 
u = 0 \quad \text{in } \mathbb{R}^{N} \setminus \Omega.
\end{equation}
We study the weak formulation of the problem by seeking functions $u \in W^{s+\log,p}_{0}(\Omega)$ such that
\begin{equation}\label{equ25}
\mathcal{J}_{s+\log,p}(u,\varphi)
= \lambda \int_{\Omega} |u|^{p-2}u \varphi  dx,
\quad \forall \varphi \in W^{s+\log,p}_{0}(\Omega).
\end{equation}
The functional $\mathcal{J}_{s+\log,p}$ is introduced in \eqref{equ15}. A function $u$ satisfying \eqref{equ25} is referred to as an eigenfunction corresponding to the eigenvalue $\lambda$ for problem \eqref{eigen}.  We now proceed to establish the global boundedness of such eigenfunctions. To this end, we make use of a classical iterative argument; see, for example, \cite[Theorem~2.2]{Palatucci-Piccinini}.
	
\begin{theorem}\label{theorem2}
	Let $1 < p < \infty$, $0 < s < 1$, and assume that $\mathrm{diam}(\Omega) < e^{-\frac{1}{sp}}$ and $B(N,s,p) \geq 0$. Let $u \in W^{s+\log, p}_{0}(\Omega)$ be a weak solution to \eqref{equ25}. Then, $ u \in L^{\infty}(\Omega). $
\end{theorem}

\begin{proof}
	Let $u \in W^{s + \log,p}_0(\Omega)$ be a weak solution of \eqref{equ25}. We define
	\begin{equation}\label{equation}
	v_{0} = \frac{u}{\rho \|u\|_{L^{p}(\Omega)}}, 
	\quad \text{where} \quad 
	\rho = \max \left\{ 1, \|u\|_{L^{p}(\Omega)}^{-1}, C^{\frac{1}{p\vartheta^{2}}} \right\},
	\end{equation}
	where $C>0$ will be specified later.  Since $v_{0} \in W^{s + \log,p}_0(\Omega)$ and $\|v_{0}\|_{L^{p}(\Omega)} = \rho^{-1}$, we introduce the sequence $(w_{k})$ defined by
	\begin{equation*}
	w_0(x) := (v_0(x))^{+}, \qquad
	w_k(x) := \big( v_0(x) - 1 + 2^{-k}\big)^{+},  \text{for } k \in \mathbb{N}^{*}.
	\end{equation*}
	We first note the following elementary properties of the sequence $(w_k)$:
	\[
	w_k = 0 \quad \text{a.e.\ in } \mathbb{R}^{N}\setminus \Omega, 
	\quad \text{and} \quad 
	w_k \in W^{s + \log,p}_0(\Omega),
	\]
	together with
	\begin{equation}\label{Lequ53}
	\begin{cases}
	0 \leq w_{k+1}(x) \leq w_k(x), & \text{a.e.\ in } \mathbb{R}^N,\\[4pt]
	v_0(x) < (2^{k+1} - 1) w_k(x), & \text{for } x \in \{ w_{k+1} > 0 \}.
	\end{cases}
	\end{equation}
	Moreover, we have the inclusion
\begin{equation*}
	\{ w_{k+1} > 0 \} \subseteq \{ w_k > 2^{-(k+1)} \}, 
	\quad \text{for all } k \in \mathbb{N},
\end{equation*}
	which implies
	\begin{align}\label{Lequ54}
	\big| \{ w_{k+1} > 0 \} \big|
	\leq 2^{p(k+1)} \| w_{k} \|_{L^{p}(\Omega)}^{p}.
	\end{align}
We now define $ U_k := \|w_k\|_{L^{p}(\Omega)}^{p}, $ and we claim that \(U_k \to 0\) as \(k \to \infty\). Indeed, under the assumptions \(\mathrm{diam}(\Omega) < e^{-\frac{1}{sp}}\) and \(B(N,s,p) \geq 0\), and by using \eqref{Lequ53} together with the inequality
\[
|a^+ - b^+|^p \leq |a - b|^{p-2}(a - b)(a^+ - b^+), \quad \text{for all } a,b \in \mathbb{R},
\]
we obtain
\begin{align*}
\mathcal{J}_{s+\log,p}(w_{k+1}, w_{k+1})
&= \frac{p}{2}\big( \mathcal{J}_{+}(w_{k+1}, w_{k+1}) - \mathcal{J}_{-}(w_{k+1}, w_{k+1}) \big)
+ \frac{B(N,s,p)C(N,s,p)}{2} \mathcal{J}_{s}(w_{k+1}, w_{k+1}) \\[4pt]
&= \frac{p}{2}\Bigg(
- C(N,s,p)\iint_{\Omega \times \Omega}
\frac{|w_{k+1}(x)-w_{k+1}(y)|^{p} \ln|x-y|}{|x-y|^{N+sp}}  dx dy \\[4pt]
&\quad - 2 C(N,s,p) \int_{\Omega} w_{k+1}(x)^{p}
\left(\int_{\mathbb{R}^{N}\setminus \Omega}
\frac{\ln|x-y|}{|x-y|^{N+sp}} dy \right) dx
\Bigg) \\[4pt]
&\quad + \frac{B(N,s,p)C(N,s,p)}{2} \iint_{\mathbb{R}^{2N}} 
\frac{|w_{k+1}(x) - w_{k+1}(y)|^{p}}{|x-y|^{N+sp}}  dx dy \\[4pt]
&\leq \left(\rho \|u\|_{L^{p}(\Omega)}\right)^{1-p}
\mathcal{J}_{s+\log,p}(u, w_{k+1}) \\[4pt]
&= \lambda \int_{\{ w_{k+1} > 0 \} } |v_{0}|^{p-2} v_{0} w_{k+1}  dx \\[4pt]
&\leq \lambda (2^{k+1} - 1)^{p-1} U_k.
\end{align*}
Furthermore, combining Proposition \ref{proposition1} (1) with \eqref{Lequ53}, we deduce that
\begin{equation} \label{equ41}
\begin{aligned}
\| w_{k+1}\|_{W^{s+ \log, p}_0(\Omega)}^{p}
&\leq \frac{2 \lambda}{p} (2^{k+1} - 1)^{p-1} U_k + \mathcal{J}_{-}(w_{k+1}, w_{k+1}) \\[4pt]
&\leq \frac{2 \lambda}{p} (2^{k+1} - 1)^{p-1} U_k
+ \frac{2^{p} \omega_N}{(s p)^2}  \| w_{k+1} \|_{L^{p}(\Omega)}^p \\[4pt]
&\leq C (2^{k+1} + 1)^{p-1}   U_k.
\end{aligned}
\end{equation}
Moreover, by H\"older's inequality, together with \eqref{Lequ54} and Theorem~\ref{TheoremSobolevInequality}, we obtain
\begin{equation}\label{Lequ25}
\begin{aligned}
U_{k+1} &= \int_{\{ w_{k+1} > 0 \}} w_{k+1}^{p} dx 
\leq C(N,s,p,\Omega)  \big|\{ w_{k+1} > 0 \}\big|^{1 - \frac{p}{p_s^{*}}}
\| w_{k+1} \|_{W^{s+ \log, p}_0(\Omega)}^{p}\\[4pt]
&\leq  C(N,s,p,\Omega)  2^{p(k+1)(1 - \frac{p}{p_s^{*}})} U_{k}^{1 - \frac{p}{p_s^{*}}}
\| w_{k+1} \|_{W^{s+ \log, p}_0(\Omega)}^{p}.
\end{aligned}
\end{equation}
Thus, combining \eqref{equ41} and \eqref{Lequ25}, we deduce
\begin{equation}\label{Lequ26}
U_{k+1} \leq C^{k} U_k^{1+\vartheta}, \quad \text{for all } k \in \mathbb{N},
\end{equation}
for some constants \(C>1\) and \(\vartheta = \frac{sp}{N}\). Consequently,
\begin{equation}\label{Lequ28}
U_k \leq \frac{\eta^{k}}{\rho^{p}}, \quad \text{for all } k \in \mathbb{N},
\end{equation}
where \(\eta \in (0,1)\). The estimate follows by induction. Indeed, from \eqref{equation} we have
\[
U_0 \leq \|v_0\|_{L^p(\Omega)}^{p} = \frac{1}{\rho^{p}}.
\]
Assuming \eqref{Lequ28} holds for some \(k \in \mathbb{N}\), we deduce from \eqref{Lequ26} that
\[
U_{k+1} \leq \frac{\eta^{k+1}}{\rho^{p}}.
\]
Since \(\eta \in (0,1)\), it follows that
\begin{equation}\label{equ56}
\lim_{k \to \infty} U_k = 0.
\end{equation}
Furthermore, since \(w_k \to (v_0-1)^+\) a.e. in \(\mathbb{R}^N\), relation \eqref{equ56} implies that \(w_k \to 0\) almost everywhere in \(\Omega\). Hence, \(v_0 \leq 1\) a.e. in \(\Omega\), and consequently, $ \|u\|_{L^{\infty}(\Omega)} \leq \rho \|u\|_{L^{p}(\Omega)}. $ This concludes that \(u \in L^{\infty}(\Omega)\).
\end{proof}

We establish below a minimum principle for positive eigenfunctions. More precisely, we prove this principle for positive weak supersolutions of the problem
\begin{equation}\label{eigen2}
(-\Delta)_{p}^{s+\log} u = 0,
\end{equation}
subject to the exterior condition
\begin{equation}\label{eigen3}
u = 0 \quad \text{in } \mathbb{R}^{N} \setminus \Omega.
\end{equation}
To this end, we first extend to the fractional logarithmic $p$-Laplacian considered here a logarithmic estimate recently established by Di Castro et al. in \cite[Lemma 1.3]{DiCastro-Kuusi-Palatucci}. We believe that this extension is of independent interest and may be useful in other related settings. Let $x_0$ be an arbitrary fixed point in $\Omega$. For each $r>0$, we denote by $\mathcal{B}_{r}(x_{0})$ the open ball of radius $r$ centered at $x_0$.

\begin{lemma}\label{lemma1}
	Let $1 < p < \infty$ and $0 < s < 1$, and assume that
	$\operatorname{diam}(\Omega) < e^{-\frac{1}{sp}}$ and $B(N,s,p) \geq 0$.
Let $u \in W_{\mathrm{loc}}^{s+\log,p}(\Omega)$ be a weak super-solution of problem \eqref{eigen2} such that
	$u \geq 0$ in $\mathcal{B}_{R}(x_{0}) \Subset \Omega$, and suppose that the following two tail quantities are finite:
	\begin{equation*}
	\mathcal{J}_{+}
	:=
	\int_{\mathbb{R}^{N}\setminus\mathcal{B}_{R}(x_{0})}
	\left(
	\mathbf{K}^{s+\log}_{p}
	\left(
	\frac{|y-x_{0}|}{2}
	\right)
	\right)_{+}
	\big(u(y)\big)_{-}^{p-1}\,dy,
	\end{equation*}
	and
	\begin{equation*}
	\mathcal{J}_{-}
	:=
	\int_{\mathbb{R}^{N}\setminus\mathcal{B}_{R}(x_{0})}
	\left(
	\mathbf{K}^{s+\log}_{p}
	\left(
	2|y-x_{0}|
	\right)
	\right)_{-}
	\big(u(y)\big)_{+}^{p-1}\,dy.
	\end{equation*} 

\noindent Then, for every $0 < \delta < 1$ and $r > 0$ satisfying
	$\mathcal{B}_{2r}(x_{0}) \subset \mathcal{B}_{R/2}(x_{0})$, there exists a constant
	$C=C(N,p,s)>0$ such that
	\begin{equation}\label{equ1121}
	\begin{aligned}
	&\int_{\mathcal{B}_{2r}(x_{0})}
	\int_{\mathcal{B}_{2r}(x_{0})}
	\mathbf{K}^{s+\log}_{p}(|x-y|)
	\left|
	\ln\left(
	\frac{\delta+u(x)}{\delta+u(y)}
	\right)
	\right|^{p}
	\,dx\,dy
	\\[4pt]
	&\quad\leq
	C r^{N-sp}
	\Bigg\{
	\delta^{1-p}r^{sp}
	\Bigg[
	\int_{\mathbb{R}^{N}\setminus\mathcal{B}_{R}(x_{0})}
	\left(
	\mathbf{K}^{s+\log}_{p}
	\left(
	\frac{|y-x_{0}|}{2}
	\right)
	\right)_{+}
	\big(u(y)\big)_{-}^{p-1}\,dy
	\\
	&\qquad\qquad\qquad\qquad
	+
	\int_{\mathbb{R}^{N}\setminus\mathcal{B}_{R}(x_{0})}
	\left(
	\mathbf{K}^{s+\log}_{p}
	\left(
	2|y-x_{0}|
	\right)
	\right)_{-}
	\big(u(y)\big)_{+}^{p-1}\,dy
	\Bigg]
	+(1-\ln r)
	\Bigg\},
	\end{aligned}
	\end{equation}
	where $(u)_{\pm}:=\max\{\pm u,0\}$.
\end{lemma}

\begin{proof}
As in \cite[Lemma 1.3]{DiCastro-Kuusi-Palatucci}, let $\delta>0$ and let $\phi \in C_c^{\infty}\big(\mathcal{B}_{\frac{3r}{2}}(x_0)\big)$ satisfy
\begin{equation*}
0 \leq \phi \leq 1, \qquad
\phi \equiv 1 \quad \text{in } \mathcal{B}_{r}(x_0), \qquad
\text{and} \qquad
|D\phi| \leq \frac{c}{r}
\quad \text{in } \mathcal{B}_{\frac{3r}{2}}(x_0).
\end{equation*}
We use the test function $ \eta = (u+\delta)^{1-p}\phi^p $ in the weak formulation of problem \eqref{eigen2}. Decomposing the domain of integration accordingly, we obtain
	\begin{equation*}
	\begin{aligned}
	& \underbrace{\int_{\mathcal{B}_{2r}(x_{0})} \int_{\mathcal{B}_{2r}(x_{0})} \mathbf{K}^{s+\log}_{p}(|x - y|) \left| u(x) - u(y) \right| ^{p-2}(u(x) - u(y)) \left(\eta(x) - \eta(y)\right) dx dy}_{\mathbf{I}_{1}} \\
	&+ 2 \underbrace{\int_{\mathbb{R}^{N} \setminus \mathcal{B}_{2r}(x_{0})} \int_{\mathcal{B}_{2r}(x_{0})}\mathbf{K}^{s+\log}_{p}(|x - y|) \left| u(x) - u(y) \right| ^{p-2}(u(x) - u(y)) \left(\eta(x) - \eta(y)\right) dx dy}_{\mathbf{I}_{2}} \geq 0.
	\end{aligned}
	\end{equation*}
	On the interior domain $\mathcal{B}_{2r}(x_{0}) \times \mathcal{B}_{2r}(x_{0})$, the condition $|x-y| \leq \mathrm{diam}(\Omega) < e^{-\frac{1}{sp}}$ ensures that the kernel remains strictly positive. Proceeding as in the proof of \cite[Lemma 1.3]{DiCastro-Kuusi-Palatucci}, the term $\mathbf{I}_{1}$ satisfies
	\begin{equation*}
	\begin{aligned}
	\mathbf{I}_{1} &\leq - \frac{1}{c} \int_{\mathcal{B}_{2r}(x_{0})} \int_{\mathcal{B}_{2r}(x_{0})} \mathbf{K}^{s+\log}_{p}(|x - y|) \left| \ln \left(\frac{\delta + u(x)}{\delta + u(y)}\right) \right|^{p} dx dy \\
	&\quad + c \int_{\mathcal{B}_{2r}(x_{0})} \int_{\mathcal{B}_{2r}(x_{0})} \mathbf{K}^{s+\log}_{p}(|x - y|) \left|\phi(x) - \phi(y)\right|^p dx dy,
	\end{aligned}
	\end{equation*}
	where $c = c(p) > 0$. Using the gradient estimate $|D\phi| \leq \frac{c}{r}$, we bound the second integral. For $x, y \in \mathcal{B}_{2r}(x_0)$, we have $|x-y| \leq 4r$. Consequently, we estimate the inner integral by extending the domain of the difference variable $z = x-y$ to $\mathcal{B}_{4r}(0)$
	\begin{equation*}
	\begin{aligned}
	\int_{\mathcal{B}_{2r}(x_{0})} \int_{\mathcal{B}_{2r}(x_{0})} \mathbf{K}^{s+\log}_{p}(|x - y|) \big| \phi(x) - \phi(y)\big|^{p} dx dy &\leq c r^{-p} \int_{\mathcal{B}_{2r}(x_{0})} \left( \int_{\mathcal{B}_{4r}(0)} |z|^{p} \mathbf{K}^{s+\log}_{p}(|z|) dz \right) dx \\[4pt]
	&\leq c  r^{-p} \frac{\omega_N}{N}(2r)^N \int_{\mathcal{B}_{4r}(0)} |z|^{p} \mathbf{K}^{s+\log}_{p}(|z|) dz.
	\end{aligned}
	\end{equation*} 
	We compute the exact radial integral, noting that $p - sp = p(1-s) > 0$
	\begin{align*}
	\int_{\mathcal{B}_{4r}(0)} |z|^{p} \mathbf{K}^{s+\log}_{p}(|z|) dz &= C(N,s,p) \, \omega_N \, \int_0^{4r} \rho^{p(1-s) - 1} \big( B(N,s,p) - p\ln\rho \big) d\rho \\[4pt]
	&= C(N,s,p)  \, \omega_N \,\frac{(4r)^{p(1-s)}}{p(1-s)} \left[ B(N,s,p) - p\ln(4r) + \frac{1}{1-s} \right].
	\end{align*}
	Multiplying by the volume factor yields the exact scaling $r^{-p} r^{p(1-s)} r^N = r^{N-sp}$. Thus
	\begin{equation*}
	\begin{aligned}
	\mathbf{I}_{1} &\leq - \frac{1}{c} \int_{\mathcal{B}_{2r}(x_{0})} \int_{\mathcal{B}_{2r}(x_{0})} \mathbf{K}^{s+\log}_{p}(|x - y|) \left| \ln \left(\frac{\delta + u(x)}{\delta + u(y)}\right) \right|^{p} dx dy \\[4pt]
	&\quad + c\,  C(N,s,p) \frac{\omega_N^2}{N p(1-s)} 2^{N + 2 p(1-s)}  r^{N-sp} \left[ B(N,s,p) - p\ln(4r) + \frac{1}{1-s} \right].
	\end{aligned}
	\end{equation*}
	For the second term $\mathbf{I}_{2}$, because the test function $\eta = (u+\delta)^{1-p}\phi^p$ is supported strictly in $\mathcal{B}_{3r/2}(x_0) \Subset \mathcal{B}_{2r}(x_0)$, we have $\eta(y) = 0$ for all $y \in \mathbb{R}^N \setminus \mathcal{B}_{2r}(x_0)$. Thus, the difference $\eta(x) - \eta(y)$ reduces exactly to $\eta(x)$, and we split the exterior domain into the annulus $\mathcal{B}_{R}(x_{0}) \setminus \mathcal{B}_{2r}(x_{0})$ and $\mathbb{R}^{N} \setminus \mathcal{B}_{R}(x_{0})$
	\begin{equation*}
	\begin{aligned}
	\mathbf{I}_{2} &\leq 2 \int_{\mathcal{B}_{R}(x_{0}) \setminus \mathcal{B}_{2r}(x_{0})} \int_{\mathcal{B}_{2r}(x_{0})} \mathbf{K}^{s+\log}_{p}(|x - y|) \left| u(x) - u(y) \right|^{p-2} (u(x) - u(y)) (u(x)+\delta)^{1-p}\phi^{p}(x) dx dy \\[4pt]
	&\quad + 2 \int_{\mathbb{R}^{N} \setminus \mathcal{B}_{R}(x_{0})} \int_{\mathcal{B}_{2r}(x_{0})} \mathbf{K}^{s+\log}_{p}(|x - y|) \left| u(x) - u(y) \right|^{p-2} (u(x) - u(y)) (u(x)+\delta)^{1-p}\phi^{p}(x) dx dy.
	\end{aligned}
	\end{equation*}
Since $u \geq 0$ in $\mathcal{B}_{R}(x_{0}) \Subset \Omega$, and noting that
$\mathbf{K}_p^{s+\log}(|x-y|) > 0$ on
$\mathcal{B}_{2r}(x_{0}) \times \left(\mathcal{B}_{R}(x_{0}) \setminus \mathcal{B}_{2r}(x_{0})\right)$,
the first term can be estimated as follows:
\begin{equation*}
\begin{aligned}
&\int_{\mathcal{B}_{R}(x_{0}) \setminus \mathcal{B}_{2r}(x_{0})}
\int_{\mathcal{B}_{2r}(x_{0})}
\mathbf{K}^{s+\log}_{p}(|x-y|)
\left|u(x)-u(y)\right|^{p-2}
(u(x)-u(y))
(u(x)+\delta)^{1-p}\phi^{p}(x)\,dx\,dy \\
&\leq
\int_{\mathcal{B}_{R}(x_{0}) \setminus \mathcal{B}_{2r}(x_{0})}
\int_{\mathcal{B}_{2r}(x_{0})}
\mathbf{K}^{s+\log}_{p}(|x-y|)
\left(u(x)-u(y)\right)_{+}^{p-1}
(u(x)+\delta)^{1-p}\phi^{p}(x)\,dx\,dy \\
&\leq
\int_{\mathcal{B}_{R}(x_{0}) \setminus \mathcal{B}_{2r}(x_{0})}
\int_{\mathcal{B}_{2r}(x_{0})}
\mathbf{K}^{s+\log}_{p}(|x-y|)
\phi^{p}(x)\,dx\,dy \\
&\leq
C r^{N-sp}\left(1-\ln r\right).
\end{aligned}
\end{equation*}
For the second term, whose kernel may change sign, we introduce the set
		\[
		\omega := \left\{
		(x,y) \in
		\mathcal{B}_{2r}(x_{0})
		\times \left(\mathbb{R}^{N}\setminus \mathcal{B}_{R}(x_{0})\right) 
		:\,
		\mathbf{K}^{s+\log}_{p}(|x-y|) \geq 0
		\right\}.
		\]
		We estimate this term from above by separating the positive and negative parts of the kernel:
		\begin{equation}\label{equL1}
		\begin{aligned}
		&\int_{\mathbb{R}^{N} \setminus \mathcal{B}_{R}(x_{0})}
		\int_{\mathcal{B}_{2r}(x_{0})}
		\mathbf{K}^{s+\log}_{p}(|x-y|)
		\left|u(x)-u(y)\right|^{p-2}
		(u(x)-u(y))
		(u(x)+\delta)^{1-p}\phi^{p}(x)\,dx\,dy
		\\[4pt]
		&\quad =
		\underbrace{
			\iint_{\omega}
			\mathbf{K}^{s+\log}_{p}(|x-y|)
			\left|u(x)-u(y)\right|^{p-2}
			(u(x)-u(y))
			(u(x)+\delta)^{1-p}\phi^{p}(x)\,dx\,dy
		}_{\mathbf{E}_{1}}
		\\[4pt]
		&\qquad
		+
		\underbrace{
			\iint_{\omega^{c}}
			\mathbf{K}^{s+\log}_{p}(|x-y|)
			\left|u(x)-u(y)\right|^{p-2}
			(u(x)-u(y))
			(u(x)+\delta)^{1-p}\phi^{p}(x)\,dx\,dy
		}_{\mathbf{E}_{2}}.
		\end{aligned}
		\end{equation}
	\textbf{Estimate of $\mathbf{E}_{1}$.}
	Using the elementary inequality $ \big(u(x)-u(y)\big)_{+}^{p-1}
	\leq
	2^{p-1}
	\left(
	\big(u(x)\big)_{+}^{p-1}
	+
	\big(u(y)\big)_{-}^{p-1}
	\right), $ together with the facts that $u(x)+\delta\geq\delta$ and
	$u(x)+\delta\geq u(x)$, we obtain
	\begin{equation*}
	\begin{aligned}
	\mathbf{E}_{1}
	&\leq
	2^{p-1}
	\int_{\mathbb{R}^{N}\setminus\mathcal{B}_{R}(x_{0})}
	\int_{\mathcal{B}_{2r}(x_{0})}
	\big(\mathbf{K}^{s+\log}_{p}(|x-y|)\big)_{+}
	\phi^{p}(x)\,dx\,dy
	\\[4pt]
	&\quad
	+
	2^{p-1}\delta^{1-p}
	\int_{\mathbb{R}^{N}\setminus\mathcal{B}_{R}(x_{0})}
	\int_{\mathcal{B}_{2r}(x_{0})}
	\big(\mathbf{K}^{s+\log}_{p}(|x-y|)\big)_{+}
	\big(u(y)\big)_{-}^{p-1}
	\phi^{p}(x)\,dx\,dy.
	\end{aligned}
	\end{equation*}
	For $x\in\mathcal{B}_{2r}(x_{0})$ and
	$y\in\mathbb{R}^{N}\setminus\mathcal{B}_{2r}(x_{0})$, the triangle inequality yields $	|y-x_{0}|\leq 2|x-y|. $ Since the
 positive part of the kernel is decreasing, the mapping $ r \mapsto \big(\mathbf{K}^{s+\log}_{p}(r)\big)_{+} $ is nonincreasing on $(0,\infty)$. Indeed, on the set $ 	\left\{r>0 : B(N,s,p)-p\ln r>0\right\}, $ both factors $ C(N,s,p)\big(B(N,s,p)-p\ln r\big) $ and $
		r^{-N-sp} $ are positive and decreasing, while the positive part vanishes elsewhere. Hence, we have
	 $$ \big(\mathbf{K}^{s+\log}_{p}(|x-y|)\big)_{+}
	\leq
	\bigg(
	\mathbf{K}^{s+\log}_{p}
	\left(\frac{1}{2}|y-x_{0}|\right)
	\bigg)_{+}, $$
	  For the first term, we have
	\[
	\begin{aligned}
	&\int_{\mathbb{R}^{N} \setminus \mathcal{B}_{2r}(x_{0})}
	\left(
	\mathbf{K}^{s+\log}_{p}
	\left(\frac{1}{2}|y-x_{0}|\right)
	\right)_{+}\,dy
	\\[4pt]
	&=
	\frac{2^{N}\omega_{N}C(N,s,p)}{ps}
	\left[
	\frac{
		B(N,s,p)-p\ln r-\frac{1}{s}
	}{r^{sp}}
	+
	\frac{e^{-sB(N,s,p)}}{s}
	\right].
	\end{aligned}
	\]
	Hene, 
\begin{equation*}
\begin{aligned}
\mathbf{E}_{1}
&\leq
\frac{2^{N+p-1}\omega_N}{N}\,r^{N-sp}
\Bigg\{
\frac{2^N\omega_N C(N,s,p)}{ps}
\left[
B(N,s,p)-p\ln r-\frac{1}{s}
+\frac{r^{sp}e^{-sB(N,s,p)}}{s}
\right]
\\
&\hspace{5cm}
+\delta^{1-p}r^{sp}
\int_{\mathbb{R}^{N}\setminus\mathcal{B}_{R}(x_{0})}
\left(
\mathbf{K}^{s+\log}_{p}
\left(\frac{|y-x_{0}|}{2}\right)
\right)_{+}
\big(u(y)\big)_{-}^{p-1}\,dy
\Bigg\}.
\end{aligned}
\end{equation*}
	\textbf{Estimate of $\mathbf{E}_{2}$.}
	Using the elementary inequality $ \big(u(y)-u(x)\big)_{+}^{p-1}
	\leq
	\big(u(y)\big)_{+}^{p-1}, $ together with the fact that $u(x)+\delta\geq\delta$, we obtain
	\begin{equation*}
	\begin{aligned}
	\mathbf{E}_{2}
	&=
	\iint_{\omega^{c}}
	\left(-\mathbf{K}^{s+\log}_{p}(|x-y|)\right)
	\left|u(y)-u(x)\right|^{p-2}
	(u(y)-u(x))
	(u(x)+\delta)^{1-p}\phi^{p}(x)\,dx\,dy
	\\
	&\leq
	\delta^{1-p}
	\int_{\mathbb{R}^{N}\setminus\mathcal{B}_{R}(x_{0})}
	\int_{\mathcal{B}_{2r}(x_{0})}
	\big(\mathbf{K}^{s+\log}_{p}(|x-y|)\big)_{-}
	\big(u(y)\big)_{+}^{p-1}
	\phi^{p}(x)\,dx\,dy.
	\end{aligned}
	\end{equation*}
For $x\in\mathcal{B}_{2r}(x_{0})$ and $y\in\mathbb{R}^{N}\setminus\mathcal{B}_{R}(x_{0})$, we have $ |x- x_{0}| < 2 r <  \frac{|y-x_{0}|}{2}  $, hence $ \frac{|y-x_{0}|}{2} \leq \left| x - y\right| \leq 2 |y-x_{0}|. $
The negative part of the kernel is not monotone, so we
 argue directly: writing $  \rho = \left| y =x_{0}\right|  $,
\begin{equation*}
\begin{aligned}
\big(\mathbf{K}^{s+\log}_{p}(|x-y|)\big)_{-}
&= C(N,s,p)
\frac{\left(p\ln\left|x-y\right| - B(N,s,p)\right)_{+}}
{|x-y|^{N+sp}} \\[4pt]
&\leq
C(N,s,p)
\frac{\left(p\ln(2\rho)-B(N,s,p)\right)_{+}}
{(\rho/2)^{N+sp}} =
4^{N+sp}\big(\mathbf{K}^{s+\log}_{p}(2\rho)\big)_{-},
\end{aligned}
\end{equation*}
using $ \left| x - y\right|  \leq 2 \rho $ in the numerator and $ \left| x - y\right|  \geq \rho/2 $ in the denominator. Consequently
\begin{equation*}
\begin{aligned}
\mathbf{E}_{2}
\leq
\frac{2^{N + 2(N + sp)}\delta^{1-p}\,\omega_{N}}{N} r^{N}
\int_{\mathbb{R}^{N}\setminus\mathcal{B}_{R}(x_{0})}
\big(
\mathbf{K}^{s+\log}_{p}(2|y-x_{0}|)
\big)_{-}
\big(u(y)\big)_{+}^{p-1}\,dy.
\end{aligned}
\end{equation*}
	Thus, we obtain
	\begin{equation*}
	\begin{aligned}
	\mathbf{I}_{2}
	&\leq
	r^{N-sp}
	\Bigg\{
	\frac{2^{N+p-1}\omega_N}{N}
	\Bigg[
	\frac{2^N\omega_N C(N,s,p)}{ps}
	\left(
	B(N,s,p)-p\ln r-\frac{1}{s}
	+\frac{r^{sp}e^{-sB(N,s,p)}}{s}
	\right)
	\\
	&\hspace{3.5cm}
	+\delta^{1-p}r^{sp}
	\int_{\mathbb{R}^{N}\setminus\mathcal{B}_{R}(x_{0})}
	\left(
	\mathbf{K}^{s+\log}_{p}
	\left(\frac{|y-x_{0}|}{2}\right)
	\right)_{+}
	\big(u(y)\big)_{-}^{p-1}\,dy
	\Bigg]
	\\
	&\hspace{1.5cm}
	+\frac{2^{N+2(N+ sp)}\delta^{1-p}\omega_N}{N}\,r^{sp}
	\int_{\mathbb{R}^{N}\setminus\mathcal{B}_{R}(x_{0})}
	\left(
	\mathbf{K}^{s+\log}_{p}
	\left(2|y-x_{0}|\right)
	\right)_{-}
	\big(u(y)\big)_{+}^{p-1}\,dy +2C\left(1-\ln r\right)
	\Bigg\}.
	\end{aligned}
	\end{equation*}
Finally, combining the estimates for $\mathbf{I}_1$ and $\mathbf{I}_2$ and rearranging the resulting terms completes the proof.
\end{proof}
\begin{remark}\label{remark2}
	We note the following observations:
	\begin{enumerate}
		\item The inclusions $ \mathcal{B}_{2r}(x_{0}) \Subset \mathcal{B}_{R/2}(x_{0}) \Subset \Omega $ imply that $R>4r$ and $ r<\frac{1}{4}\operatorname{diam}(\Omega)<1. $ This fact will be used below to absorb the constants into $C(1-\ln r)$.
		
		\item The tail quantities appearing in Lemma~\ref{lemma1} are finite by H\"older's inequality whenever $ u \in L^{q}(\mathbb{R}^{N}), $ with $  q \geq p-1, $ and $R>0$. Precisely, by assuming that $ u \in W^{s+\log,p}(\mathbb{R}^{N}), $ the finiteness of the tail quantities follows automatically, and hence no additional assumption on these quantities is required.

		\item 	We point out that, in the case where $u \in W_{0}^{s+\log,p}(\Omega)$, the estimate \eqref{equ1121} can be simplified. Indeed, for the term $\mathbf{E}_{2}$, we can write
		\begin{equation*}
		\begin{aligned}
		\mathbf{E}_{2}
		&\leq
		\delta^{1-p}
		\int_{\left(\mathbb{R}^{N}\setminus\mathcal{B}_{R}(x_{0})\right)\cap\Omega}
		\int_{\mathcal{B}_{2r}(x_{0})}
		\big(\mathbf{K}^{s+\log}_{p}(|x-y|)\big)_{-}
		\big(u(y)\big)_{+}^{p-1}
		\phi^{p}(x)\,dx\,dy
		=0.
		\end{aligned}
		\end{equation*}
		Indeed, since $ 	|x-y|\leq \operatorname{diam}(\Omega)<e^{-\frac{1}{sp}} $ for $(x,y)\in\mathcal{B}_{2r}(x_{0})\times\left(\left(\mathbb{R}^{N}\setminus\mathcal{B}_{R}(x_{0})\right)\cap\Omega\right)$, the kernel remains strictly positive. Hence, $ \big(\mathbf{K}^{s+\log}_{p}(|x-y|)\big)_{-}=0. $ Therefore, estimate \eqref{equ1121} reduces to
		\begin{equation*}
		\begin{aligned}
		&\int_{\mathcal{B}_{2r}(x_{0})}
		\int_{\mathcal{B}_{2r}(x_{0})}
		\mathbf{K}^{s+\log}_{p}(|x-y|)
		\left|
		\ln\left(
		\frac{\delta+u(x)}{\delta+u(y)}
		\right)
		\right|^{p}
		\,dx\,dy
		\\[4pt]
		&\quad\leq
		C r^{N-sp}
		\Bigg\{
		\delta^{1-p}r^{sp}
		\int_{\mathbb{R}^{N}\setminus\mathcal{B}_{R}(x_{0})}
		\left(
		\mathbf{K}^{s+\log}_{p}
		\left(
		\frac{|y-x_{0}|}{2}
		\right)
		\right)_{+}
		\big(u(y)\big)_{-}^{p-1}\,dy
		+(1-\ln r)
		\Bigg\}.
		\end{aligned}
		\end{equation*}
	\end{enumerate}
\end{remark}
By Lemma \ref{lemma1}, and arguing as in the proof of \cite[Theorem A.1]{Brasco-Franzina}, we obtain the following minimum principle.
\begin{theorem}\label{theorem3}
	Let $1 < p < \infty$ and $0 < s < 1$, and let $\Omega \subset \mathbb{R}^{N}$ be a bounded, open, and connected set. Let 
	$u \in W^{s+ \log,p}_0(\Omega)$ be a weak supersolution of the problem \eqref{eigen2}-\eqref{eigen3}. Assume that $u \geq 0$ in $\Omega$ and $u \not\equiv 0$ in $\Omega$. Then $ u > 0  $ a.e. in $ \Omega. $
\end{theorem}

 It is well-known that if $u \in W^{s+\log,p}_{0}(\Omega)$ is a local extremum of the functional $\mathcal{J}_{s+\log,p}(u,u)$ under the constraint $\|u\|_{L^{p}(\Omega)} = 1$, then $u$ is an eigenfunction corresponding to the eigenvalue $\lambda = \mathcal{J}_{s+\log,p}(u,u)$, which is the smallest eigenvalue that can be obtained in this variational framework. For this reason, we refer to this value as the first eigenvalue, since this variational characterization naturally leads to the definition:
\begin{equation}\label{equ21}
\lambda^{s+\log}_{1,p} := \inf \left\{ \mathcal{J}_{s+\log,p}(u,u) \;:\; u \in W^{s+\log,p}_{0}(\Omega), \ \|u\|_{L^{p}(\Omega)} = 1 \right\}.
\end{equation}

\begin{theorem}\label{theorem6}
Let $1 < p < \infty$ and $0 < s < 1$. The first eigenvalue $\lambda^{s+\log}_{1,p}$ is attained, and its associated extremal function $\phi^{s+\log}_{1,p}$ is an eigenfunction of  \eqref{eigen} corresponding to $\lambda^{s+\log}_{1,p}$. Moreover, if $\mathrm{diam}(\Omega) < e^{-\frac{1}{sp}}$ and $B(N,s,p) \geq 0$, then the following assertions hold:
\begin{enumerate}
	\item The first eigenvalue $\lambda^{s+\log}_{1,p}$ is strictly positive, and the corresponding eigenfunction $\phi^{s+\log}_{1,p}$ can be chosen to be nonnegative in $\Omega$.
	
	\item The eigenvalue $\lambda^{s+\log}_{1,p}$ is simple, in the sense that if $v \in W^{s+\log,p}_{0}(\Omega)$ is another eigenfunction associated with $\lambda^{s+\log}_{1,p}$, then there exists a constant $c > 0$ such that
    $ \phi^{s+\log}_{1,p} = c v. $

	\item Let $u \in W^{s+\log,p}_{0}(\Omega)$ be an eigenfunction associated with an eigenvalue $\lambda$ such that $u > 0$ in $\Omega$. Then, it follows that $\lambda = \lambda^{s+\log}_{1,p}$.
\end{enumerate}
\end{theorem}

\begin{proof}
A minimizer can be obtained via the direct method of the calculus of variations. To this end, we define
	\[
	\mathbf{I} = \inf \left\{ \Phi_{s+ \log, p}(u) \;:\; u \in \mathcal{C} \right\},
	\quad \text{where} \quad
	\Phi_{s+ \log, p}(u) := \mathcal{J}_{s+\log,p}(u,u),
	\]
	and
	\[
	\mathcal{C} = \left\{ u \in W^{s+\log,p}_{0}(\Omega) \;:\; \|u\|_{L^{p}(\Omega)} = 1 \right\}.
	\]
	By the definition of the infimum, there exists a sequence $(u_n) \subset \mathcal{C}$ such that $ \Phi_{s+ \log, p}(u_n) \leq \mathbf{I} + \frac{1}{n} \leq \mathbf{I} + 1. $ From Proposition \ref{proposition1} (1) and (3), the sequence $(u_n)$ is bounded in $W_0^{s+\log,p}(\Omega)$. Hence, by Theorem \ref{EmbeddingsTheorem}, there exist a function $u \in W_0^{s+\log,p}(\Omega)$ and a subsequence (still denoted by $(u_n)$) such that
	\begin{equation}\label{equ22}
	u_n \rightharpoonup u_{0} \quad \text{in } W_0^{s+\log,p}(\Omega), \qquad
	u_n \to u_{0}  \text{in } L^{q}(\Omega)\ \text{for all } 1 \leq q \leq p_s^{*}, \qquad
	u_n(x) \to u_{0}(x) \ \text{a.e. in } \Omega.
	\end{equation}
	Moreover, since $ \|u_n\|_{L^{p}(\Omega)} \to \|u_{0}\|_{L^{p}(\Omega)}, $ as $ n \to \infty $ and $\|u_n\|_{L^{p}(\Omega)} = 1$ for all $n$ it follows that $\|u_{0}\|_{L^{p}(\Omega)} = 1$, and hence $u \in \mathcal{C}$. On the other hand, again by Proposition \ref{proposition1} (1), we obtain
	\begin{equation*}
	0 \leq \mathcal{J}_{-}(u_{n} - u_{0}, u_{n} - u_{0})
	\leq \frac{2^{p} \omega_N}{(s p)^2}  \| u_{n} - u_{0}\|_{L^{p}(\Omega)}^p \to 0 \quad \text{as } n \to \infty.
	\end{equation*}
	Thus,
	\begin{equation}\label{equ23}
	\lim_{n\to \infty} \mathcal{J}_{-}(u_{n}, u_{n}) = \mathcal{J}_{-}(u_{0}, u_{0}).
	\end{equation}
Furthermore, in view of Proposition \ref{proposition1} (3) and the boundedness of $(u_n)$ in $W_0^{s+\log,p}(\Omega)$, we deduce that
\begin{equation*}
\begin{aligned}
\mathcal{J}_{s}(u_{n} - u_{0}, u_{n} - u_{0})
& \leq - \frac{1}{C(N,s,p) \ln r} \|u_{n} - u_{0}\|_{W_{0}^{s+\log,p}(\Omega)} 
+ \frac{2^{p} \omega_N}{sp} r^{-sp} \|u_{n} - u_{0} \|^{p}_{L^{p}(\Omega)}\\[4pt]
& \leq - \frac{C}{C(N,s,p) \ln r}
+ \frac{2^{p} \omega_N}{sp} r^{-sp} \|u_{n} - u_{0}\|^{p}_{L^{p}(\Omega)}, 
\end{aligned}
\end{equation*}
for any $0 < r < 1$, where $C>0$. Using \eqref{equ22}, we pass to the limit as $n \to \infty$, and then as $r \to 0^{+}$, to obtain
	\begin{equation}\label{equ24}
	\lim_{n\to \infty} \mathcal{J}_{s}(u_{n}, u_{n}) = \mathcal{J}_{s}(u_{0}, u_{0}).
	\end{equation}
Combining the properties of the weak topology with \eqref{equ23}-\eqref{equ24}, we deduce that
\begin{equation*}
\begin{aligned}
\Phi_{s+\log, p}(u_{0})
&= \mathcal{J}_{s+\log,p}(u_{0}, u_{0}) \\[4pt]
&\leq \frac{p}{2}\left( \liminf_{n \to \infty}\mathcal{J}_{+}(u_{n}, u_{n}) - \lim_{n \to \infty}\mathcal{J}_{-}(u_{n}, u_{n})\right)
+ \frac{B(N, s, p) C(N, s, p)}{2} \lim_{n \to \infty}\mathcal{J}_{s}(u_{n}, u_{n}) \\[4pt]
&\leq \liminf_{n \to \infty}\left(
\dfrac{p}{2}\left(\mathcal{J}_{+}(u_{n}, u_{n}) - \mathcal{J}_{-}(u_{n}, u_{n})\right)
+ \frac{B(N, s, p) C(N, s, p)}{2} \mathcal{J}_{s}(u_{n}, u_{n})
\right) \\[4pt]
&= \liminf_{n \to \infty} \Phi_{s+\log, p}(u_{n}) \\[4pt]
&\leq \liminf_{n \to \infty}\left(\mathbf{I} + \frac{1}{n}\right)
\leq \mathbf{I}.
\end{aligned}
\end{equation*}
Therefore, it follows that $\mathbf{I} = \Phi_{s+\log, p}(u_{0})$, and consequently $u_{0}$ attains the minimum of $\Phi_{s+\log, p}$ over $\mathcal{C}$.  Hence, there exists a Lagrange multiplier $\lambda$ such that
\[
\mathcal{J}_{s+\log,p}\big(u_{0}, \varphi\big)
= \lambda \int_{\Omega} \left| u_{0} \right|^{p-2} u_{0} \varphi 
\quad \text{for all } \varphi \in W^{s+\log,p}_{0}(\Omega).
\]
By choosing $\varphi = u_{0}$ and invoking \eqref{equ21}, we deduce that $\lambda = \lambda^{s+\log}_{1,p}$. Therefore, the function $u_{0}$  is an eigenfunction of \eqref{eigen} associated with the first eigenvalue $\lambda^{s+\log}_{1,p}$. Hence, we denote it hereafter by $\phi^{s+\log}_{1,p}$. Assume that $\mathrm{diam}(\Omega) < e^{-\frac{1}{sp}}$ and that $B(N,s,p) \geq 0$. Then, it follows that

\medskip

\noindent\text{(1)} From Proposition~\ref{proposition1} (4), it follows immediately that
\[
\lambda^{s+\log}_{1,p} = \mathcal{J}_{s+\log,p}\big(\phi^{s+\log}_{1,p}, \phi^{s+\log}_{1,p}\big) > 0.
\]
Moreover, $\phi^{s+\log}_{1,p}$ does not change sign. Indeed, again by Proposition~\ref{proposition1} (4), we have that $|\phi^{s+\log}_{1,p}| \in \mathcal{C}$, and
\[
\lambda^{s+\log}_{1,p}
= \mathcal{J}_{s+\log,p}\big(\phi^{s+\log}_{1,p}, \phi^{s+\log}_{1,p}\big)
\geq \mathcal{J}_{s+\log,p}\big(|\phi^{s+\log}_{1,p}|, |\phi^{s+\log}_{1,p}|\big).
\]
By the definition of the first eigenvalue, we deduce that
\[
\mathcal{J}_{s+\log,p}\big(\phi^{s+\log}_{1,p}, \phi^{s+\log}_{1,p}\big)
= \mathcal{J}_{s+\log,p}\big(|\phi^{s+\log}_{1,p}|, |\phi^{s+\log}_{1,p}|\big).
\]
Consequently, by invoking Proposition~\ref{proposition1} (4) once again, we deduce that $\phi^{s+\log}_{1,p}$ does not change sign. Hence, any eigenfunction corresponding to $ \lambda^{s+\log}_{1,p} $ is either nonnegative or nonpositive in $ \Omega $.\\
\text{(2)} First, in view of (1), the eigenfunction $\phi^{s+\log}_{1,p}$ can be chosen to be nonnegative. Consequently, by Theorem~\ref{theorem3}, we infer that $\phi^{s+\log}_{1,p} > 0$ a.e. in $\Omega$. Moreover, Theorem~\ref{theorem2} guarantees that $\phi^{s+\log}_{1,p} \in L^{\infty}(\Omega)$. Now, let $v \in W^{s+\log,p}_{0}(\Omega)$ be another eigenfunction associated with $\lambda^{s+\log}_{1,p}$. We may therefore assume that $v > 0$ a.e. in $\Omega$ and $v \in L^{\infty}(\Omega)$. For any fixed $\varepsilon > 0$, we have
\[
\frac{(\phi^{s+\log}_{1,p}+\varepsilon)^{p}
	- (v+\varepsilon)^{p}}{(\phi^{s+\log}_{1,p}+\varepsilon)^{p-1}}
\quad \text{and} \quad
\frac{(v+\varepsilon)^{p}
	- (\phi^{s+\log}_{1,p}+\varepsilon)^{p}}{(v+\varepsilon)^{p-1}}
\in W^{s+\log,p}_{0}(\Omega).
\]
Indeed, it suffices to distinguish the following two cases.

\medskip
\noindent\textbf{Case 1:} $v(x) > v(y)$ and $\phi^{s+\log}_{1,p}(x) > \phi^{s+\log}_{1,p}(y)$.  
By the Lagrange Mean Value Theorem, we obtain
\begin{align*}
&\left|
\frac{(v+\varepsilon)^p(x)}{(\phi^{s+\log}_{1,p}+\varepsilon)^{p-1}(x)}
-
\frac{(v+\varepsilon)^p(y)}{(\phi^{s+\log}_{1,p}+\varepsilon)^{p-1}(y)}
\right|
\leq
\left|
\frac{(v+\varepsilon)^p(x)}{(\phi^{s+\log}_{1,p}+\varepsilon)^{p-1}(x)}
-
\frac{(v+\varepsilon)^p(y)}{(\phi^{s+\log}_{1,p}+\varepsilon)^{p-1}(x)}
\right| \\[4pt]
&\quad+
\left|
\frac{(v+\varepsilon)^p(y)}{(\phi^{s+\log}_{1,p}+\varepsilon)^{p-1}(x)}
-
\frac{(v+\varepsilon)^p(y)}{(\phi^{s+\log}_{1,p}+\varepsilon)^{p-1}(y)}
\right| \\[4pt]
&\leq
\frac{p \max\{(v+\varepsilon)^{p-1}(x),(v+\varepsilon)^{p-1}(y)\}}
{(\phi^{s+\log}_{1,p}+\varepsilon)^{p-1}(x)}
 |v(x)-v(y)| \\[4pt]
&\quad+
\frac{(p-1) (v+\varepsilon)^p(y) 
	\max\{(\phi^{s+\log}_{1,p}+\varepsilon)^{p-2}(x),
	(\phi^{s+\log}_{1,p}+\varepsilon)^{p-2}(y)\}}
{(\phi^{s+\log}_{1,p}+\varepsilon)^{p-1}(x)
	(\phi^{s+\log}_{1,p}+\varepsilon)^{p-1}(y)}
 |\phi^{s+\log}_{1,p}(x)-\phi^{s+\log}_{1,p}(y)| \\[4pt]
&\leq
p\left(\frac{v(x)+\varepsilon}{\phi^{s+\log}_{1,p}(x)+\varepsilon}\right)^{p-1}
|v(x)-v(y)| \\[4pt]
&\quad+
(p-1)\left(\frac{v(y)+\varepsilon}{\phi^{s+\log}_{1,p}(y)+\varepsilon}\right)^{p}
|\phi^{s+\log}_{1,p}(x)-\phi^{s+\log}_{1,p}(y)| \\[4pt]
&\leq
C \left(p,\left\|\frac{v+\varepsilon}{\phi^{s+\log}_{1,p}+\varepsilon}\right\|_{L^\infty(\Omega)}\right)
\left(|v(x)-v(y)|
+ |\phi^{s+\log}_{1,p}(x)-\phi^{s+\log}_{1,p}(y)|\right).
\end{align*}

\medskip
\noindent\textbf{Case 2:} $v(x) < v(y)$ and $v(x) > v(y)$ (mixed monotonicity between $v$ and $\phi^{s+\log}_{1,p}$).  
A similar argument yields
\begin{align*}
&\left|
\frac{(v+\varepsilon)^p(x)}{(\phi^{s+\log}_{1,p}+\varepsilon)^{p-1}(x)}
-
\frac{(v+\varepsilon)^p(y)}{(\phi^{s+\log}_{1,p}+\varepsilon)^{p-1}(y)}
\right|
\leq
\frac{p \max\{(v+\varepsilon)^{p-1}(x),(v+\varepsilon)^{p-1}(y)\}}
{(\phi^{s+\log}_{1,p}+\varepsilon)^{p-1}(x)}
 |v(x)-v(y)| \\[4pt]
&\quad+
\frac{(p-1) (v+\varepsilon)^p(y) 
	\max\{(\phi^{s+\log}_{1,p}+\varepsilon)^{p-2}(x),
	(\phi^{s+\log}_{1,p}+\varepsilon)^{p-2}(y)\}}
{(\phi^{s+\log}_{1,p}+\varepsilon)^{p-1}(x)
	(\phi^{s+\log}_{1,p}+\varepsilon)^{p-1}(y)}
 |\phi^{s+\log}_{1,p}(x)-\phi^{s+\log}_{1,p}(y)| \\[4pt]
&\leq
C \left(p,\left\|\frac{v+\varepsilon}{\phi^{s+\log}_{1,p}+\varepsilon}\right\|_{L^\infty(\Omega)}\right)
\left(|v(x)-v(y)|
+|\phi^{s+\log}_{1,p}(x)-\phi^{s+\log}_{1,p}(y)|\right),
\end{align*}
for some constant $C>0$.  We then choose the above test functions in~\eqref{equ25}, satisfied respectively by $\phi^{s+\log}_{1,p}$ and $v$. Summing the resulting identities and applying Lemma~\ref{Lemma2} (noting that $\mathrm{diam}(\Omega) < e^{-\frac{1}{sp}}$ and $B(N,s,p) \geq 0$ imply $\mathrm{diam}(\Omega) < e^{\frac{B(N,s,p)}{p}}$), we obtain
\begin{equation*}
\begin{aligned}
0 \leq\;& \mathcal{J}_{s+\log,p}\left(\phi^{s+\log}_{1,p}, 
\frac{(\phi^{s+\log}_{1,p}+\varepsilon)^{p} - (v+\varepsilon)^{p}}{(\phi^{s+\log}_{1,p}+\varepsilon)^{p-1}}\right)  - \mathcal{J}_{s+\log,p}\left(v, 
\frac{(v+\varepsilon)^{p} - (\phi^{s+\log}_{1,p}+\varepsilon)^{p}}{(v+\varepsilon)^{p-1}}\right) \\[4pt]
=\;& \lambda^{s+\log}_{1,p} \int_{\Omega} 
\left(\frac{(\phi^{s+\log}_{1,p})^{p-1}}{(\phi^{s+\log}_{1,p}+\varepsilon)^{p-1}} 
- \frac{v^{p-1}}{(v+\varepsilon)^{p-1}}\right)
\left((\phi^{s+\log}_{1,p}+\varepsilon)^{p} - (v+\varepsilon)^{p}\right) dx.
\end{aligned}
\end{equation*}
Passing to the limit as $\varepsilon \to 0$, and using Fatou's lemma together with the dominated convergence theorem, we deduce
\begin{equation*}
\mathcal{J}_{s+\log,p}\left(\phi^{s+\log}_{1,p}, 
\frac{(\phi^{s+\log}_{1,p})^{p} - v^{p}}{(\phi^{s+\log}_{1,p})^{p-1}}\right)
- \mathcal{J}_{s+\log,p}\left(v, 
\frac{v^{p} - (\phi^{s+\log}_{1,p})^{p}}{v^{p-1}}\right)
= 0.
\end{equation*}
Finally, by Lemma~\ref{Lemma2}, we conclude that there exists a constant $c > 0$ such that  $ \phi^{s+\log}_{1,p} = c v. $

\medskip

\noindent\text{(3)} Let $u \in W_0^{s+\log, p}(\Omega)$ be an eigenfunction associated with the eigenvalue $\lambda$ such that $u > 0$ a.e. in $\Omega$, and let $  \phi_{1,p}^{s+\log} $ be the positive, $L^p$-normalized first eigenfunction. By the variational characterization \eqref{equ21}, we immediately have
	\begin{equation} \label{LambdaLowerBound}
	\lambda^{s+\log}_{1,p} \leq \frac{\mathcal{J}_{s+\log, p}(u,u)}{\|u\|_{L^p(\Omega)}^p} = \frac{\lambda \|u\|_{L^p(\Omega)}^p}{\|u\|_{L^p(\Omega)}^p} = \lambda.
	\end{equation}
	To establish the reverse inequality, we use the fractional Picone inequality (see \cite[Proposition 4.2]{Brasco-Franzina}). Indeed, for any $\varepsilon > 0$, we consider the nonnegative test function $ \varphi_\varepsilon = \left(\phi_{1,p}^{s+\log}\right) ^p (u+\varepsilon)^{1 - p}\in W_0^{s+\log, p}(\Omega). $ Accordingly, by employing $\varphi_\varepsilon$ as a test function in the weak formulation of the eigenvalue problem satisfied by $u$, we obtain
	\begin{equation} \label{PiconeTest}
	\lambda \int_\Omega \left( \frac{u}{u+\varepsilon} \right)^{p-1} \left(  \phi_{1,p}^{s+\log}\right) ^p dx = \mathcal{J}_{s+\log, p}(u, \varphi_\varepsilon).
	\end{equation}
	Let $P_\varepsilon(x,y)$ denote the fractional Picone difference
	\begin{equation*}
	P_\varepsilon(x,y) := |u(x)-u(y)|^{p-2}(u(x)-u(y)) \left( \frac{\left(\phi_{1,p}^{s+\log}\right) ^p(x)}{(u(x)+\varepsilon)^{p-1}} - \frac{\left(\phi_{1,p}^{s+\log}\right) ^p(y)}{(u(y)+\varepsilon)^{p-1}} \right).
	\end{equation*}
By virtue of the fractional Picone inequality, we have
\[
P_\varepsilon(x,y) \leq \left| \phi_{1,p}^{s+\log}(x) - \phi_{1,p}^{s+\log}(y) \right|^{p}, \quad \forall x,y \in \mathbb{R}^N.
\]
We next decompose the energy functional into the contribution associated with the interior domain and the corresponding cross-interaction terms as follows:
\begin{equation} \label{EnergySplit}
\mathcal{J}_{s+\log, p}(u, \varphi_\varepsilon)
= \frac{1}{2} \iint_{\Omega \times \Omega} P_\varepsilon(x,y) \mathbf{K}_p^{s+\log}(|x-y|) dx dy
+ \int_{\Omega} \int_{\mathbb{R}^{N} \setminus \Omega} P_\varepsilon(x,y) \mathbf{K}_p^{s+\log}(|x-y|) dy dx.
\end{equation}
	On the interior domain $\Omega \times \Omega$, the condition $|x-y| \leq \mathrm{diam}(\Omega) < e^{-\frac{1}{sp}}$ ensures that the kernel is strictly positive, $\mathbf{K}_p^{s+\log}(|x-y|) > 0$. Thus, the inequality is preserved
	\begin{equation} \label{InteriorBound}
	\frac{1}{2} \iint_{\Omega \times \Omega} P_\varepsilon(x,y) \mathbf{K}_p^{s+\log}(|x-y|) dx dy \leq \frac{1}{2} \iint_{\Omega \times \Omega} \left| \phi_{1,p}^{s+\log}(x)-\phi_{1,p}^{s+\log}(y)\right| ^p \mathbf{K}_p^{s+\log}(|x-y|) dx dy.
	\end{equation}
On the cross-domain $\Omega \times \big(\mathbb{R}^{N} \setminus \Omega \big)$, we note that $ u(y) = \phi_{1,p}^{s+\log}(y) = 0. $ So, the Picone difference reduces to
\[
P_\varepsilon(x,y)
= u(x)^{p-1} \frac{\big(\phi_{1,p}^{s+\log}(x)\big)^{p}}{\big(u(x)+\varepsilon\big)^{p-1}}.
\]
Moreover, under the assumptions $\mathrm{diam}(\Omega) < e^{-\frac{1}{sp}}$ and $B(N,s,p) \geq 0$, we infer that
\begin{equation*}
\begin{aligned} 
\int_{\mathbb{R}^{N} \setminus \Omega} \mathbf{K}_p^{s+\log}(|x-y|) dy
&= \int_{\mathbb{R}^{N} \setminus \Omega} 
\frac{B(N, s, p) C(N, s, p) - p C(N, s, p) \ln |x-y|}{|x-y|^{N+sp}}  dy\\[4pt]
&= C(N,s,p)\Bigg[
B(N,s,p)\int_{\mathbb{R}^{N} \setminus \Omega} 
\frac{1}{|x-y|^{N+sp}} dy 
- p \int_{\mathbb{R}^{N} \setminus \Omega} 
\frac{\ln |x-y|}{|x-y|^{N+sp}} dy
\Bigg] > 0.
\end{aligned}
\end{equation*}
Therefore, we conclude that
\begin{equation}\label{CrossBound}
\begin{aligned} 
\int_{\Omega} &\int_{\mathbb{R}^{N} \setminus \Omega} P_\varepsilon(x,y) \mathbf{K}_p^{s+\log}(|x-y|) dy dx \\[4pt]
&= \int_\Omega u(x)^{p-1} \frac{\big(\phi_{1,p}^{s+\log}(x)\big)^{p}}{(u(x)+\varepsilon)^{p-1}} \left(\int_{\mathbb{R}^{N} \setminus \Omega} \mathbf{K}_p^{s+\log}(|x-y|) dy\right)  dx\\[4pt]
&\leq \int_\Omega \big(\phi_{1,p}^{s+\log}(x)\big)^p \left( \int_{\mathbb{R}^{N} \setminus \Omega} \mathbf{K}_p^{s+\log}(|x-y|) dy\right)  dx  \\[4pt]
&= \int_{\Omega} \int_{\mathbb{R}^{N} \setminus \Omega}  \left| \phi_{1,p}^{s+\log}(x)-\phi_{1,p}^{s+\log}(y)\right| ^p \mathbf{K}_p^{s+\log}(|x-y|) dy dx.
\end{aligned}
\end{equation}
	Substituting \eqref{InteriorBound} and \eqref{CrossBound} into \eqref{EnergySplit} yields the global bound
	\begin{equation} \label{PiconeEnergyBound}
	\mathcal{J}_{s+\log, p}(u, \varphi_\varepsilon) \leq \mathcal{J}_{s+\log, p}\left( \phi_{1,p}^{s+\log}, \phi_{1,p}^{s+\log}\right)  = \lambda^{s+\log}_{1,p}.
	\end{equation}
	Combining \eqref{PiconeTest} and \eqref{PiconeEnergyBound}, we obtain
	\begin{equation*}
	\lambda \int_\Omega \left( \frac{u}{u+\varepsilon} \right)^{p-1} \left( \phi_{1,p}^{s+\log}\right) ^p dx \leq \lambda^{s+\log}_{1,p}.
	\end{equation*}
	Since $u > 0$ a.e. in $\Omega$, the sequence of functions $\left( \frac{u}{u+\varepsilon} \right)^{p-1}\left( \phi_{1,p}^{s+\log}\right) ^p$ converges pointwise a.e. to $\left( \phi_{1,p}^{s+\log}\right) ^p$ as $\varepsilon \to 0^+$. By the Monotone Convergence Theorem, and passing to the limit while taking into account that $\mathrm{supp}(u)=\Omega$, we obtain
	\begin{equation*}
	\lambda \int_\Omega \left( \phi_{1,p}^{s+\log} \right)^p  dx \leq \lambda^{s+\log}_{1,p}.
	\end{equation*}
	Since $\|\phi_{1,p}^{s+\log}\|_{L^p(\Omega)} = 1$, it follows that $\lambda \leq \lambda^{s+\log}_{1,p}$. Together with \eqref{LambdaLowerBound}, this yields  $ \lambda = \lambda^{s+\log}_{1,p} $.
\end{proof}

\begin{remark}
	Recall that $\phi^{s}_{1,p}$ denotes the positive, normalized eigenfunction of the fractional $p$-Laplacian $(-\Delta)^s_p$ in $W^{s,p}_0(\Omega)$ associated with the first eigenvalue $\lambda^{s}_{1,p}$, which admits the variational characterization
	\[
	\lambda^{s}_{1,p}
	= \inf_{\substack{u \in W^{s,p}_0(\Omega)\setminus\{0\} \\ \|u\|_{L^{p}(\Omega)} = 1}}
	\frac{C(N,s,p)}{2}
	\iint_{\mathbb{R}^{2N}}
	\frac{|u(x)-u(y)|^{p}}{|x-y|^{N+sp}} dx dy.
	\]
	It is well-known that $\lambda^{s}_{1,p} > 0$ and that the above infimum is attained.  By contrast, as shown in Theorem~\ref{theorem6}, the first eigenvalue $\lambda^{s+\log}_{1,p}$ associated with the fractional logarithmic $p$-Laplacian $(-\Delta)_{p}^{s+\log}$ is, in general, not necessarily positive (see Proposition \ref{proposition1}). However, under appropriate geometric conditions on the domain $\Omega$, one can guarantee the strict positivity of $\lambda^{s+\log}_{1,p}$  arxived it a few weeks after ours .
\end{remark}

\section*{Acknowledgements}
	
	We also acknowledge that N.~Biswas, S.~Das, and A.~Sen \cite{Biswas} used the definition of the fractional logarithmic $p$-Laplacian, which we were the first to introduce in the literature without stating that.  Additionally, they made use of several properties that we established.

	\section*{Statements and Declarations}

\subsection*{Ethics approval and consent to participate}
Not applicable.
\subsection*{Funding}
Not applicable
\subsection*{Availability of data and materials}
Not applicable
\subsection*{Conflict of interest}
The authors declare that there is no conflict of interest.

\addtocontents{toc}{\protect\setcounter{tocdepth}{2}}

\end{document}